\documentclass[11pt,leqno]{amsart}
\usepackage{amsthm,amsfonts,amssymb,amsmath,bm,hyperref}%
\usepackage{xcolor,multicol}
\usepackage{graphics}
\long\def\beginpgfgraphicnamed#1#2\endpgfgraphicnamed{\includegraphics{#1}}
\usepackage{mathtools}
\mathtoolsset{showonlyrefs}
\numberwithin{equation}{section}
\renewcommand\d{\partial}
\renewcommand\a{\alpha}
\renewcommand\b{\beta}
\renewcommand\o{\omega}
\newcommand\s{\sigma}
\renewcommand\t{\tau}
\newcommand\R{\mathbb R}\newcommand\N{\mathbb N}\newcommand\Z{\mathbb Z}
\newcommand\C{\mathbb C}

\newcommand\D{{\partial}}

\def\g{\gamma}
\def\G{\Gamma}
\def\t{\tau}

\def\k{\kappa}

\def\dsp{\displaystyle}
\def\l{\lambda}

\def\epsilon{\varepsilon}
\def\e{\varepsilon}

\def\S{\Sigma}

\def\SS{\mathbb S}
\def\p{\partial}
\def\tr{{\rm tr}\,}
\def\and{\quad\mbox{and}\quad}

\newcommand\br{\begin{rem}}
\newcommand\er{\end{rem}}
\newcommand\bp{\begin{pmatrix}}
\newcommand\ep{\end{pmatrix}}
\newcommand\be{\begin{equation}}
\newcommand\ee{\end{equation}}
\newcommand\ba{\begin{equation}\begin{aligned}}
\newcommand\ea{\end{aligned}\end{equation}}

\newcommand\cA{\mathcal A}

\newcommand{\Id}{{\rm Id }}

\newtheorem{defi}{Definition}[section]
\newtheorem{theo}[defi]{Theorem}
\newtheorem{prop}[defi]{Proposition}
\newtheorem{lem}[defi]{Lemma}
\newtheorem{cor}[defi]{Corollary}
\newtheorem{rem}[defi]{Remark}
\newtheorem{ass}[defi]{Assumption}
\newtheorem{hyp}[defi]{Hypothesis}

\def\op{{\rm op} }
\usepackage{ulem}

\def\vs{\vskip5pt}

\numberwithin{equation}{section}

\def\vs{\vskip5pt}

\def\Id{\operatorname{Id}}

\def\mat22#1#2#3#4{\begin{pmatrix}#1&#2\\ #3&#4\end{pmatrix}}

\begin{document}
\title[The onset of instability in first-order systems]{The onset of instability in first-order systems}
\author{Nicolas Lerner}
\address{Institut de Math\'ematiques de Jussieu-Paris Rive Gauche UMR CNRS 7586, Universit\'e Pierre-et-Marie-Curie}
\email{nicolas.lerner@imj-prg.fr}
\author{Toan Nguyen} 
\address{Department of Mathematics,
Pennsylvania State University.
University Park
State College, Pennsylvania}
\email{nguyen@math.psu.edu}
\author{Benjamin Texier}
\address{Institut de Math\'ematiques de Jussieu-Paris Rive Gauche UMR CNRS 7586, Universit\'e Paris-Diderot}
\email{benjamin.texier@imj-prg.fr}
\thanks{T.N. was supported by the Fondation Sciences Math\'ematiques de Paris through a postdoctoral grant. B.T. thanks Yong Lu and Baptiste Morisse for their remarks on an earlier version of the manuscript. The authors thank the referees for detailed and useful remarks.}
\date{\today}
\begin{abstract} 
We study the Cauchy problem for first-order quasi-linear systems of partial differential equations. 
When the spectrum of the initial principal symbol
is not included in the real line,
i.e., when hyperbolicity is violated at initial time, then the Cauchy problem is strongly unstable, in the sense of Hadamard. This phenomenon, which extends the linear Lax-Mizohata theorem, was explained
by G. M\'etivier in [{\it Remarks on the well-posedness of the nonlinear Cauchy problem}, Contemp.~Math.~2005].
In the present paper, we are interested in the transition from hyperbolicity to non-hyperbolicity, that is  the limiting case where hyperbolicity holds at initial time, but is violated at positive times: under such an hypothesis,
we generalize a recent work by N. Lerner, Y. Morimoto and C.-J. Xu [{\it Instability of the Cauchy-Kovalevskaya solution for a class of non-linear systems}, American J.~Math.~2010] on  complex scalar systems,
as we prove that even a weak defect of hyperbolicity implies a strong Hadamard instability. Our examples include Burgers systems, Van der Waals gas dynamics, and Klein-Gordon-Zakharov systems. Our analysis relies on an approximation result for pseudo-differential flows, introduced by B.~Texier in [{\it Approximations of pseudo-differential flows}, Indiana Univ. Math. J. 2016].
\end{abstract}

\maketitle

\setcounter{tocdepth}{1}
\tableofcontents
\renewcommand{\refname}{References}

\section{Introduction}
We study well-posedness issues in Sobolev spaces for the Cauchy problem for first-order, quasi-linear systems of partial differential equations:
\begin{equation}\label{1st-system}
 \D_t u + \sum_{1 \leq j \leq d} A_j(t,x,u)\D_{x_j}u = F(t,x,u),%
 \end{equation}
where $t \geq 0,$ $x \in \R^d,$ $u(t,x) \in \R^N,$ %
the maps $A_j$ are smooth  from $\R_+ \times \R_x^d \times \R_u^N$ 
to the space of $N\times N$ real matrices and $F$ is smooth from $\R_+ \times \R_x^d \times \R^N_u$ into $\R^N.$

 We prove in this article a general ill-posedness result in Sobolev spaces for the Cauchy problem for \eqref{1st-system}, under an assumption of a weak defect of hyperbolicity that describes the transition from hyperbolicity to ellipticity. This extends recent results of G.~M\'etivier \cite{M2} and N.~Lerner, Y.~Morimoto and C.-J.~Xu \cite{LMX}. Here ``well-posedness" is understood in the sense of Hadamard \cite{Ha}, meaning existence and regularity of a flow; ``hyperbolicity", as discussed in Section \ref{sec:hist}, means reality of the spectrum of the principal symbol, and ``ellipticity" corresponds to existence of non-real eigenvalues for the principal symbol.
 
   We begin this introduction with a discussion of hyperbolicity and well-posedness (Section \ref{sec:hist}), then give three results: Theorem \ref{th:0} describes ill-posedness of elliptic initial-value problems, while Theorems \ref{th:1/2} and \ref{th:1} are ill-posedness results for systems undergoing a transition from hyperbolicity to ellipticity. These results are illustrated in a series of examples in Section \ref{sec:ex}. Our main assumption (Assumption \ref{ass:main}) and main result (Theorem \ref{th:main}) are stated in Sections \ref{sec:ass} and \ref{sec:th}.

\subsection{Hyperbolicity as a necessary condition for well-posedness} \label{sec:hist}
Lax-Mizo\-ha\-ta theorems, named after Peter Lax and Sigeru Mizohata,
  state that well-posed non-characteristic initial-value problems for first-order systems  are necessarily hyperbolic, meaning that all eigenvalues of the principal symbol are real. 
 
 P.~Lax's original result \cite{L2} is stated in a $C^\infty$ framework, for linear equations, i.e.~such that $A_j(t,x,u) \equiv A_j(t,x).$ Lax uses a relatively strong definition of well-posedness that includes continuous dependence not only in the data, but also in a source. This allows him in  particular to consider WKB approximate solutions; the proof of \cite{L2} shows that in the non-hyperbolic case, if the eigenvalues are separated, the $C^0$ norms of high-frequency WKB solutions grow faster than the $C^k$ norms of the datum and source, for any $k.$ The separation assumption ensures that the eigenvalues are smooth, implying smoothness for the coefficients of the WKB cascade of equations. 
  In the same $C^\infty$ framework for linear equations but without assuming spectral separation, S.~Mizohata \cite{Mi} proved that existence, uniqueness and continuous dependence on the data cannot hold in the non-hyperbolic case. 

 Later S.~Wakabayashi \cite{W} and K.~Yagdjian \cite{Y1,Y2} extended the analysis to the quasi-linear case, but it was only in 2005 that a precise description of the lack of regularity of the flow was given, by M\'etivier: Theorem 3.2 in \cite{M2} states that in the case that the $A_j$ are analytic, under the assumption that for some fixed vector $u^0 \in \R^N$ and some frequency $\xi^0 \in \R^d$ the principal symbol $\sum_{1 \leq j \leq d} A_j(u^0) \xi^0_j$ is not hyperbolic, some analytical data uniquely generate analytical solutions, but the corresponding flow for \eqref{1st-system} is not H\"older continuous from high Sobolev norms to $L^2,$ locally around a Cauchy-Kovalevskaya solution issued from the constant datum $u^0.$

 M\'etivier's result is a {\it long-time Cauchy-Kovalevskaya} result. Without loss of generality, assume indeed that $u^0 = 0.$ Then Theorem 3.2 in \cite{M2} states that data that are small in high norms may generate solutions that are instantaneously large in low norms. In this view, assume in \eqref{1st-system} the hyperbolic ansatz: $u(t,x) = \e v(t/\e,x/\e),$ where $\e > 0.$ Setting $F \equiv 0$ for simplicity, and $\t = t/\e,$ $y = x/\e,$  the equation in $v$ is 
 \begin{equation} \label{2} \d_\t v + \sum_{1 \leq j \leq d} A_j(\e \t, \e y, \e v) \d_{y_j} v = 0.
 \end{equation}

 If  all fluxes $A_j$ are analytic in their arguments, the Cauchy-Kovalevskaya theorem ensures the existence and uniqueness of a solution $v$ issued from an analytic datum $v(0,x),$ over a time interval $O(1)$ in the fast variable  $\t.$ What's more, by regularity of the coefficients of \eqref{2} with respect to $\e,$ the solution $v$ stays close, in analytical semi-norms, to the solution $w$ of the constant-coefficient system
 \begin{equation}\label{3}
  \d_\t w + \sum_{1 \leq j \leq d} A_j(0) \d_{y_j} w = 0
  \end{equation}
 over time intervals $O(1).$ By Assumption on $A_j(0),$ the Fourier transform $\hat w(\t,\xi^0)$ of $w$ in the spectral direction $\xi^0$ grows like $e^{\t C(\xi^0)},$ for some $C > 0.$ This implies a similar growth for  $v(t/\epsilon,\xi^0),$
  and in turn a growth in $\e e^{t C(\xi^0)/\e}$ for $\hat u(t,\xi^0),$ {\it but only on time intervals $O(\e),$} due to the initial rescaling in time. The content of M\'etivier's result is therefore to show that the solution $v$ to \eqref{2} exists, and the growth persists, over ``long" time intervals $O(|\log \e|),$ so that the exponential amplification is effective. 

\medskip

 In the scalar complex case, the results of N.~Lerner, Y.~Morimoto and C.-J.~Xu  \cite{LMX} extended the analysis of M\'etivier to the situation 
 where
 the symbol is initially hyperbolic, but hyperbolicity is instantaneously lost, in the sense that a characteristic root is real at $t = 0,$ but leaves the real line at  positive times. The main result of \cite{LMX} states that such a weak defect of hyperbolicity implies a strong form of ill-posedness; the analysis is based on representations of solutions by the method of characteristics, following \cite{M1}. This argument does not carry over to systems, even in the case of a diagonal principal symbol, 
 if the components of the solution are coupled through the lower-order term $F(u).$

\medskip

 Our goal in this article is to extend the instability results of {\cite{LMX}} on complex scalar equations  to the case of   quasi-linear first-order systems \eqref{1st-system}.
 In the process, we recover a version of the results of {\rm \cite{M2},} with a method of proof that does not rely on analyticity.

\subsection{On the local character of our assumptions and results} \label{sec:local} Our assumptions are local in nature. They bear on the germ, at a given point $(t_0,x_0,\xi_0) \in \R_+ \times \R^d \times \R^d,$ representing time, position, and frequency, of the principal symbol evaluated at a given reference solution. Under these local assumptions, we prove local instabilities, which extend the aforementioned Lax-Mizohata theorems, and which roughly say that there are no local solutions possessing a minimal smoothness with initial data taking values locally in an elliptic region. These local instabilities are independent of the global properties of the system \eqref{1st-system}. In particular, the system \eqref{1st-system} may have formal conserved quantities; see for instance the compressible Euler equations \eqref{vdw-ex} introduced in Section \ref{sec:ex}.

\subsection{Transition from hyperbolicity to ellipticity} \label{sec:weak}

 Our starting point is to assume that there exists a {local} {smooth} solution 
$\phi$ to \eqref{1st-system} with {a large} Sobolev regularity: %
 \begin{equation} \label{phi} %
 \phi \in C^\infty([0,T_0],  \,H^{s_1}(U)),
 \end{equation}
  for some $T_0 > 0,$ some open set $U \subset \R^d,$ and some Sobolev regularity index $s_1 = 1 + d/2 + s_2,$ where $s_2 > 0$ is large enough, depending on the parameters in our problem\footnote{We use regularity of $\phi$ in particular in the construction of the local solution operator; see Appendix \ref{sec:Duhamel}, specifically the proof of Lemma {\rm \ref{lem:duh-remainder}}, in which $q_0$ is the order of a Taylor expansion involving $\phi.$}. If the matrices $A_j$ and the source $F$ depend analytically on $(t,x,u),$ then we can choose $\phi$ to be a Cauchy-Kovalevskaya solution. However, we do not use analyticity in the rest of the paper. 
The linearized principal symbol at $\phi$ is
  \begin{equation}\label{symbolA} A(t,x, \xi)  := \sum_{1 \leq j \leq d} \xi_jA_j(t,x,\phi(t,x)).
  \end{equation}  
The upcoming ill-posedness results are based on readily verifiable conditions bearing on the jet at $t = 0$ of the  characteristic polynomial $P$ of the principal symbol:
\be \label{def:P} P(t,x,\xi,\l) := \det\big(\l \Id - A(t,x,\xi)\big).\ee
Most of these conditions are stable under perturbations of the principal symbol,
 and all can be expressed in terms of the fluxes $A_j$ and the initial datum $\phi(0).$ In particular, it is of key importance that the verification of these conditions does not require any knowledge of the behavior of the reference solution $\phi$ at positive times. 
 
Also, it should be mentioned that our hypotheses
 do not require the computation of eigenvalues and are expressed explicitly in terms of derivatives of $P$ given by \eqref{def:P}  at initial time.

\subsubsection{Hadamard instability} \label{sec:Hadamard}

If \eqref{1st-system} does possess a flow, how regular can we reasonably expect it to be? A good reference point is the regularity of the flow generated by a symmetric system. If for all $j$ and all $u,$ the matrices $A_j(u)$ are symmetric, then local-in-time solutions to the initial-value problem for \eqref{1st-system} exist and are unique in $H^s,$ for $s > 1 + d/2$ \cite{F,K0,L0}; moreover, given a ball $B_{H^s}(0,R) \subset H^s,$ there is an associated existence time $T > 0.$ The flow is Lipschitz $B_{H^s}(0,R) \cap H^{s+1} \to L^\infty([0,T],H^s),$ continuous $B_{H^s}(0,R) \to L^\infty([0,T],H^s),$ but not uniformly continuous $B_{H^s}(0,R) \to L^\infty([0,T],H^s)$ in general \cite{K0}. Micro-locally symmetrizable systems also enjoy these properties \cite{M3}.
 
 \medskip
 
 Accordingly, ill-posedness will be understood as follows:
\begin{defi} \label{def:insta} We will say that the initial-value problem for the system \eqref{1st-system} is ill-posed in the vicinity of the reference solution $\phi$ satisfying \eqref{phi}, if for some $x_0 \in U,$  given any parameters $m,\a,\delta > 0,$ $T$ such that 
\begin{equation} \label{param} \begin{aligned} &
 m \in\R, \quad \frac{1}{2} < \a \leq 1,\quad B(x_0,\delta) \subset U, \quad 0<T\le T_{0},\end{aligned}\end{equation}
where $U$ and $T_0$ are as in \eqref{phi}, there is no neighborhood $\mathcal U$ of $\phi(0)$ in $H^m(U),$ 
 such that,  for all $u(0) \in {\mathcal U},$ the system 
  \eqref{1st-system} has a solution $u \in L^\infty([0,T], W^{1,\infty}(B(x_0,\delta)))$ issued from $u(0)$ which satisfies  
 \begin{equation} \label{stab}
 \sup_{\substack {u_{0}\in \mathcal U\\0\le t\le T}}
 \frac{\| u(t) - \phi(t) \|_{W^{1,\infty}(B(x_0,\delta))}}
 {\| u_0 - \phi(0) \|^\a_{H^m(U)}}<+\infty.\end{equation}
\end{defi} 
Thus \eqref{1st-system} is ill-posed near the reference solution $\phi$ if either data arbitrarily close to $\phi(0)$ fail to generate trajectories, corresponding to absence of a solution, or if trajectories issued close to $\phi(0)$ deviate from $\phi,$ corresponding to absence of H\"older continuity for the solution operator. In the latter case, we note that:
 \begin{itemize}
 \item the deviation is relative to the initial closeness, so that $\phi$ is unstable in the sense of Hadamard, not in the sense of Lyapunov;
 \item the deviation is instantaneous: $T$ is arbitrarily small; it is localized: $\delta$ is arbitrarily small.
 \item The initial closeness is measured in a strong $H^m$ norm, where $m$ is arbitrarily large\footnote{That is, the only restriction on $m$ is the Sobolev regularity of $\phi:$ we need, in particular, $m \leq s_1$  for \eqref{stab} to make sense.}, while the deviation is measured in a weaker $W^{1,\infty}$ norm, defined as $|f|_{W^{1,\infty}} = |f|_{L^\infty} + |\nabla_x f|_{L^\infty}.$  
\end{itemize}

In our proofs of ill-posedness in the sense of Definition \ref{def:insta}, we will always {\it assume} existence of a solution issued from a small perturbation of $\phi(0),$ and proceed to disprove \eqref{stab}. 

Note that the flows of ill-posed problems in the sense of Definition \ref{def:insta} exhibit a lack of {\it H\"older} continuity. F. John introduced in \cite{J} a notion of ``well-behaved" problem, weaker than well-posedness. In well-behaved problems, Cauchy data generate unique solutions, and, in restriction to balls in the $W^{M,\infty}$ topology, for some integer $M,$ the flow is H\"older continuous in appropriate norms. The notions introduced in \cite{J} were developed in the article \cite{B} by H. Bahouri, who used sharp Carleman estimates.

The restriction to $\a > 1/2$ in Definition \ref{def:insta} is technical. Precisely, it comes from the fact that we prove ill-posedness by disproving \eqref{stab}, as indicated above. This gives weak bounds on the solution, which we use to bound the nonlinear terms. Consider nonlinear terms in \eqref{1st-system} which are controlled by $\ell_0$-homogeneous terms in $u,$ with $\ell_0 \geq 2,$ that is such that $\d_u A_j = O(u^{\ell_0-2})$ and $\d_u F = O(u^{\ell_0 -1}).$ These bounds hold if, for instance, $A_j(u) \d_{x_j} u = u^{\ell_0-1} \d_x u$ and $F(u) = u^{\ell_0},$ using scalar notation. Then, the proof of our general result (Theorem \ref{th:main}) shows ill-posedness with $\a > 1/\ell_0.$ (See indeed Lemma \ref{lem:source:new} and its proof, and note the constraint $2K' > K$ which appears at the end of the proof in Section \ref{sec:end0}.)

Finally, we point out that Definition \ref{def:insta} describes only the behavior of solutions which belong to $W^{1,\infty}.$ This excludes in particular shocks, which are expected to form in finite time for systems \eqref{1st-system}, even in the case of smooth data. Shocks with jump across elliptic zones could exhibit some stability properties.
\subsubsection{Initial ellipticity} \label{sec:init-ell}

Our first result states that the ellipticity condition
\be \label{cond:ell} 
 P(0,\o_0) = 0, \qquad \o_0 = (x_0,\xi_0,\l_0) \in U \times (\R^d \setminus \{ 0 \}) \times (\C \setminus \R),
 \ee
where $P$ is the characteristic polynomial defined in \eqref{def:P}, implies ill-posedness: 

\begin{theo} \label{th:0} Under the ellipticity condition \eqref{cond:ell}, the Cauchy problem for system \eqref{1st-system} is ill-posed in the vicinity of the reference solution $\phi,$ in the sense of Definition {\rm \ref{def:insta}}. 
\end{theo}

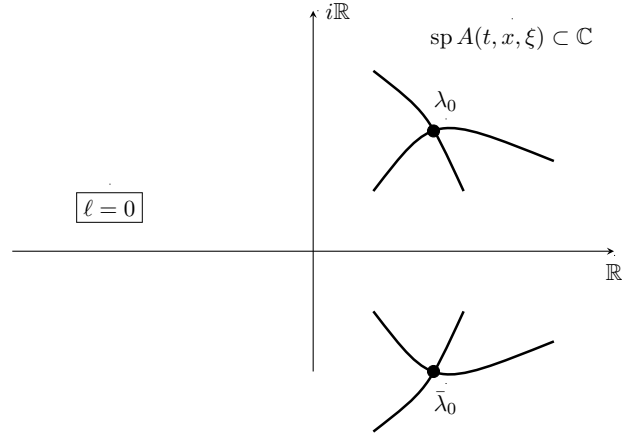
\begin{figure}
\scalebox{.8}{
\beginpgfgraphicnamed{figure-sp0}
 \begin{tikzpicture}
 
 \begin{scope}[>=stealth]
 \draw[line width=.5pt][->] (-5,0) -- (5,0);
 
 \draw[line width=.5pt][->] (0,-2) -- (0,4);

  \draw (5cm,-2pt) -- (5cm,-2pt) node[anchor=north] {${\mathbb R}$};
 
  \draw (3.3cm,110pt) -- (3.3cm,110pt) node[anchor=north]{$\mbox{sp}\,A(t,x,\xi) \subset {\mathbb C}$};
  \draw (2pt,4cm) -- (2pt,4cm) node[anchor=west] {$i {\mathbb R}$} ; 
  
  \draw (-3,1.5,1pt) -- (-3,1.5, 1pt) node[anchor=north] {$\boxed{\ell = 0}$} ; 
  
  \filldraw (2,2) circle (.1cm) ; 
  
  \draw (2.2,2.2) -- (2.2,2.2) node[anchor=south] {$\lambda_0$} ; 
  
   \filldraw (2,-2) circle (.1cm) ; 
  
  \draw (2.2,-2.2) -- (2.2,-2.2) node[anchor=north] {$\bar \lambda_0$} ;

 \draw[very thick] (1,1) .. controls (2,2.3) .. (4,1.5);

 \draw[very thick] (1,3) .. controls (1.9,2.3) .. (2.5,1);

 \draw[very thick] (1,-1) .. controls (2,-2.3) .. (4,-1.5);

 \draw[very thick] (1,-3) .. controls (1.9,-2.3) .. (2.5,-1);

  \end{scope}
  
 \end{tikzpicture}
 \endpgfgraphicnamed
 }
\caption{In Theorem \ref{th:0}, corresponding to $\ell = 0$ in Assumption \ref{ass:main}, the principal symbol at $(0,x_0,\xi_0)$ has non-real eigenvalues $\l_0, \bar \l_0.$ These may correspond to coalescing points in the spectrum, for $(t,x,\xi)$ near $(0,x_0,\xi_0).$}
\end{figure}

Theorem \ref{th:0} (proved in Section \ref{sec:proof0}) states that hyperbolicity is a necessary condition for the well-posedness of the initial-value problem \eqref{1st-system}, and partially recovers M\'etivier's result\footnote{Theorem 3.2 in \cite{M2} shows not only instability, but also existence and uniqueness, under assumption of analyticity for the fluxes, the source and the initial data.}.
An analogue to Theorem \ref{th:0} in the high-frequency regime is given in \cite{LT}, based on \cite{T3} just like our proof of Theorem \ref{th:0}; the main result of \cite{LT} precisely describes how resonances may induce local defects of hyperbolicity in strongly perturbed semi-linear hyperbolic systems, and thus destabilize WKB solutions.

\subsubsection{Non semi-simple defect of hyperbolicity} \label{sec:ell=1/2}

We now turn to situations in which the initial principal symbol is hyperbolic:%
\be \label{init:hyp}
 \mbox{$P(0,x,\xi,\l) = 0$ \,\, implies \,\, $\l \in \R,$ \,\, for all $(x,\xi) \in U \times (\R^d \setminus \{ 0 \}),$}
\ee
 and aim to describe situations in which some roots of $P$ are non-real for $t > 0.$ Let 
$$ 
 \Gamma := \big\{ \o = (x,\xi,\l) \in  U \times (\R^d \setminus \{0\}) \times \R, \quad P(0,\o) = 0 \big\},
$$ 
 By reality of the coefficients of $P,$ non-real roots occur in conjugate pairs. In particular, eigenvalues must coalesce at $t = 0$ if we are to observe non-real eigenvalues for $t > 0.$ 
 
 Let then $\o_0 \in \G,$ such that
  \begin{equation} \label{cond:coal}
   \d_\l P(0,\o_0) = 0, \quad \d_\l^2 P(0,\o_0) \neq 0.
   \end{equation}
The eigenvalue $\l_0$ of $A(0,x_0,\xi_0)$ thus has multiplicity exactly two. %
 Assume in addition that $\o_0$ satisfies condition
 \begin{equation} \label{cond:coal2} 
  (\d^2_\l P \d_t P)(0,\o_0)  > 0.
  \end{equation}
  
The eigenvalues are continuous\footnote{By continuity of $A$ and Rouch\'e's theorem; see \cite{K} or \cite{Tp}.}, implying that condition  \eqref{cond:coal2} is open, meaning that if it holds at $\o_0,$ then it holds at any nearby $\o$ in $\G.$ 
\begin{theo} \label{th:1/2} Assume that conditions \eqref{cond:coal}-\eqref{cond:coal2} hold for some $\o_0 \in \G,$ and that the other eigenvalues of $A(0,x_0,\xi_0)$ are simple. Then the Cauchy problem for system \eqref{1st-system} is ill-posed in the vicinity of the reference solution $\phi,$ in the sense of Definition {\rm \ref{def:insta}}.
\end{theo}

The conditions \eqref{cond:coal}-\eqref{cond:coal2} are relevant, and, as far as we know, new, also in the linear case.

Van der Waals systems and Klein-Gordon-Zakharov systems illustrate Theorem \ref{th:1/2}; see Sections \ref{sec:ex}, \ref{sec:VdW} and \ref{sec:KG}. %

The proof of Theorem \ref{th:1/2}, given in Section \ref{sec:proof1/2}, reveals that  under \eqref{cond:coal}-\eqref{cond:coal2}, the eigenvalues that coalesce at $t = 0$ branch out of the real axis. The branching time is typically not identically equal to $t = 0$ around $(x_0,\xi_0);$ for $(x,\xi)$ close to $(x_0,\xi_0),$ it is equal to $t_\star(x,\xi) \geq 0,$ with a smooth transition function $t_\star.$ At $(t_\star(x,\xi),x,\xi)$ the branching eigenvalues are not time-differentiable, in particular not semi-simple. Details are given in Section \ref{sec:branch}, in the proof of Theorem \ref{th:1/2}. Figure 3 pictures the typical shape of the transition function. The elliptic domain is $\{ t > t_\star \},$ and the hyperbolic domain is $\{ t < t_\star\}.$ %

Under the assumptions of Theorem \ref{th:1/2} and assuming analyticity of the coefficients, B.~Morisse \cite{Mo} proves existence and uniqueness in addition to instability in Gevrey spaces, further extending G. M\'etivi\-er's analysis \cite{M2}.

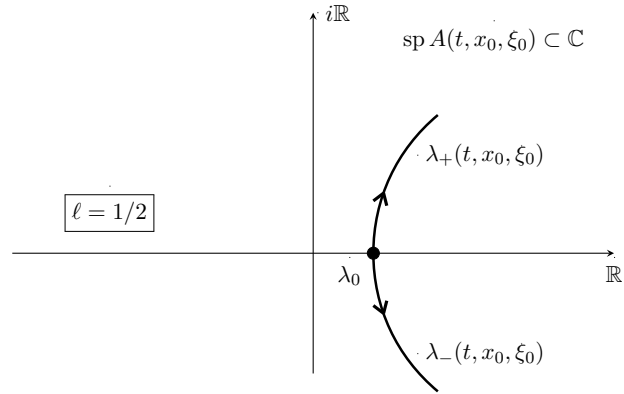
\begin{figure}
\scalebox{0.8}{
\beginpgfgraphicnamed{figure-sp12}
 \begin{tikzpicture}

 \begin{scope}[>=stealth]
 \draw[line width=.5pt][->] (-5,0) -- (5,0);
 
 \draw[line width=.5pt][->] (0,-2) -- (0,4);
 
  \draw (5cm,-2pt) -- (5cm,-2pt) node[anchor=north] {${\mathbb R}$};
 
  \draw (3cm,110pt) -- (3cm,110pt) node[anchor=north]{$\mbox{sp}\,A(t,x_0,\xi_0) \subset {\mathbb C}$};
  \draw (2pt,4cm) -- (2pt,4cm) node[anchor=west] {$i {\mathbb R}$} ; 
  
  \draw (-3,1.5,1pt) -- (-3,1.5, 1pt) node[anchor=north] {$\boxed{\ell = 1/2}$} ; 
  
  \filldraw (1,0) circle (.1cm) ; 
  
  \draw (.6cm,-2pt) -- (.6cm,-2pt) node[anchor=north] {$\lambda_0$} ; 
  
  \draw[very thick] (1cm,0) arc (0:-50:-30mm) ; 
  
  \draw (2.5,2.4,2pt) -- (2.5,2.4,2pt) node[anchor=west] {$\lambda_+(t,x_0,\xi_0)$} ; 
  
  \draw[very thick] (1cm,0) arc (0:50:-30mm) ;
  
  \draw[line width=1.5pt] (1.17,1) -- (1.223,.79) ; 
  
  \draw[line width=1.5pt] (1.17,1) -- (1.0,.87) ; 
  
  \draw (2.5,-.9,2pt) -- (2.5,-.9,2pt) node[anchor=west] {$\lambda_-(t,x_0,\xi_0)$} ; 
  
  \draw[line width=1.5pt] (1.17,-1) -- (1.223,-.79) ; 
  
  \draw[line width=1.5pt] (1.17,-1) -- (1.0,-.87) ; 
  
  \end{scope}
  
   \end{tikzpicture}
\endpgfgraphicnamed
} 
\caption{In Theorem \ref{th:1/2}, corresponding to $\ell = 1/2$ in Assumption \ref{ass:main}, a bifurcation occurs at $(0,x_0,\xi_0)$ in the spectrum of the principal symbol. The eigenvalues are not time differentiable. The arrows indicate the direction of time.}
\end{figure}

\begin{figure} \label{fig:domain:1/2}
\scalebox{.8}{ 
\beginpgfgraphicnamed{figure-tstar12}
\begin{tikzpicture}
 
 \put(60,30){${\rm hyperbolic}$}
  
  \put(15,60){${\rm elliptic}$}
 \draw[very thick] (0,0) parabola (3,4) ;
 
 \draw[very thick] (0,0) parabola (-3,4) ; 
 
 \put(90,100){$t_\star$} ;

 \begin{scope}[>=stealth]
 \draw[line width=.5pt][->] (-5,0) -- (5,0);
 
 \draw[line width=.5pt][->] (0,0) -- (0,5);
 
  \node[anchor=north] {$(x_0,\xi_0)$} ;
 
 \filldraw (0,0) circle (0.1cm) ;  
  \draw (5cm,1pt) -- (5cm,1pt) node[anchor=north] {$(x,\xi)$};
 
  \draw (1pt,5cm) -- (1pt,5cm) node[anchor=west] {$t$} ; 
  
  \draw (-3,1.5,1pt) -- (-3,1.5, 1pt) node[anchor=north] {$\boxed{\ell = 1/2}$} ; 
  \end{scope}

 \end{tikzpicture}
\endpgfgraphicnamed
}
\caption{In Theorem \ref{th:1/2}, the transition occurs at $t = t_\star(x,\xi) \geq 0,$ near $(x_0,\xi_0).$} %
\end{figure}
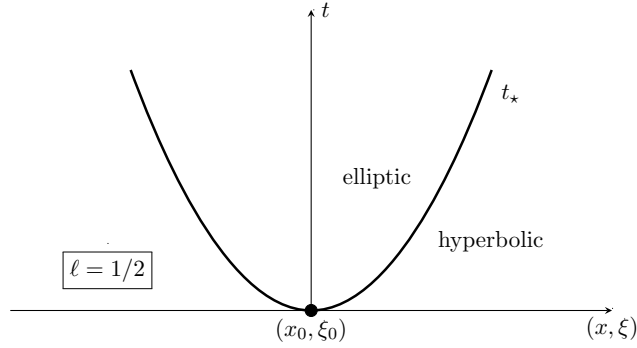

\subsubsection{Semi-simple defect of hyperbolicity} \label{sec:ell=1}

Time-differentiable defects of hyperbolicity of size two can be simply characterized in terms of $P:$ 
 \begin{prop} \label{prop:evalues} 
 Let $P(t,x,\xi,\lambda)$ be the characteristic polynomial \eqref{def:P} of the principal symbol $A(t,x,\xi)$ \eqref{symbolA}. 
We assume initial hyperbolicity \eqref{init:hyp}. Let $\o = (x,\xi,\l) \in \G.$ If $\d_\l P(0,\o) = 0 \neq \d_\l^2 P(0,\o),$ then for the branches $\l$ of eigenvalues of $A$ which coalesce at $(0,x,\xi),$ there holds
  \begin{equation} \label{non-R}
  \left. \begin{aligned} \mbox{$\l(\cdot,x,\xi)$ is differentiable at $t = 0$} \\ \Im m \, \d_t \l(0,x,\xi) \neq 0 \end{aligned}\right\} \iff \left\{\begin{aligned} \d_t P(0,\o) = 0, \\ (\d^2_{t\l} P(0,\o))^2 < (\d^2_t P \d_\l^2 P)(0,\o).\end{aligned}\right.
   \end{equation}
 \end{prop}

\begin{proof}
 We assume $\l \in C^2.$ The proof in the general case is postponed to Appendix \ref{app:eigenvalue}. For $t$ in a neighborhood of $0,$ there holds $P(t,x,\xi,\l(t,x,\xi)) \equiv 0.$ Differentiating with respect to $t,$ we find
$$ 
  \d_t P(0,\o) + \d_t \l(0,x,\xi) \d_\l P(0,\o) = 0.
$$ 
Since $\lambda(0,x,\xi)$ is real-valued, by reality of $P,$ the derivatives $\d_t P$ and $\d_\l P$ are real. If we assume $\Im m \, \d_t \l(0,x,\xi) \neq 0,$ then there holds $\d_t P(\o) = \d_\l P(\o) = 0.$ Differentiating again with respect to $t,$ we find
 \begin{equation} \label{p3}
 \d^2_t P(0,\o) + 2 \d_t \l(0,x,\xi) \d_{t\l}^2 P(0,\o) + (\d_t \l(0,x,\xi))^2 \d^2_\l P(0,\o)  =0 .
 \end{equation}
 Equation \eqref{p3}, a second-order polynomial equation in $\d_t \l(0,x,\xi),$ has non-real roots if and only if the second condition in the right-hand side of \eqref{non-R} holds. 
 \end{proof}

We now examine the situation in which a double and semi-simple eigenvalue $\l_0$ belongs to a branch $\l$ of double and semi-simple eigenvalues at $t = 0,$ which all satisfy conditions \eqref{non-R}, that is:

\begin{hyp} \label{hyp} For some  $\o_0 = (x_0,\xi_0,\l_0) \in \G$ satisfying \eqref{cond:coal} and \eqref{non-R}, and such that $\l_0$ is a semi-simple eigenvalue of $A(0,x_0,\xi_0),$ for all $\o = (x,\xi,\l)$ in a neighborhood of $\o_0$ in $\G,$ there holds
 $$ \d_\l P(0,\o) = \d_t P(0,\o)  = 0, $$
 and $\l$ is a semi-simple eigenvalue of $A(0,x,\xi).$ 
\end{hyp}

Semi-simplicity of an eigenvalue means simpleness as a root of the minimal polynomial.

Condition $(\d^2_{t\l} P(\o))^2 < (\d^2_t P \d_\l^2 P)(\o)$ is open; in particular if it holds at $\o_0 \in \G,$ it holds at all  nearby $\o \in \G.$ Thus under Hypothesis \ref{hyp}, conditions \eqref{cond:coal} and \eqref{non-R} hold in a neighborhood of $\o_0$ in $\G.$  

 \begin{theo} \label{th:1} Assume that  Hypothesis {\rm \ref{hyp}} holds, and that the other eigenvalues of $A(0,x_0,\xi_0)$ are simple. Then the Cauchy problem for system \eqref{1st-system} is ill-posed in the vicinity of the reference solution $\phi,$ in the sense of Definition {\rm \ref{def:insta}}. 
 \end{theo}

\begin{figure}
\scalebox{.8}{
\beginpgfgraphicnamed{figure-sp1}
 \begin{tikzpicture}
 
 \begin{scope}[>=stealth]
 \draw[line width=.5pt][->] (-5,0) -- (5,0);
 
 \draw[line width=.5pt][->] (0,-2) -- (0,4);
 
 \draw (5cm,-2pt) -- (5cm,-2pt) node[anchor=north] {${\mathbb R}$};
 
  \draw (3cm,110pt) -- (3cm,110pt) node[anchor=north]{$\mbox{sp}\,A(t,x,\xi) \subset {\mathbb C}$};
  \draw (2pt,4cm) -- (2pt,4cm) node[anchor=west] {$i {\mathbb R}$} ; 
  
  \draw (-3,1.5,1pt) -- (-3,1.5, 1pt) node[anchor=north] {$\boxed{\ell = 1}$} ; 
  
  \filldraw (1,0) circle (.1cm) ; 
  
  \draw (.6cm,-2pt) -- (.6cm,-2pt) node[anchor=north] {$\lambda_0$} ; 
  
  \draw[very thick,->] (1cm,0) -- (2cm, -1cm) ; 
  
  \draw[very thick] (2cm,-1cm) -- (3cm,-2cm) ;

  \draw (3.7,2.4,2pt) -- (3.7,2.4,2pt) node[anchor=west] {$\lambda_+(t,x,\xi)$} ; 
  
  \draw[very thick,->] (1cm,0) -- (2cm,1cm) ;
  
  \draw[very thick] (2cm,1cm) -- (3cm, 2cm) ; 
  
  \draw (3.7,-.9,2pt) -- (3.7,-.9,2pt) node[anchor=west] {$\lambda_-(t,x,\xi)$} ; 
  
  \end{scope}
  
 \end{tikzpicture}
\endpgfgraphicnamed
}
\caption{In Theorem \ref{th:1}, corresponding to $\ell = 1$ in Assumption \ref{ass:main}, a bifurcation occurs at $(0,x,\xi)$ in the spectrum of the principal symbol, for all $(x,\xi)$ near $(x_0,\xi_0).$ The eigenvalues are time-differentiable. The arrows indicate the direction of time.}
\end{figure}

\begin{figure} \label{fig:domain:1}
\scalebox{.8}{
\beginpgfgraphicnamed{figure-tstar1}
 \begin{tikzpicture}
 
 \begin{scope}[>=stealth]
 \draw[line width=.5pt][->] (-5,0) -- (5,0);
 
 \draw[line width=.5pt][->] (0,-2) -- (0,3.5);
 
 \draw[line width = 2pt]  (-4,0) -- (4,0); 
 
 \draw (2cm,1pt) -- (2cm,1pt) node[anchor=south] {$t_\star \equiv 0$} ; 
 
 \draw(-.7,1pt) -- (-.7,1pt) node[anchor=north] {$(x_0,\xi_0)$} ;
 
 \filldraw (0,0) circle (0.1cm) ;  
  \draw (5cm,1pt) -- (5cm,1pt) node[anchor=north] {$(x,\xi)$};
 
  \draw (1pt,3.5cm) -- (1pt,3.5cm) node[anchor=west] {$t$} ; 
  
  \draw (-3,1.5,1pt) -- (-3,1.5, 1pt) node[anchor=north] {$\boxed{\ell = 1}$} ; 
  \end{scope}
  
  \put(17,-33){${\rm hyperbolic}$}
  
  \put(23,33){${\rm elliptic}$}
 \end{tikzpicture}
\endpgfgraphicnamed
}
\caption{In Theorem \ref{th:1}, the transition occurs at $t = 0,$ uniformly near $(x_0,\xi_0).$}
\end{figure}
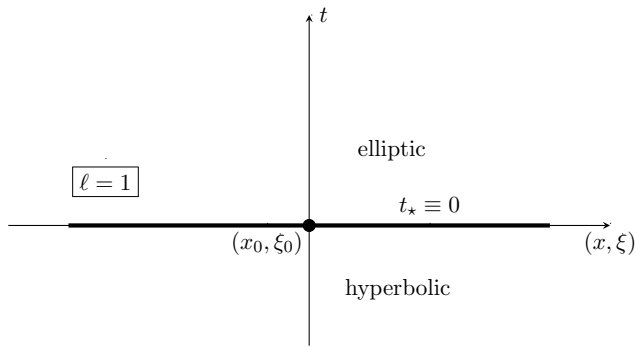

An analogue to Theorem \ref{th:1} in the high-frequency regime is the result of Y.~Lu \cite{Lu}, in which it is shown how {\it higher-order resonances}, not present in the data, may destabilize precise WKB solutions.  %
 Theorem \ref{th:1} is illustrated by the Burgers systems of Sections \ref{sec:ex} and \ref{sec:Burgers}. 

 \subsection{Remarks} \label{sec:extension} 

Taken together,  our results assert that, for principal symbols with eigenvalues of multiplicity at most two,
 if one of 
 \begin{itemize}
 \item[(a)] condition \eqref{cond:ell}, 
 \item[(b)] conditions \eqref{cond:coal}-\eqref{cond:coal2}, 
 \item[(c)] Hypothesis \ref{hyp}
 \end{itemize}
  holds, then ill-posedness ensues.
 
 \medskip
 
We note that condition \eqref{cond:coal2} is stable by perturbation, and that conditions \eqref{cond:coal}-\eqref{cond:coal2} are generically {\it necessary and sufficient} for occurence of non-real eigenvalues in symbols that are initially hyperbolic. Indeed: 
\begin{itemize}
\item non-real eigenvalues may occur only if the initial principal symbol has double eigenvalues, implying necessity of condition \eqref{cond:coal}, and 
\item as shown by the proof of Theorem \ref{th:1/2}, the opposite sign $(\d_\l^2 P \d_t P)(0,\o_0) < 0$ in condition \eqref{cond:coal2} implies {\it real} eigenvalues for small $t > 0.$
\end{itemize} 
Here {\it generically} means that the above discussion leaves out the degenerate case $\d_t P(0,\o_0) = 0.$ 

\medskip

We consider the case $\d_t P(0,\o_0) = 0$ in Theorem \ref{th:1}. Note however that there is a significant gap between (b) and (c), the assumptions of Theorems \ref{th:1/2} and \ref{th:1}. Indeed, while condition $\d_t P = 0$ in Hypothesis \ref{hyp} lies at the boundary of the case considered in Theorem \ref{th:1/2}, Hypothesis \ref{hyp} describes a situation which is quite degenerate, since we ask for the closed conditions $\d_\l P = 0,$ $\d_t P = 0$ (and also for semi-simplicity) to hold {\it on a whole branch of eigenvalues} near $\l_0.$ 

\medskip

 Non-semi-simple eigenvalues are typically not differentiable at the coalescing point, the canonical example being
\be \label{ex}
 \left(\begin{array}{cc} 0 & 1 \\ \pm t^\a & 0 \end{array}\right)
\ee
with $\a = 1.$ The proof of Theorem \ref{th:1/2} shows that the principal symbol at $(t,x_0,\xi_0)$ can be reduced to \eqref{ex}, with $\a = 1$ and a negative sign, implying non-real, and non-differentiable eigenvalues.

By constrast, semi-simple eigenvalues admit one-sided directional derivatives (see for instance Chapter 2 of T.~Kato's treatise \cite{K}, or \cite{SW,Tp}). In particular, there is some redundancy in our assumptions of semi-simplicity and condition \eqref{non-R}.  

\medskip

We finally observe that our analysis extends somewhat beyond the framework of Theorems \ref{th:0}, \ref{th:1/2} and \ref{th:1}. Consider for instance in one space dimension a smooth principal symbol of form 
\be \label{ex:not}
 \xi \left(\begin{array}{cc} 0 & 1 \\ x^2 t - t^2 + t^3 a(x) & 0 \end{array}\right), \quad \mbox{with eigenvalues $\l_\pm = \pm \xi \big( x^2 t - t^2 + t^3 a(x) \big)^{1/2}$},
\ee 
with $a \in \R,$ so that the eigenvalues are time-differentiable only at $x = 0:$ conditions \eqref{non-R} hold only at $x = 0.$ Semi-simplicity does not hold at $(t,x) = (0,0).$ Condition \eqref{cond:coal2} does not hold at $(t,x,\l) = (0,0,0).$ However, by the implicit function theorem, eigenvalues cross at $(s(x), x)$ for a smooth $s$ with $s(x) = x^2 + O(x^3).$ By inspection, condition \eqref{cond:coal2} holds at $(s(x),x).$ Since $x$ is arbitrarily small, Theorem \ref{th:1/2} applies, yielding instability.

\subsection{Examples} \label{sec:ex}

{\it Burgers systems.} Our first
 example is the family of Burgers-type systems in one space dimension
\begin{equation} \label{burgers}
\D_t \begin{pmatrix} u_1\\u_2\end{pmatrix} + \begin{pmatrix} u_1 & - b(u)^2 u_2\\u_2&u_1\end{pmatrix}\D_x\begin{pmatrix} u_1\\u_2\end{pmatrix} = F(u),
\end{equation}
in which $b > 0$ and $F = (F_1, F_2) \in \R^2$ are smooth. When $b(u)$ is not constant, the instability result for scalar equations in \cite{LMX} does not directly apply. We show in particular that if $F_2$ is not initially zero, then Theorem \ref{th:1} yields ill-posedness for the Cauchy problem for \eqref{burgers}.  Under the same condition $F_2(t = 0) \not\equiv 0,$ Theorem \ref{th:1} also applies to two-dimensional systems 
$$ \d_t u + \left(\begin{array}{cc} u_1 \d_{x_1} & - b(u)^2 u_2 (\d_{x_1} + \d_{x_2}) \\ u_2 (\d_{x_1} + \d_{x_2}) & u_1 \d_{x_1} \end{array}\right) u = F(u).$$
Details are given in Section \ref{sec:Burgers} and \ref{sec:burgers:2}.  

\medskip

{\it Van der Waals gas dynamics.} Our results also apply to the one-dimensional, isentropic Euler equations in Lagrangian coordinates
  \begin{equation} \label{vdw-ex}
   \left\{ \begin{aligned} \d_t u_1 + \d_x u_2  & = 0, \\ \d_t u_2 + \d_x p(u_1) &  = 0, \end{aligned}\right.
   \end{equation}
 with a Van der Waals equations of state, for which there holds $p'(u_1) \leq 0,$ for some $u_1 \in \R.$
We prove that if $(\phi_1,\phi_2)$ is a smooth solution such that, for some $x_0 \in \R,$ there holds 
$$ {\rm (i)}\, p'(\phi_1(x_0,0)) < 0, \quad{ \mbox{or} } \quad {\rm (ii)}\,p'(\phi_1(0,x_0)) = 0, \,\,p''(\phi_1(0,x_0)) \d_x \phi_2(0,x_0) > 0,$$ then the initial-value problem for \eqref{vdw-ex} is ill-posed in any neighborhood of $\phi.$ Condition (i) is an ellipticity assumption (under which Theorem \ref{th:0} applies), and condition (ii) is an open condition on the boundary of the domain of hyperbolicity (under which Theorem \ref{th:1/2} applies). System \eqref{vdw-ex} possesses the formal conserved quantity $\dsp{\int_{\R} (|v(t,x)|^2 + 2 P(u(t,x))) \, dx,}$ where $P' = p.$ As briefly discussed in Section \ref{sec:local}, the instabilities we put in evidence are local and do not preclude nor are contradicted by global stability properties of the system, such as formal conservation laws.   This example is developed in Section \ref{sec:VdW}.

\medskip

{\it Klein-Gordon-Zakharov systems.}
Our last class of examples is given by the following one-dimensional Klein-Gordon systems coupled to wave equations with Zakharov-type nonlinearities:
\begin{equation} \label{kgo2} \left\{ \begin{aligned} \d_t \left(\begin{array}{c} u \\ v \end{array}\right) + \d_x \left(\begin{array}{c} v \\ u \end{array}\right) +  \left(\begin{array}{cc} \a & 0 \\ 0 & 0 \end{array}\right) \d_x \left(\begin{array}{c} n \\ m \end{array}\right) & = & (n + 1) \left(\begin{array}{c} v \\ - u \end{array}\right), \\ \d_t \left(\begin{array}{c} n \\ m \end{array}\right) +  c \d_x \left(\begin{array}{c} m \\  n \end{array}\right)  +  \left(\begin{array}{cc} \a & 0 \\ 0 & 0 \end{array}\right) \d_x \left(\begin{array}{c} u \\ v \end{array}\right) & = &  \d_x  \left(\begin{array}{c} 0 \\ u^2 + v^2 \end{array}\right).\end{aligned}\right.\end{equation}
 The linear differential operator in $(u,v)$ in the subsystem in $(u,v)$ is a Klein-Gordon operator, with critical frequency scaled to 1. The linear differential operator in $(n,m)$ in the subsystem in $(n,m)$ is a wave operator, with acoustic velocity $c.$ The source in $\d_x(u^2 + v^2)$ is similar to the nonlinearity in the Zakharov equation \cite{SS,T2}. Systems of the form \eqref{kgo2}, with $\a = 0,$ are used to describe laser-matter interactions; in the high-frequency limit, they can be formally derived from the Maxwell-Euler equations \cite{CEGT}. We consider the case $|c| < 1,$ corresponding to the physical situation of an acoustic velocity being smaller than the characteristic Klein-Gordon frequencies.

  It was shown in \cite{CEGT} that, for $\a = 0,$ system \eqref{kgo2} is conjugated via a non-linear change of variables to a semi-linear system, implying in particular well-posedness in $H^s(\R),$ for $s > 1/2.$

Here we show that if $\a \neq 0$ and $\phi = (u,v,n,m)$ is a smooth solution such that
$$ u(0,x_0) = 0, \quad v(0,x_0) = - \frac{c}{2\a}, \quad \a c \d_x u(0,x_0) > 0, \quad \mbox{for some $x_0 \in \R$},$$
then Theorem \ref{th:1/2} applies and the Cauchy problem for \eqref{kgo2} is ill-posed in the vicinity of $\phi.$ This situations is analogous to the Turing instability, in which $0$ is a stable equilibrium for both ordinary differential equations $X' = A X$ and $X ' = BX,$ but not for $X' = (A + B) X.$

 We come back to this example in detail in Section \ref{sec:KG}.

\section{Main assumption and result} \label{sec:main}

 Theorems \ref{th:0}, \ref{th:1/2} and \ref{th:1} can all be cast in the same framework, which we now present. 

\subsection{Bounds for the symbolic flow of the principal symbol} \label{sec:ass}

\subsubsection{Degeneracy index and associated parameters} \label{sec:ass:param}

Let $\ell \in \{0, 1/2, 1\}$\footnote{Theorem \ref{th:0} corresponds to the case $\ell = 0,$ while Theorem \ref{th:1/2} corresponds to the case $\ell = 1/2,$ and Theorem \ref{th:1} to $\ell = 1.$}. Associated with $\ell,$ define
 $$ h = \frac{1}{1 + \ell} = \left\{\begin{aligned} 1, & \quad \mbox{if $\ell = 0,$} \\ 2/3, & \quad \mbox{if $\ell = 1/2,$} \\ 1/2, & \quad \mbox{if $\ell = 1,$}\end{aligned}\right. \quad \mbox{and} \quad \zeta = \left\{\begin{aligned} 0, & \quad \mbox{if $\ell \in \{0,1 \},$} \\ 1/3, & \quad \mbox{if $\ell = 1/2.$}\end{aligned}\right.$$
Parameters $h$ and $\zeta$ define our time, space and frequency scales.

\subsubsection{The time transition function and the elliptic domain} Introduce a time transition function $t_\star$ such that
 
 \smallskip
 
 \noindent $\bullet$ if $\ell = 0$ or $\ell = 1,$ then $t_\star \equiv 0,$ 
 
 \smallskip
 
 \noindent $\bullet$ if $\ell = 1/2,$ then $t_\star$ depends smoothly on $(x,\xi)$ and singularly on $\e,$ and is slowly varying in $x,$ in the sense that there holds for some smooth function $\theta_\star:$  
 \be \label{def:tstar} t_\star(\e,x,\xi) = \e^{- h} \theta_\star(\e^{1-h} x, \xi), \quad \mbox{with} \quad  \theta_\star \geq 0, \,\,\, \theta_\star(0,\xi_0) = 0, \,\,\, \nabla_{x,\xi} \theta_\star(0,\xi_0) = 0, \,\, \mbox{if $\ell = 1/2.$}
\ee
Define then the elliptic domain\footnote{The elliptic domains corresponding to Theorems \ref{th:1/2} and \ref{th:1} are pictured on Figures 3 and 5.} by \be \label{def:mathcalD} {\mathcal D} := \big\{ (\t;t,x,\xi), \quad t_\star(x,\xi) \leq \t \leq t \leq T(\e),  \quad |x| \leq \delta, \quad |\xi - \xi_0| \leq \delta \e^{\zeta} \big\},\ee
for some $\xi_0 \in \R^d \setminus \{ 0 \},$ some $\delta > 0,$ with
\begin{equation} \label{max-time}
  T(\e)^{\ell+1} =  T_\star |\log\e|, \quad T_\star > 0.
    \end{equation}

 \subsubsection{The rescaled and advected principal symbol} \label{ass:a}

We consider a reference solution $\phi$ satisfying \eqref{phi}, and the associated principal symbol \eqref{symbolA}. The rescaled and advected principal symbol is\footnote{In the elliptic case, corresponding to Theorem \ref{th:0}, we have $\ell = 0,$ $h = 1,$ $Q \equiv \Id,$ $\mu \equiv 0,$ so that $(x_\star,\xi_\star) \equiv (x,\xi),$ and then $A_\star$ is simply $A_\star(\e,t,x,\xi) = A(\e t,x,\xi).$}
\be \label{def:Astar}
 A_\star(\e,t,x,\xi) := %
  \Big( Q (A - \mu) Q^{-1} \Big)\big(\e^h t, \, x_0 + \e^{1-h} x_\star(\e^h t,x,\xi), \, \xi_\star(\e^{h} t, x, \xi) \, \big),
\ee
for some $x_0 \in \R^d,$ with $0 < h \leq 1$ as in \eqref{max-time}, where, in domain ${\mathcal D}:$  

\smallskip

\noindent $\bullet$ the symbol $\mu = \mu(t,x,\xi)$ is real and smooth, and $Q(t,x,\xi) \in \C^{N \times N}$ is smooth and pointwise invertible,
\smallskip

\noindent $\bullet$ the bicharacteristics $(x_\star,\xi_\star)$ solve
  \begin{equation} \label{bichar}
   \d_t \left(\begin{array}{c} x_\star \\ \xi_\star \end{array}\right) = \left(\begin{array}{c} - \d_\xi \mu \\ \e^{1-h} \d_x \mu \end{array}\right)\big(t, x_0 + \e^{1-h} x_\star, \xi_\star\big),  \quad \left(\begin{array}{c} x_\star \\ \xi_\star \end{array}\right)(0,x,\xi) = \left(\begin{array}{c} x \\ \xi \end{array}\right).
   \ee%
We assume that the symbol $A_\star$ is block diagonal, and for the blocks $A_{\star j}$ of $A_\star$ consider either the bound 
 \be \label{block:1} \e^{h-1} |\d_x^\a \d_\xi^\b A_{\star j}| \leq C_{\a\b} < \infty, \quad \mbox{for some $C_{\a\b} > 0,$ in ${\mathcal D},$ uniformly in $\e,$}\ee
 or the block structure 
 \be \label{block:2} \begin{aligned} \e^{h-1} A_{\star j} & = \left(\begin{array}{cc} 0 & \e^{h-1} A_{\star j12} \\ \e^{1-h} A_{\star j21} & 0 \end{array}\right), \\ & \mbox{with $|\d_x^\a \d_\xi^\b A_{\star j 12}| + |\d_x^\a \d_\xi^\b A_{\star j21}| \leq C_{\a\b}$ for some $C_{\a\b} > 0,$ in ${\mathcal D},$ uniformly in $\e.$} \end{aligned}
\ee

 If $\ell = 0,$ then $h = 1.$ As a consequence, in the block diagonalization of $A_\star$ all blocks satisfy \eqref{block:1}, by the assumed smoothness of the components of $A_\star.$ 
 
 If $\ell = 1/2,$ we assume that some block of $A_\star$ satisfies \eqref{block:2}, and the other blocks of $A_\star$ satisfy \eqref{block:1}.
 
 If $\ell = 1,$ we assume that all blocks of $A_\star$ satisfy \eqref{block:1}. 

\subsubsection{Symbolic flow and growth functions}
The symbolic flow $S = S(\t;\e, t,x,\xi)$ of $A_\star$ is defined as the solution to the family of linear ordinary differential equations
 \begin{equation} \label{tildefund}
  \d_t S  + i \e^{h-1} {A}_\star(\e,t,x,\xi) S  = 0,
  \qquad S(\t;\t)  \equiv \Id.
  \end{equation} 
In the above Section \ref{ass:a}, we assumed that $A_\star$ is block diagonal. Accordingly, the solution $S$ to \eqref{tildefund} is block diagonal, with blocks $S_{(1)},S_{(2)},\dots$ 

Let $\g^\pm(x,\xi)$ be two continuous functions defined on $\{ |x| \leq \delta, \,\ |\xi - \xi_0| \leq \delta \e^{\zeta}\},$ such that $\g^-(0,\xi_0) = \g^+(0,\xi_0),$ and
\begin{equation} \label{def:bf-e}
 {\bf e}_{\g^\pm}(\t;t,x,\xi) := \exp\left( \g^\pm(x,\xi)\big( (t - t_\star(x,\xi))_+^{\ell + 1} - (\t - t_\star(x,\xi))_+^{\ell + 1} \big) \right),
\ee
where $t_+ := \max(t,0)$ and the time transition function $t_\star$ is defined in \eqref{def:tstar}.
We understand ${\bf e}_{\g^\pm}$ as {\it growth functions}\footnote{In the elliptic case, we have $\ell = 0,$ $t_\star \equiv 0,$ so that the growth functions are simply ${\bf e}_{\g^\pm} = e^{\g^\pm (t - \t)}.$}, measuring how fast the solution $S$ to \eqref{tildefund} is growing, as seen in Assumption \ref{ass:main} below. The associated $\g^\pm$ are {\it rates of growth}. 

\subsubsection{Bounds} We postulate bounds for $S$ in the elliptic domain ${\mathcal D}$ in terms of the growth functions ${\bf e}_{\g^\pm}:$

\begin{ass} \label{ass:main} For some $(x_0,\xi_0) \in U \times (\R^d \setminus \{0 \}),$ some $\ell,$ $Q,\mu,\g^\pm,$ and $t_\star$ as above, any $T_\star > 0,$ for some $\delta > 0,$ for symbol $A_\star$ satisfying the structural assumptions of Section {\rm \ref{ass:a}}, for some $\e_0 > 0,$ there holds for all $0 < \e < \e_0,$ for the solution $S$ to \eqref{tildefund}:%
 \begin{itemize}
 \item the lower bound, for some smooth family of unitary vectors $\vec e(x) \in \C^N,$ for $|x| < \delta:$ 
  \begin{equation} \label{ass:low}
  \e^{-\zeta} {\bf e}_{\g^-}(0;T(\e),x,\xi_0) \lesssim  \big| \, S(0;T(\e),x,\xi_0) \vec e(x)\,\big|\,,
  \ee
 \item  %
  the upper bound
   for the $j$th diagonal block $S_j$ of $S,$  for $(\t,t,x,\xi) \in {\mathcal D}:$ %
  \begin{equation} \label{ass:up}
   | S_{(j)}(\t;t,x,\xi)| \lesssim \left(\begin{array}{cc} 1 & \e^{-\zeta} \\ \e^\zeta & 1 \end{array}\right) {\bf e}_{\g^+}(\t;t,x,\xi).
   \ee
 \end{itemize}
\end{ass}

In \eqref{ass:low}, notation $a \lesssim b,$ where $a$ and $b$ are functions of $(\e,\t,t,x,\xi),$ is used to mean existence of a {\it uniform} bound
 \be \label{notation:lesssim}
  a(\e,\t,t,x,\xi) \leq C |\ln \e|^{C'} b(\e,\t,t,x,\xi),
 \ee
 where $C > 0$ and $C' > 0$ do not depend on $(\e,\t,t,x,\xi).$ This means in particular that powers of $|\ln \e|$ play the role of constants in our analysis. They are indeed destined to be absorbed by arbitrarily small powers of $\e.$

In \eqref{ass:up}, we use notation $\lesssim$ for matrices. Here we mean {\it block-wise} inequalities ``modulo constants", in the sense of \eqref{notation:lesssim}. That is, in \eqref{ass:up} we assume that the $j$-th diagonal block of $S$ itself has a block structure, with the top left block being bounded {\it entry-wise} by ${\bf e}_{\g^+},$ the top right block being bounded {\it entry-wise} by $\e^{-\zeta} {\bf e}_{\g^+},$ etc. 

We further comment on Assumption \ref{ass:main} in Section \ref{guide:ass}.   

\subsection{Hadamard instability} \label{sec:th}

The non-linear
  information
   contained in
  Assumption
  \ref{ass:main} on the symbolic flow of the principal symbol \eqref{symbolA} transposes into an instability result for the quasi-linear system \eqref{1st-system}: 

\begin{theo} \label{th:main} Under Assumption {\rm \ref{ass:main}}, system \eqref{1st-system} is ill-posed in the vicinity of the reference solution $\phi,$ in the sense of Definition {\rm \ref{def:insta}}.
 \end{theo}

 Theorem \ref{th:main} states that either there exists no solution map, or the solution map fails to enjoy any H\"older-type continuity estimates. The proof of Theorem \ref{th:main} is given in Section \ref{sec:proof}. Key ideas in the proof are sketched in Section \ref{sec:key}.

 Theorems \ref{th:0}, \ref{th:1/2} and \ref{th:1} all follow from Theorem \ref{th:main}.

\subsection{Comments on Assumption {\rm \ref{ass:main}}} \label{guide:ass} 

Our main assumption is flexible enough to cover the three different situations described in Theorems \ref{th:0}, \ref{th:1/2} and \ref{th:1}. Before further commenting on its ingredients in Section \ref{sec:ingr} and its verification in Section \ref{sec:verif}, we point out two key features:

\smallskip

$\bullet$ Assumption \ref{ass:main} is {\it nonlinear}. It bears on the whole system \eqref{1st-system}, not just the principal symbol. For instance, instability occurs for the Burgers systems of Section \ref{sec:ex} under a condition bearing on the nonlinear term $F.$ 

\smallskip

$\bullet$ Assumption \ref{ass:main} is {\it finite-dimensional}, in the sense that it postulates bounds for solutions to ordinary differential equations in a finite-dimensional setting. These are turned into bounds for the solutions to the partial differential equations via Theorem \ref{duh1}. An informal discussion of the role of Theorem \ref{duh1} is given in Section \ref{sec:key}. 

\subsubsection{On the ingredients of Assumption {\rm \ref{ass:main}}} \label{sec:ingr}

${}^{}$

\smallskip

$\bullet$ Our localization constraints in $(x,\xi) \in \R^{2d}$ respect the uncertainty principle. Indeed, we localize spatially in a box of size $\sim \e^{1 - h}.$ We localize in frequency in a box of size $\sim \e^{\zeta}$ around $\xi_0$ but then in the proof use highly-oscillating data and an $\e^h$-semi-classical quantization, so that it is $\e^h \xi$ which belongs in a box of size $\e^{\zeta}$ around $\xi_0,$ meaning a frequency localization in a box of size $\e^{\zeta - h}.$ If $\ell = 0$ or $\ell = 1,$ then $\zeta = 0.$ The area of the $(x,\xi)$-box is then $\e^{2d(1 - 2h)} \geq 1,$ since $h = 1$ or $h = 1/2.$ If $\ell = 1/2,$ then $h = 2/3$ and $\zeta = 1/3.$ The area of the $(x,\xi)$-box is $\e^{2d(1 - h + \zeta - h)}= 1.$

\smallskip

$\bullet$ The index $\ell$ measures the degeneracy of the defect of hyperbolicity. We have $\ell = 0$ in the case of an initial ellipticity (Theorem \ref{th:0}), $\ell = 1/2$ in the case of a non semi-simple defect of hyperbolicity (Theorem \ref{th:1/2}) and $\ell = 1$ in the case of a semi-simple defect of hyperbolicity (Theorem \ref{th:1}). The instability is recorded in time $O(\e |\ln \e|)^{1/(1 + \ell)}$ for initial frequencies $O(1/\e).$ In particular, the higher the degree of degeneracy, the longer we need to wait in order to record the instability. 

\smallskip

$\bullet$ In the case $\ell > 0,$ eigenvalues of the principal symbol are initially real (hyperbolicity). Instability occurs as (typically) a pair of eigenvalues branch out of the real axis at $t = 0.$ The matrix $Q$ should be understood as a change of basis, which includes a projection onto the space of bifurcating eigenvalues. The scalar $\mu$ corresponds to the real part of the bifurcating eigenvalues. Assumption \ref{ass:main} is formulated for the principal symbol evaluated along the bicharacteristics of $\mu.$

\smallskip

$\bullet$ In the non semi-simple case, the defect of hyperbolicity is typically {\it not} uniform in $(x,\xi).$ That is, if eigenvalues branch out of the real axis at initial time at the distinguished point $(x_0,\xi_0),$ then the branching will typically occur for ulterior times $t_\star(x,\xi) > 0$ for $(x,\xi)$ close to $(x_0,\xi_0).$ This is clearly seen in Lemma \ref{lem:sp-1402}, under the assumptions of Theorem \ref{th:1/2}, and pictured on Figure 3. 

\smallskip

$\bullet$ The parameter $\g^+$ corresponds to an upper rate of growth. In the elliptic case, $\g^+$ is equal to the largest imaginary part in the initial spectrum, as seen in Section \ref{sec:proof0}. In the case of a smooth defect of hyperbolicity, $\g^+ = \Im m \, \d_t \l(0,x,\xi),$ where $\l$ is a bifurcating eigenvalue, as seen in Section \ref{sec:proof1}.   

\smallskip

$\bullet$ In the case $\ell = 1/2,$ the block structure \eqref{block:2} derives from a reduction of the principal symbol to normal form; see Sections \ref{sec:branch} and \ref{sec:prep} in the proof of Theorem \ref{th:1/2}. 

\smallskip

$\bullet$ In the case $\ell = 1,$ the block structure \eqref{block:1} reveals a cancellation, seen on \eqref{cancel:1} in the proof of Theorem \ref{th:1}.

\smallskip

$\bullet$ The smoothly varying direction $\vec e\,(x)$ along which the lower bound \eqref{ass:low} holds is not necessarily an eigenvector of $A_\star;$ see the discussion in Section \ref{sec:on:transition} and Lemma \ref{lem:low:1/2}. 
 
\subsubsection{On verification of Assumption {\rm \ref{ass:main}}} \label{sec:verif}
 
 We give in Theorems \ref{th:0}, \ref{th:1/2} and \ref{th:1} sufficient conditions, expressed in terms of the spectrum of $A$ and the jet of the characteristic polynomial of $A$ at $t = 0$ for Assumption \ref{ass:main} to hold. These sufficient conditions are satisfied in particular by Burgers, Van der Waals, and Klein-Gordon-Zakharov systems (Section \ref{sec:examples}). These conditions bear only on the coefficients of the system (the differential operator  and the source) and $\phi(0),$ the initial datum of the reference solution. In particular, we may in practice verify these conditions without of course having any knowledge of $\phi(t)$ for $t > 0.$

\subsubsection{On spectral conditions describing the transition from hyperbolicity to ellipticity} \label{sec:on:transition}
 
 Conditions \eqref{cond:ell}, \eqref{cond:coal}-\eqref{cond:coal2} and \eqref{non-R} are all expressed in terms of the characteristic polynomial of $A.$ Their generalizations in the form of conditions \eqref{ass:low} and \eqref{ass:up} are expressed in terms of the symbolic flow of $A.$ Our point here is to explain why conditions bearing on the {\it spectrum} of $A$ do not seem to be appropriate. The discussion below also highlights three difficulties in the analysis of the case $\ell = 1/2:$ the lack of smoothness of the eigenvalues of the principal symbol, the lack of uniformity of the transition time (in the sense that the function $t_\star$ does depend on $(x,\xi)$), and the lack of smoothness of the eigenvectors. 
 
  A simple way to express the fact that an eigenvalue $\l$ of $A$ branches out of the real axis at $t = 0$ is 
 \be \label{cond:branch} \l(0,x,\xi) \in \R, \quad \mbox{for all $(x,\xi)$ near $(x_0,\xi_0),$ \qquad with} \quad \Im m \, 
\d_t \l(0,x_0,\xi_0) > 0.\ee
But then by reality of the coefficients of $A,$ eigenvalue branch out of $\R$ in pairs, so that $(0,x_0,\xi_0)$ is a branching point in the spectrum, and typically eigenvalues are {\it not} differentiable at a branching point, so that \eqref{cond:branch} is not nearly general enough. For instance, the eigenvalues of the principal symbol for the one-dimensional compressible Euler equations \begin{equation} \label{vdw}
   \left\{\begin{aligned} \d_t u + \d_x v & = 0, \\ \d_t v + \d_x p(u) & = 0,
   \end{aligned}\right.
   \end{equation}
 are 
 $$\l_\pm(t,x,\xi) = \pm \xi \big( p'(u(t,x)) \big)^{1/2}.$$
  For a Van der Waals equation of state, for which there holds
 $ p'(u) \leq 0,$ for some $u \in \R,$ a transition from hyperbolicity to ellipticity occurs for data $u(0,\cdot)$ satisfying
 \be \label{prep:vdw} p'(u(0,x_0)) = 0, \qquad \d_t \big( p'(u(0,x_0)\big)_{|t = 0} = - p''(u(0,x_0)) \d_x v(0,x_0) < 0.\ee
The associated eigenvalues are $O(t^{1/2}),$ in particular not time-differentiable at $t = 0,$ so that condition \eqref{cond:branch} is not appropriate. 

A way around this difficulty is to consider the integral growth condition 
 \begin{equation} \label{cond:branch2}
  \int_0^{t} \Im m \, \l(\t,x, \xi) \,d\t =  %
   \g({t},x,\xi)  %
  t^{\ell+1}, \qquad \g(0,x_0,\xi_0) > 0,
  \end{equation}
 for some $\ell \geq 0$ and some rate function {$\g$ that is continuous in $(t,x,\xi)$ at $(0,x_0,\xi_0),$ and some local solution $\l$ of $P = 0.$ Condition \eqref{cond:branch2} may be verified with the Puiseux expansions of the eigenvalues at $t = 0,$ such as in the Van der Waals example (for details in Puiseux expansions, see for instance chapter 2 in \cite{K}, or Proposition 4.2 in \cite{Tp}). There are, however, at least two serious problems with \eqref{cond:branch2}.
 
 The first is that in \eqref{cond:branch2}, it is assumed that the loss of hyperbolicity occurs at $t = 0$ over a whole neighborhood of $(x_0,\xi_0),$ which is typically {\it not} the case. Consider in this view the preparation condition \eqref{prep:vdw} for the datum. From the second condition in \eqref{prep:vdw}, we find by application of the implicit function theorem that in the vicinity of $(0,x_0)$ the set $\{ p'(u) = 0 \}$ is the graph of a smooth map $x \to t_*(x).$ 
The transition curve $x \to t_*(x),$ defined locally in a neighborhood of $x_0,$ parameterizes the loss of hyperbolicity: for $t < t_*(x),$ there holds $p'(u) > 0,$ implying $\Im m  \, \l_\pm \equiv 0,$ while for $t > t_*(x)$ there holds $p'(u) < 0,$ implying $\Im m \, \l_\pm \neq 0.$ On the curve $t = t_*(x),$ there holds $p'(u) \equiv 0,$ implying that the eigenvalues coalesce: $\l_- = \l_+.$ This means in particular that for $(x,\xi)$ close to, and different from, $(x_0,\xi_0)$ we should not expect the loss of hyperbolicity to be instantaneous as in \eqref{cond:branch2}, but rather to happen at time $t_*(x,\xi),$ and condition \ref{cond:branch2} should be replaced by
 \begin{equation} \label{cond:branch3}
   \int_{t_*(x,\xi)}^{t} \Im m \, \l(\t,x, \xi) \,d\t =  %
   \g({t},x,\xi)  %
  (t - t_*(x,\xi))^{\ell+1}, \qquad \g(0,x_0,\xi_0) > 0,
  \end{equation}
 for some smooth time transition function $t_* \geq 0,$ with $t_*(x_0,\xi_0) = 0.$

The second issue with \eqref{cond:branch2}, still present in \eqref{cond:branch3}, is that while failure of time-differenta\-bi\-lity of the eigenvalues is accounted for in \eqref{cond:branch2}, the associated lack of regularity of eigenvectors is not. For instance, in the Van der Waals system \eqref{vdw}, the eigenvectors of the principal symbol are $e_\pm := (1, \pm (p'(u))^{1/2}).$ 
  In particular, under condition \eqref{prep:vdw}, 
  the eigenvectors $e_\pm$ are not time-differentiable at $t = 0.$ It is then not clear how to convert conditions \eqref{cond:branch2} and \eqref{cond:branch3} into growth estimates for the corresponding system of partial differential equations. Indeed, for instance in the simpler case of ordinary differential equations, spectral estimates such as \eqref{cond:branch2} or \eqref{cond:branch3} are typically converted into growth estimates for the solutions via projections onto spectral subspaces, an operation that requires smooth projections.  

 We conclude this discussion by sketching a way around the issue of the lack of regularity of eigenvectors. Going back to the Van der Waals example, consider the ordinary differential equations
 \begin{equation} \label{flow0} \d_t S + i \xi \left(\begin{array}{cc} 0 & 1 \\ p'(u) & 0 \end{array}\right) S = 0, \qquad S(\t;\t) \equiv \Id,
 \end{equation}
 parameterized by $(x,\xi).$ 
 Under condition \eqref{prep:vdw}, there holds 
  $$p'(u(t,x)) = - \a(x) t + O(t^2),$$
  for $(t,x)$ close to $(0,x_0).$ Restricting for simplicity to the case $p'(u(t,x)) = - t,$ we find that the entries $(y,z)$ of a column of $S$ satisfy the system of ordinary differential equations
  $$  y' + i \xi z = 0, \qquad z' - it\xi y  = 0,$$
  implying that $y$ satisfies the Airy equation
 \begin{equation} \label{airy00} y'' = t \xi^2 y,
 \end{equation}
  for which sharp lower and upper bounds are known.
   
This motivates consideration, in Section \ref{sec:ass}, of the symbolic flow associated with the principal symbol $A.$ An important issue is then the conversion of growth conditions for the symbolic flow into estimates for the solutions to the system of partial differential equations. This is achieved via Theorem \ref{duh1}.

\subsection{{On the proof of Theorem {\rm \ref{th:main}}}} \label{sec:key}

 We give here an informal account of key points in the proof of Theorem \ref{th:main}. The proof is in three parts: (1) preparation steps which transform the equation into the prepared equation \eqref{eq:v}-\eqref{def:g}, (2) the use of a Duhamel representation formula, (3) lower and upper bounds. 
 
 \medskip
 
 (1) We introduce a spatial scale $h$ and write perturbations equations about the reference solution $\phi.$ We then block diagonalize the principal symbol (this is $Q$ from the Assumption \ref{ass:main}), localize in space around the distinguished point $x_0,$ factor out the real part of the branching eigenvalues (this is symbol $\mu$ from Assumption \ref{ass:main}) and change to a reference frame defined by the bicharacteristics of $\mu.$ Finally we operate a stiff localization in the elliptic domain ${\mathcal D}$ given by Assumption \ref{ass:main}, and rescale time. The key point in these preparation steps is to carefully account for the {\it linear errors} in the principal symbol, which take the form of commutators. The resulting principal symbol is a perturbation of the symbol $A_\star(\e,t,x,\xi)$ defined in Assumption \ref{ass:main}. 
 
 \medskip

(2) Assumption \ref{ass:main} provides bounds for the flow of $A_\star.$  As pointed out in Section \ref{guide:ass}, these bounds bear on solutions to ordinary differential equations in finite dimension, in particular they are, at least theoretically, easier to verify than bounds bearing on spectra of differential operators. We use Theorem \ref{duh1}, drawn from \cite{T3} and proved in Appendix \ref{sec:Duhamel}, in order to convert these bounds into estimates for a solution to \eqref{1st-system}.
 
 Consider a pseudo-differential Cauchy problem\footnote{Notations and results pertaining to pseudo-differential calculus are recalled in Appendix \ref{sec:symbols}.}
 \be \label{eq:B} \d_t u + \op_\e({\mathcal A}) u = g, \qquad u(0) = u_0 \in L^2,\ee
 where ${\mathcal A}$ is a symbol of order zero. Above, $\op_\e({\mathcal A})$ denotes the $\e^h$-semi-classical quantization of symbol ${\mathcal A},$ as defined in \eqref{semicl:0}. Associated with the above Cauchy problem in infinite dimensions, consider the Cauchy problem is finite dimensions
 \be \label{flow:B} \d_t S + {\mathcal A} S = 0, \qquad S(\t;\t) = \Id.\ee
 Theorem \ref{duh1} asserts that if $S(\t;t)$ and its $(x,\xi)$-derivatives grow in time like $\exp(\g t^{1 + \ell}),$ with a rate $\g > 0$ and a degeneracy index $\ell \geq 0,$ then $\op_\e(S)$ furnishes a good approximation to a solution operator for $\d_t + \op_\e({\mathcal A}),$ in time $O(|\ln \e|)^{1/(1 + \ell)}.$ That is, the solution of \eqref{eq:B} is given by 
  \begin{equation} \label{duh-rem} u(t) \simeq \op_\e(S(0;t)) u_0 + \int_0^t \op_\e(S(\t;t)) g(\t) \, d\t.
  \end{equation}
  
\medskip

(3) The preparation steps (see (1), above) reduced our problem to a system of form \eqref{eq:B}. Via representation \eqref{duh-rem}, upper and lower bounds for the solution $u$ to \eqref{eq:B} are easily derived from the bounds of Assumption \ref{ass:main}, and from {\it postulated} bound for the source $g.$ In our proof, the source $g$ comprises in particular nonlinear errors. Since we have no way of bounding solutions to \eqref{1st-system} near $\phi$ (the impossibility of controlling the growth of solutions with respect to the initial data being precisely what we endeavor to prove), we {\it assume} a priori bounds for the solution. The compared growth of $\op_\e(S) u_0$ and the Duhamel term from \eqref{duh-rem} eventually provide a contradiction. Note that the a priori bound (see \eqref{new-big-bound} in Section \ref{sec:apb}) is particularly weak, since we allow for arbitrarily large losses of derivatives.  

\smallskip

We finally note that G\r{a}rding's inequality (see for instance Theorem 1.1.26 in \cite{Le}) asserts that nonnegativity of symbol ${\mathcal A}$ implies semi-positivity of operator $\op_\e({\mathcal A}).$ This is the classical tool for converting bounds for symbols into estimates for the associated equations. It is shown in \cite{T3} how estimates derived from G\r{a}rding's inequality fail to be sharp in the non-self-adjoint case, as opposed to bounds based on Theorem \ref{duh1}.

\section{Proof of Theorem \ref{th:main}} \label{sec:proof}

As discussed in Section \ref{sec:key}, the proof decomposes into three parts: 
\begin{itemize}
\item[(1)] {\it Preparation} steps which transform the original equation \eqref{1st-system} into the prepared equation \eqref{eq:v}-\eqref{def:g}. This step covers Sections \ref{sec:init} to \ref{sec:trunc}. 
\item[(2)] The use of a {\it Duhamel representation formula} in which the solution to the prepared equation is expressed as the sum of the ``free" solution (defined as the action of an approximate solution operator on the datum) and a remainder -- Sections \ref{integral} and \ref{sec:bd-sol-op}; 
\item[(3)] {\it lower bounds} for the free solution, and {\it upper bounds} for the remainder conclude the proof in a third step -- Sections \ref{sec:bd-source} to \ref{sec:end0}. 
\end{itemize} %

\subsection{Initial perturbation} \label{sec:init}

The goal is to prove ill-posedness, in the sense of Definition \ref{def:insta}. Parameters $m,\a,\delta,T$ are given, as in \eqref{param}, and we endeavor to disprove \eqref{stab}. Define 
 \be \label{def:psi}
 \varphi_0(\e,x) := \Re e \, \Big( \op_\e(Q_\e(0)^{-1}) \Big( e^{i (\cdot) \cdot \xi_0/\e^{h}} \theta \vec e \, \Big)(x) \, \Big), \quad \e > 0, \quad h = \frac{1}{1 + \ell}, \,\,\, \ell \geq 0,
\ee  
where $(x_0,\xi_0)$ is the distinguished point in the cotangent space given in Assumption \ref{ass:main}, and 

\smallskip

$\bullet$ $\op_\e(\cdot)$ denotes a pseudo-differential operator in $\e^{h}$-semi-classical quantization: 
 \be \label{semicl:0} \op_\e(a) v := (2\pi)^{-d}\int_{\R^d} e^{i x \cdot \xi} a(t,x,\e^{h} \xi) \hat v(\xi) \, d\xi\,;\ee

$\bullet$ $Q_\e(0) = Q(0, x_0 + \e^{1-h} x, \xi),$ with $Q$ as in Assumption \ref{ass:main};

\smallskip

$\bullet$ the vector $\vec e$ is as in Assumption \ref{ass:main};

\smallskip

$\bullet$ the spatial cut-off $\theta \in C^\infty_c(\R^d)$ has support included in $B(0,\delta),$ and is such that $\theta \equiv 1$ in $B(0,1/2).$%

\medskip

 Consider the following family of data, indexed by $\e > 0:$  %
 \begin{equation} \label{phie}
  u^\e(0,x) = \phi(0,x) + \e^{K} \varphi_0\left(\frac{x - x_0}{\e^{1-h}}\right)
  \ee
  where $\phi(0,x)$ is the datum for the background solution $\phi$ \eqref{phi}, and $K$ will be chosen large enough so that $u^\e(0)$ is a small perturbation of $\phi(0)$ in $H^m$ norm.

\medskip

Theorem \ref{th:main} is a consequence of the following result for the family of initial-value problems \eqref{1st-system}-\eqref{phie}, indexed by $\e:$
 \begin{theo}\label{theo-hadamardinstab}
 {Given the parameters defined in \eqref{param}, given a local solution $\phi$ of \eqref{1st-system} satisfying \eqref{phi} with $s_1$ large enough, under Assumption
 {\rm \ref{ass:main}},  %
 if $K$ is large enough:}
  \begin{itemize}
 \item either for all $T$ and $\delta$ with $0 < T \leq T_0,$ $B(x_0,\delta) \subset U$ there is no $\e_0 > 0$ such that for all $0 < \e < \e_0,$ the initial-value problem for \eqref{1st-system} with the initial datum \eqref{phie} has a solution in $L^\infty([0,T], W^{1,\infty}(B(x_0,\delta)),$
 
 \smallskip
 
 \item or for some $T$ and $\delta$ with $0 < T \leq T_0,$ $B(x_0,\delta) \subset U,$ some $\e_0  >0,$ all $0 < \e < \e_0,$ the initial-value problem for the system \eqref{1st-system} with the initial datum \eqref{phie} has a solution $u^\e$ in $L^\infty([0,T], W^{1,\infty}(B(x_0,\delta)),$ and there holds 
\begin{equation}\label{instab-est}\sup_{\begin{smallmatrix} 0 < \e < \e_0 \\ 0 \leq t \leq \e^{h} T(\e) \end{smallmatrix}} \frac{\|u^\epsilon(t)-\phi(t)\|_{W^{1,\infty}(B(x_0,\e^{1-h} \delta))}}{\| u^\epsilon(0) - \phi(0)\|^\a_{H^m(U)}} = +\infty\end{equation}
 {where $T(\e)$ is defined in \eqref{max-time}, so that, in particular, $\e^{h} T(\e) \to 0$ as $\e \to 0.$} \end{itemize}
 \end{theo}

\begin{lem} \label{lem:cor:inti} 
Theorem {\rm \ref{theo-hadamardinstab}} implies Theorem {\rm \ref{th:main}.}
\end{lem}

\begin{proof} There holds 
$$ \varphi_0\left(\frac{x - x_0}{\e^{1-h}}\right) = \Re e \, e^{i (x - x_0) \xi_0/\e} \tilde \varphi\left(\frac{x - x_0}{\e^{1-h}}\right) \,,$$
where $\tilde \varphi := \op_\e(Q(0,\cdot,\xi_0 + \cdot)) (\theta \vec e\,),$ hence 
 $$ \left\| \varphi_0 \left(\frac{x - x_0}{\e^{1-h}}\right) \right\|_{H^m(U)} \lesssim \e^{-m + (1-h) d/2}.$$
 Let $K>m - (1-h) d/2.$ Then,  
 \begin{equation} \label{lim} %
  \|u^\e(0,\cdot) -\phi(0,\cdot)\|_{H^m(U)}
  \lesssim \e^{K - m + (1-h) d/2} \underset{\e \to 0}{\longrightarrow} 0.
 \end{equation}
Thus given a neighborhood $\mathcal U$ of $\phi(0)$ in $H^m(U),$ and if $\e$ is small enough, then $u^\e(0)$ lies in ${\mathcal U}.$

 If for some $\e$ small enough, the initial-value problem \eqref{1st-system}-\eqref{phie} does not have a solution, then this means ill-posedness in the sense of Definition \ref{def:insta}. If there is a solution for any small $\e,$ then \eqref{instab-est} disproves \eqref{stab}, since the sequence ${\e^{h} T(\e)}$ converges to 0, and again this means ill-posedness in the sense of Definition \ref{def:insta}. 
\end{proof}

\subsection{The posited solution and its avatars}  \label{sec:avatars}

We assume that for some $0 < T \leq T_0,$ some $0 < \delta$ with $B(x_0,\delta) \in U,$ some $\e_0 > 0,$ all $0 < \e < \e_0,$ the Cauchy problem for \eqref{1st-system} with the initial datum \eqref{phie}%
has a unique solution
 \begin{equation} \label{delta0} u^\e \in L^\infty([0,T], W^{1,\infty}(B(x_0,\delta))).
 \end{equation}
 Our goal is then to prove \eqref{instab-est}. For future reference, we list here the successive avatars of the solution that we will use in this proof: 
 $$ \left\{ \begin{array}{r|l|l|l} \dot u & \mbox{perturbation} & \dot u := (u^\e - \phi)(t, x_0 + \e^{1-h} x) & \eqref{def:dot-u} \\ u^\flat & \mbox{spatial localization and projection} & u^\flat := \op_\e(Q_\e) (\theta \dot u) & \eqref{def:uflat} \\ u_\star & \mbox{convection} & u_\star := M^\star(0;t) u^\flat & \eqref{def:tildeu} \\ v & \mbox{stiff truncation and rescaling in time} & v := \big( \op_\e(\chi) u_\star \big)(\e^h t), & \eqref{def:vpsi} 
  \end{array}\right.$$

\subsection{Amplitude of the perturbation, limiting observation time and observation radius} \label{sec:param}

The parameter $K$ measures the size of the initial perturbation \eqref{phie}. We choose $K$ to be large enough: 
\be \label{cond:K}
(2\a - 1) K > 2 \a m + (1 - \a) (1-h) d,
\ee
where $m$ measures the loss of Sobolev regularity and $\a$ the loss of H\"older continuity of the flow (as seen on target estimate \eqref{instab-est}), and $h = 1/(1 + \ell).$%

The parameter $T_\star,$ defined in \eqref{max-time}, measures the final observation time in rescaled time frame. In the original time frame, the final observation time is $\big( \e T_\star |\ln \e|\big)^{1/(1+\ell)}.$ We choose $T_\star$ to be large enough:   
\be \label{cond:T1}
 \g^-(0,\xi_0) T_\star > K,
\ee 
depending on $K$ and the lower rate of growth $\g^-$ introduced in \eqref{def:bf-e}.

The parameter $\delta$ measures the radius of the observation ball $B(0,\delta)$ where the analysis takes place. (The radius is $\e^{1-h} \delta$ in the original spatial frame, and just $\delta$ in the rescaled spatial frame associated with $\dot u;$ see Section \ref{sec:avatars} above and \eqref{def:dot-u}.) If Theorem \ref{theo-hadamardinstab} holds for a given value of $\delta,$ then it holds for any smaller radius. In particular, we may assume that the given value of $\delta$ is so small that the bounds of Assumption \ref{ass:main} hold on $|x| + |\xi - \xi_0| \leq \delta.$ In the final steps of our analysis (Sections \ref{sec:free} and \ref{sec:end0}), we will further choose $\delta$ to be small enough, depending on the growth functions $\g^\pm$ introduced in Assumption \ref{ass:main} and $T_\star$ (see condition \eqref{cond:delta0} and the proof of Corollary \ref{lem:low}).

\subsection{The perturbation equations} \label{sec:pert}
  Our analysis is local in $t,x,\xi,$ with $0 \leq t \leq \e^h T(\e) ,$ $|x - x_0| \leq \e^{1-h} \delta$ and $|\xi - \xi_0| \leq \delta,$ where $T(\e)$ is defined in \eqref{max-time}, and $T_\star$ and $\delta$ are defined in Section \ref{sec:param}. 
  
 The perturbation variable $\dot u$ is defined in a rescaled spatial frame by
  \begin{equation} \label{def:dot-u} \dot u(\e,t,x) := \big( u^\e - \phi \big)(t, x_0 + \e^{1-h} x), \quad \mbox{ with $h = 1/(1+\ell).$}
  \end{equation}
The equation in $\dot u$ is 
\begin{equation}\label{eq:dotu} \D_t \dot u + \e^{-1} A\big(t, \, x_0 + \e^{1-h} x, \, \e^{h}\d_x\big) \dot u + \dot B(\e,t,x) \dot u =  \dot F. \end{equation}
 where $A$ is the $1$-homogeneous principal symbol \eqref{symbolA}, $\dot B$ is order zero: 
\begin{equation}\label{def:dotB} 
 \dot B(\e,t,x) \dot u := \sum_j (\d_u A_j(t,x_0 + \e^{1-h} x,\phi_\e)\, \dot u \,) \D_{x_j}\phi_\e - \d_u F(t, x_0 + \e^{1-h} x, \phi_\e) \dot u,\end{equation}
 with notation $\phi_\e := \phi(t, x_0 + \e^{1-h} x).$ In \eqref{eq:dotu}, the source $\dot F$ comprises nonlinear terms: 
\be \label{def:dotF} \dot F  = G_0(\e, t,x, \dot u) \cdot (\dot u, \dot u)+ \sum_{1 \leq j \leq d} G_{1j}(\e, t,x, \dot u) \cdot (\dot u, \d_{x_j} \dot u),\ee
  where $(u,v) \to G_0(\e,t,x,\dot u) \cdot (u,v)$ and $(u,v) \to G_{1j}(\e,t,x,\dot u) \cdot (u,v)$ are bilinear, defined as
 $$ \begin{aligned} G_0(\e,t,x,\dot u) & := - \int_0^1 (1 - \t) \big( \sum_{1 \leq j \leq d} \d_u^2 A_j(\phi_\e + \t \dot u) \d_{x_j} \phi_\e - \d_u^2 F(\phi_\e + \t \dot u) \big) \, d\t, \\ 
  G_{1j}(\e,t,x,\dot u) & := - \int_0^1 \d_u A_j(\phi_\e + \t \dot u) \, d\t.\end{aligned}$$
We omitted above the arguments $(t, x_0 + \e^{1-h} x)$ of $\d^k_u A_j$ and $\d_u^2 F.$ In this proof, a perturbative analysis around $\phi$ at $(x_0, \xi_0),$ we will handle $\dot F$ as a small source, and $\dot B$  as a small perturbation of the principal symbol.

 \subsection{A priori bound} \label{sec:apb}

The goal is to prove the instability estimate \eqref{instab-est}.  We work by contradiction, as we assume that
there exists $C > 0,$ such that for all $t \in [0, \e^h T(\e)],$ there holds
$$ %
   \| u^\e(t) - \phi(t) \|_{W^{1,\infty}(B(x_0,\e^{1-h} \delta))} \leq C \| u^\e(0) - \phi(0) \|_{H^m(U)}^\a,    $$ %
   uniformly in $(\e,t),$ for $0 \leq t \leq \e^{h} T(\e) = \big( \e T_\star |\ln \e|\big)^{1/(1 + \ell)}.$ By choice of the initial datum \eqref{phie}, this implies (see the proof of Lemma \ref{lem:cor:inti})
   \begin{equation} \label{new-big-bound}
   \| u^\e(t) - \phi(t) \|_{W^{1,\infty}(B(x_0, \e^{1-h} \delta))} \leq C \e^{\a(K - m + (1-h) d/2)},
   \ee
with a possibly different constant $C > 0,$ for all $t \leq  \e^h T(\e).$ By definition of $\dot u,$ this implies 
\be \label{bound:dotu}
 \| \dot u(t) \|_{W^{1,\infty}(B(0,\delta))} \leq C \e^{K'}, \qquad \mbox{for $t \leq \e^h T(\e),$}
\ee
with notation
\be \label{def:K'} K':= \a (K - m) - (1 - \a)(1-h) d/2.
\ee
By condition \eqref{cond:K}, there holds $K' > K/2.$ 

\subsection{Uniform remainders} \label{sec:unif}

 The linear propagator in \eqref{eq:dotu} will undergo many transformations in this proof, through linear changes of variables corresponding to projections, localizations, conjugations, and so on. Every change of variable produces error terms. We will henceforth denote ${\bf R}_k,$ for $k \in \Z,$ any bounded family ${\bf R}_k(\e,t)$ in $S^{k},$ in the sense that 
  \be \label{R0} \sup_{\begin{smallmatrix} 0 < \e < \e_0 \\ 0 \leq t \leq \e^h T(\e) \end{smallmatrix}} \| {\bf R}_k(\e,t) \|_{k, r} < \infty,\ee
  for $r$ large enough, with notation $\|\cdot\|_{k,r}$ for symbols introduced in \eqref{notation:norm:eps} in Appendix \ref{sec:symbols}. In the case $k = 0,$ we say that a symbol belongs to ${\bf R}_0$ if either \eqref{R0} holds or 
  \be \label{R0'} \sup_{\begin{smallmatrix} 0 < \e < \e_0 \\ 0 \leq t \leq \e^h T(\e) \end{smallmatrix}} \sum_{|\a| \leq d + 1} \sup_{\xi \in \R^d} \big| \d_x^\a {\bf R}_k(\e,t) \big|_{L^1(\R^d_x)} < \infty.\ee
 By Proposition \ref{prop:action}, the corresponding operators $\op_\e({\bf R}_k)$ are bounded $H^k \to L^2:$ 
  \be \label{def:unif}
  \| {\bf R}_k w \|_{L^2} \lesssim \| w \|_{\e,k}, \qquad \| w \|_{\e,s} :=  \big\| (1 + | \e^{h} \xi|^2)^{s/2} \hat w(\xi) \big\|_{L^2(\R^d_\xi)},
 \ee
 uniformly in $0 < \e < \e_0$ and $0 \leq t \leq \e^h T(\e).$ We use above notation $\lesssim$ introduced in \eqref{notation:lesssim}. %

\subsection{Spatial localization and projection} \label{sec:locproj} The matrix-valued symbol $Q(x,\xi),$ introduced in Assumption \ref{ass:main}, is smooth, locally defined and invertible around $(x_0,\xi_0).$ As explained in Appendix \ref{sec:extension:symbols} we may extend smoothly $Q$ into a globally defined symbol of order zero, which is globally invertible, with $Q^{-1} \in S^0.$ We identify $Q$ with its extension in the following, and let %
\be \label{def:uflat}
 u^\flat(t,x) := \op_\e(Q_\e) (\theta \dot u ),
\ee 
 corresponding to a spatial localization followed by a micro-local change of basis. In \eqref{def:uflat}, the function $\theta = \theta(x)$ is the spatial truncation introduced in Section \ref{sec:init}, and we use notation
\be \label{def:Qe} Q_\e(t,x,\xi) := Q\big(t, \, x_0 + \e^{1-h} x, \, \xi \big).\ee
Here $\op_\e(\cdot)$ denotes a pseudo-differential operator in $\e^{h}$-semiclassical quantization, as in \eqref{semicl:0}.
Classical results on pseudo-differential calculus are gathered in Appendix \ref{sec:symbols}.  In particular, $\op_\e(Q_\e)$ maps $L^2$ to $L^2,$ uniformly in $\e,$ so that
\be \label{ap:flat}
 \| u^\flat \|_{L^2} \lesssim \|\theta \dot u\|_{L^2} \lesssim \| \dot u \|_{L^2(B(0,\delta))}.
\ee
We now deduce from the equation \eqref{eq:dotu} in $\dot u$ an equation in $u^\flat,$ via the change of unknown \eqref{def:uflat}. Here we note that the leading, first-order term in \eqref{eq:dotu} is 
$$ A\big(t, \, x_0 + \e^{1-h} x, \, \e^{h}\d_x\big) = \op_\e(i A_\e), \qquad A_\e := A(t, x_0 + \e^{1-h} x, \xi),$$
Thus the equation in $u^\flat$ is 
$$ \begin{aligned} \d_t u^\flat & 
 + \e^{-1} \op_\e(Q_\e) \op_\e(i A_\e) (\theta \dot u) + \op_\e(Q_\e) \big( \theta \dot B \dot u \big) - \op_\e\big( (\d_t Q)_\e \big) (\theta \dot u) \\ & = \op_\e(Q_\e)(\theta \dot F)  - \sum_{1 \leq j \leq d} \op_\e(Q_\e) \big( A_j(\phi_\e) \dot u \d_{x_j} \theta \big).%
 \end{aligned}
 $$
At this point the goal is to express the terms in $\theta \dot u$ above in the form of terms in $u^\flat,$ modulo small errors -- that is, to approximately invert \eqref{def:uflat}. This is done as follows.

 By composition of pseudo-differential operators with slow $x$-dependence (Proposition \ref{prop:composition:slow:x}), there holds
\be \label{invert:flat} \Id = \op_\e(Q_\e^{-1}) \op_\e(Q_\e) + \e \op_\e({\bf R}_{-1}),\ee
where ${\bf R}_{-1}$ is a uniform remainder in the sense of Section \ref{sec:unif}. 
With \eqref{invert:flat} we may thus express $\theta \dot u$ in terms of $u^\flat:$
\be \label{dotu:uflat} \theta \dot u = \op_\e(Q_\e^{-1}) u^\flat + \e \op_\e({\bf R}_{-1}) (\theta \dot u).\ee
Using inductively \eqref{invert:flat}, and composition of pseudo-differential operators (Proposition \ref{prop:composition}), we obtain
\be \label{dotu:uflat2}
 \theta \dot u = \op_\e({\bf R}_0) u^\flat + \e^n \op_\e({\bf R}_0) (\theta \dot u),
\ee
for $n$ as large as allowed by the regularity of $\phi.$ 

By \eqref{dotu:uflat}, the first-order term in the above equation in $u^\flat$ is
$$ \op_\e(Q_\e) \op_\e(i A_\e) (\theta \dot u) = \op_\e(Q_\e) \op_\e(i A_\e) \op_\e(Q^{-1}_\e) u^\flat + \e \op_\e(Q_\e) \op_\e(i A_\e) \op_\e({\bf R}_{-1}) (\theta \dot u),$$
implying, with Proposition \ref{prop:composition:slow:x},
$$
\op_\e(Q_\e) \op_\e(i A_\e) (\theta \dot u) 
= \op_\e(i Q_\e A_\e Q^{-1}_\e) u^\flat + \e \op_\e({\bf R}_0) (\theta \dot u).$$
Besides, with \eqref{dotu:uflat2}, we may write 
$$ \e^h \big( \op_\e(Q_\e) \big( \theta \dot B \dot u \big) - \op_\e\big( (\d_t Q)_\e \big) (\theta \dot u) \big) = \e^h \op_\e(B^\flat) u^\flat + \e^n \op_\e({\bf R}_0) (\theta \dot u),$$
where $B^\flat \in {\bf R}_0.$ From the above, the equation in $u^\flat$ appears as
\be \label{eq:uflat} 
\d_t u^\flat + \e^{-1} \op_\e(i Q_\e A_\e Q_\e^{-1}) u^\flat + \op_\e(B^\flat) u^\flat = F^\flat,
\ee
where 
\be \label{def:Fflat}
 B^\flat \in {\bf R}_0, \quad F^\flat := \op_\e(Q_\e)(\theta \dot F)  -  \sum_{1 \leq j \leq d} \op_\e(Q_\e) \big( A_j(\phi_\e) \dot u \d_{x_j} \theta \big) + \e^n \op_\e({\bf R}_0) (\theta \dot u).
 \ee

\subsection{Advected coordinates} \label{sec:ad}

   Let $M$ be the flow of $\op_\e(i \mu_\e)$ (or rather, as argued at the beginning of Section \ref{sec:locproj}, of $\op_\e(\tilde \mu_\e),$ where $\tilde \mu$ is a globally defined symbol extending $\mu$) 
   in the sense that %
  \begin{equation} \label{eq:M}
   \d_t M = \op_\e(i \mu_\e) M, \qquad M(\t;\t) \equiv \Id,
   \end{equation}
 where the symbol $\mu$ is introduced in Assumption \ref{ass:main}, and $\mu_\e$ is defined from $\mu$ by rescaling space as in \eqref{def:Qe}. Let $M^\star$ be the associated backwards flow, defined by 
 \be \label{eq:tildeM}
  \d_\t M^\star = - M^\star \op_\e(i \mu_\e), \qquad M^\star(\t;\t) \equiv \Id. \ee
By hyperbolicity (reality and regularity of $\mu,$ and Proposition \ref{prop:action}), both $M$ and $M^\star$ map $L^2$ to $L^2,$ uniformly in $\e,t,$ for $t \leq \e^h T(\e).$ Egorov's lemma (see for instance Theorem 4.7.8 in \cite{Le}, or Theorem 8.1 in \cite{Tay}) states that
 \be \label{R:M}
   M^\star M = \Id + \e \op_\e({\bf R}_{-1}), \qquad M M^\star = \Id + \e \op_\e({\bf R}_{-1}),
   \ee
where we recall that ${\bf R}_{-1}$ is a generic notation for bounded symbols of order $-1$ (see Section \ref{sec:unif}); in other words the equalities in \eqref{R:M} really mean that both $M^\star M - \Id$ and $M M^\star - \Id$ belongs to the class of operators of the form $\e \op_\e({\bf R}_{-1}).$ By Egorov's lemma, given $a \in S^{m},$ there also holds 
 \be \label{egorov}
  M^\star \op_\e(a_\e) M  = \op_\e(a_{\e \star}) + \e \op_\e({\bf R}_{m-1}),
 \ee
 where $a_{\e\star}$ denotes the symbol $a_\e$ evaluated along the bicharacteristic flow, in the following sense: for any symbol $b$\footnote{Except for symbol $A,$ as seen on definition on $A_\star$ in \eqref{def:Astar}.} we denote
 \be \label{def:astar}
 b_\star(t,x,\xi) := b(t, x_\star(t,x,\xi), \xi_\star(t,x,\xi)),
 \ee
  where the $(x_\star,\xi_\star)$ is the bicharacteristic flow of $\mu_\e.$   
What's more, the remainder ${\bf R}_{m-1}$ above has expansion
\be \label{R:M:bis}
 \op_\e({\bf R}_{m-1}) = \op_\e(a_{\star1}) + \e \op_\e(a_{\star2}) + \dots + \e^{n} \op_\e(a_{\star n}) + \e^{n+1} \op_\e({\bf R}_{m-n-1}),
\ee
for $n$ as large as allowed by the regularity of $\mu, a$ and $\phi,$ where $a_{\star i} \in S^{m-i}$ has support included in the support of $a_\star.$ Identities \eqref{egorov} and \eqref{R:M:bis} also hold if $M$ and $M^\star$ are interchanged, with backwards bicharacteristics replacing forward bicharacteristics.  By \eqref{R:M}, there holds
 $$ M^\star \op_\e({\bf R}_0) = M^\star \op_\e({\bf R}_0) (M M^\star + \e \op_\e({\bf R}_{-1})),$$
 implying, by \eqref{egorov}:
  $$ M^\star \op_\e({\bf R}_0) = \op_\e({\bf R}_{0\star}) M^\star + \e \big( M^\star \op_\e({\bf R}_{-1}) + \op_\e({\bf R}_{-1}) M^\star),$$ 
and reasoning inductively we arrive at
\be \label{com:M}
 M^\star \op_\e({\bf R}_0)  = \op_\e({\bf R}_0) M^\star + \e^n \op_\e({\bf R}_0),
 \ee
 where $n$ is as large as allowed by the regularity of $\phi.$ The advected variable is defined as 
 \begin{equation} \label{def:tildeu} 
 u_\star := M^\star(0;t) u^\flat.
 \end{equation}
There holds
$$ \d_t u_\star = M^\star\big( \d_t - \op_\e(i \mu_\e)\big) u^\flat,$$
and $\d_t u^\flat$ is given by equation \eqref{eq:uflat} in $u^\flat.$ 
Using \eqref{R:M}, we find 
$$ \begin{aligned} M^\star \op_\e(i Q_\e (A_\e - \mu_\e) Q_\e^{-1}) u^\flat & = \op_\e\big((i Q_\e (A_\e - \mu_\e) Q_\e^{-1})_\star\big) u_\star \\ & + \e \op_\e({\bf R}_0) u_\star - \e M^\star \op_\e(Q_\e A_\e Q_\e^{-1}) \op_\e({\bf R}_{-1}) u^\flat,\end{aligned}$$
with notation \eqref{def:astar}. 
 Thus, with \eqref{com:M}, 
 \be \label{for:star1}
 \begin{aligned} M^\star \op_\e(i Q_\e (A_\e - \mu_\e) Q_\e^{-1}) u^\flat & = \op_\e\big((i Q_\e (A_\e - \mu_\e) Q_\e^{-1})_\star\big) u_\star \\ & + \e \op_\e({\bf R}_0) u_\star + \e^n \op_\e({\bf R}_0) u^\flat.
\end{aligned} \ee
Besides, in view of \eqref{com:M}, the order-zero term $B^\flat$ in \eqref{eq:uflat} contributes to the equation in $u_\star$ the terms
 $$ M^\star(0;t) \op_\e(B^\flat) u^\flat = \op_\e(B_\star) u_\star + \e^n \op_\e({\bf R}_0) u^\flat, \qquad B_\star \in {\bf R}_0.$$
The equation in $u_\star$ thus appears as
\be \label{eq:ustar} 
 \d_t u_\star + \e^{-1} \op_\e\big((i Q_\e (A_\e - \mu_\e) Q_\e^{-1})_\star\big) u_\star + \op_\e(B_\star) u_\star =  F_\star ,
\ee
where 
\be \label{def:tildeF} 
 B_\star \in {\bf R}_0, \qquad F_\star := M^\star F^\flat + \e^n \op_\e({\bf R}_0) u^\flat.
\ee
The symbol $(Q_\e (A_\e - \mu_\e) Q_\e^{-1})_\star$ is symbol $Q_\e (A_\e - \mu_\e) Q_\e^{-1}$ evaluated along the bicharacteristics of $\mu_\e,$ as defined in \eqref{def:astar}.

\subsection{Frequency, space, and time truncation functions} \label{sec:trunc0}

We introduce frequency cut-offs $\chi_0^\flat, \tilde \chi_0, \chi_0,$ spatial cut-offs $\theta_0^\flat, \tilde \theta_0, \theta_0,$ and temporal cut-offs $\psi_0^\flat, \tilde \psi_0,$ $\psi_0.$ All are smooth and take values in $[0,1].$ Given two cut-offs $\psi_1$ and $\psi_2,$
\be \label{def:prec}
 \psi_1 \prec \psi_2 \quad \mbox{means} \quad (1 - \psi_2) \psi_1 \equiv 0. 
\ee
Equivalently, $\psi_1 \prec \psi_2$ when $\psi_2 \equiv 1$ on the support of $\psi_1.$

The supports of the frequency cut-offs $\chi_0^\flat, \tilde \chi_0, \chi_0$ are all assumed to be included in the ball $\{ |\xi| \leq \delta \}.$ All three are identically equal to 1 on a neighborhood of $\xi = 0.$ There holds $\chi_0^\flat \prec \tilde \chi_0 \prec \chi_0.$ 

The supports of the spatial cut-offs $\theta_0^\flat, \tilde \theta_0, \theta_0$ are all assumed to be included in $\{ |x| \leq \delta \}.$  All three are identically equal to 1 on a neighborhood of $x = 0.$ There holds $\theta_0^\flat \prec \tilde \theta_0 \prec \theta_0.$ We assume in addition $\theta_0 \prec \theta,$ where $\theta$ is the spatial cut-off of Section \ref{sec:locproj}. 

The temporal cut-offs $\psi_0^\flat, \tilde \psi_0,$ $\psi_0$ are nondecreasing, supported in $\{ t \geq - \delta\},$ and identically equal to one in a neighborhood of $\{t \geq 0\}.$ In particular, $\tilde \psi_0 \equiv 1$ on $\{ t \geq - \delta/3\}.$ There holds $\psi_0^\flat \prec \tilde \psi_0 \prec \psi_0.$ The truncation $\tilde \psi_0$ is pictured on Figure 6. 
 
 \begin{figure}\label{fig:psi}
\scalebox{.8}{
\beginpgfgraphicnamed{figure-psi}
 \begin{tikzpicture}

 \begin{scope}[>=stealth]
 
 \draw[line width=.5pt][->] (-3,-1) -- (3,-1);
 
 \draw[line width=.5pt][->] (2.5,-1.5) -- (2.5,3);
 
  \draw (2.5,3) -- (2.5,3) node[anchor=west] {$\tilde \psi_0$} ; 
  
  \draw (2.7,-1) -- (2.7,-1) node[anchor=north] {$0$} ;
   
  \filldraw (1.4,2) circle (.1cm) ; 
  
  \filldraw (-1.5,-1) circle (.1cm) ; 
  
   \draw (2.7,1.9) -- (2.7,1.9) node[anchor=north] {$1$} ;
   
   \draw (-1.2,-1.1) -- (-1.2,-1.1) node[anchor=north] {$-\delta$} ; 
   
   \draw (1.4,-1.1) -- (1.4,-1.1) node[anchor=north] {$-\delta/3$} ; 
  
 \draw[very thick] (0,.5) .. controls (1,2) .. (2,2) ; 
 
 \draw[very thick] (2,2) -- (3,2) ; 
 
 \draw[very thick] (-2,-1) .. controls (-1,-1) .. (0,.5) ; 
 
 \draw[very thick] (-2,-1) -- (-3,-1) ; 
 
 \filldraw (1.4,-1) circle (.1cm) ; 
 
 \draw[dashed] (1.4,-1) -- (1.4,2) ; 

  \end{scope}
  
 \end{tikzpicture}
\endpgfgraphicnamed
} 
\caption{The truncation function $\tilde \psi_0.$} 
\end{figure}
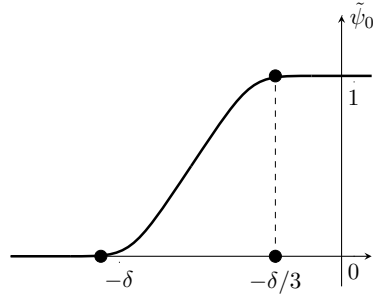

Associated with these cut-offs, define 
\be \label{def:chi}
 \chi(\e,t,x,\xi) := \chi_0\left(\frac{\xi - \xi_0}{\e^{\zeta}}\right) \theta_0(x) \psi_0\big(t - t_\star(\e,x,\xi)\big),
\ee
and define similarly $\chi^\flat$ and $\tilde \chi$ in terms of $\chi_0^\flat, \theta_0^\flat, \psi_0^\flat$ and $\tilde \chi_0, \tilde \theta_0, \tilde \psi_0$ respectively. In \eqref{def:chi}, the transition function $t_\star$ is defined in \eqref{def:tstar}, which we reproduce here:
$$ t_\star(\e,x,\xi) = \e^{-h} \theta_\star(\e^{1-h} x, \xi), \quad \mbox{with} \,\,\,\theta_\star \geq 0, \,\,\, \theta_\star(0,\xi_0) = 0, \,\,\, \nabla_{x,\xi} \theta_\star(0,\xi_0) = 0.$$
Recall that the value of $\zeta$ is fixed in Assumption \ref{ass:main}, depending on $\ell:$ there holds $\zeta = 0$ if $\ell = 0$ or $\ell = 1,$ and $\zeta =1/3$ if $\ell = 1/2.$ 

\begin{lem} \label{lem:chi} The support of $\chi$ is a neighborhood of the elliptic domain ${\mathcal D}$ \eqref{def:mathcalD}, and there holds
 $$ |\d_x^\a \d_\xi^\b \chi| \lesssim \e^{-|\b| \zeta},$$
where $\zeta$ is introduced in Section {\rm \ref{sec:ass:param}}, such that, in particular, $\zeta < h.$ The same holds of course for the other truncations $\chi^\flat$ and $\tilde \chi,$ which satisfy $\chi^\flat \prec \tilde \chi \prec \chi.$ 
\end{lem}

\begin{proof} For $(t,x,\xi)$ to belong to the support of $\chi,$ there needs to hold simultaneously $|\xi - \xi_0| \leq \e^{\zeta},$ $|x| \leq \delta,$ and $t_\star(\e,x,\xi) - \delta \leq t.$ This defines a neighborhood of ${\mathcal D}$ (precisely, of the projection of ${\mathcal D}$ onto the $(t,x,\xi)$ domain), as defined in \eqref{def:mathcalD}.

We may now assume $t_\star$ to be not identically zero, otherwise $\chi$ is not stiff. Then (see Section \ref{sec:ass:param}) $\ell = 1/2,$ $h = 2/3,$ $\zeta = 1/3.$ By the Fa\'a di Bruno formula
 $$ \d_\xi^\b \Big\{ \psi_0(t - t_\star(\e,x,\xi)) \Big\} =  \sum_{\begin{smallmatrix} 1 \leq k \leq |\b| \\ \b_1 + \dots + \b_k = \b \end{smallmatrix} } C_{(\b_k)} \psi_0^{(k)}(t - t_\star)  \prod_{1 \leq j \leq k} (\e^{-h} \d_\xi^{\b_j} \theta_\star) ,$$
where $C_{(\b_k)}$ are positive constants. We note that $\d_\xi^{\a} \theta_\star(\e,x,\xi) = O(\e^{\zeta})$ if $|\a| = 1$ and $\xi$ belongs to the support of $\chi_0,$ by assumption on $\theta_\star,$ while $\d_\xi^\a \theta_\star = O(1)$ if $|\a| \geq 2.$

Consider the case of a decomposition of $\b$ in a sum of $\b_j$ of length one. Based on the above formula and the bound on $\nabla_\xi \theta_\star,$ the corresponding bound is $\e^{-|\b|(h - \zeta)}.$ 

If $\b$ is decomposed into $\b_1 + \b_2  +\dots + \b_k,$ with $|\b_1| = 2$ and $|\b_j| = 1$ for $j \geq 2,$ then $k = |\b| - 1.$ The corresponding bound is $\e^{- h + (|\b| - 2)(h - \zeta)} \leq \e^{-|\b|(h - \zeta)}$ as soon as $\zeta \leq h/2,$ which holds true. It is now easy to verify that the decomposition of $|\b|$ into sums of multi-indices of length one corresponds to the worst possible loss in powers of $\e.$ 

We turn to $x$-derivatives of $\psi_0(t - t_\star).$ By assumption, $\theta_\star$ is a function of $\e^{2(1-h)} (x,x)$ and $\e^{1-h} (x,\xi).$ Thus $x$-derivatives bring in either powers of $\e^{-h + 2(1-h)} = 1,$ since $h = 2/3,$ or powers of $\e^{-h + 1- h + \zeta} = 1,$ since $\zeta = h/2 = 1/3.$ 

Thus $|\d_x^\a \d_\xi^\b \psi_0(t - t_\star)| \lesssim \e^{- |\b|( h - \zeta)},$ if $\ell = 1/2.$ Considering finally the full truncation function $\chi,$ we observe that the term in $\chi_0$ contributes the exact same loss per $\xi$-derivative, and the spatial truncation $\theta_0$ contributes no loss. \end{proof}

\begin{cor} \label{cor:chi} The operator $\op_\e(\chi)$ maps $L^2(\R^d)$ to $L^2(\R^d),$ uniformly in $\e,$ and so do $\op_\e(\chi^\flat)$ and $\op_\e(\tilde \chi).$ 
\end{cor}

\begin{proof} Since $\chi$ is compactly supported in $x,$ we may use pointwise bounds for $\d_x^\a \chi$ and bound \eqref{bd:action:H} in Proposition \ref{prop:action}. The result then follows from Lemma \ref{lem:chi}.
\end{proof}

\subsection{Localization in the elliptic zone and rescaling in time} \label{sec:trunc}

We define 
\be \label{def:vpsi}
 v := \op_\e(\tilde \chi(t)) \big( u_\star(\e^h t) \big),
\ee
meaning that we first rescale time in $u_\star$ and then apply $\op_\e(\tilde \chi)$ evaluated at $t,$ where $\tilde \chi$ is defined just below \eqref{def:chi}. We now derive an equation in $v,$ based on equation \eqref{eq:ustar} in $u_\star.$ 

Consider first the leading, first-order term in \eqref{eq:ustar}. When evaluated at $\e^h t,$ its symbol is precisely $A_\star$ \eqref{def:Astar} the rescaled and advected symbol for which Assumption \ref{ass:main} holds:
$$ \big((i Q_\e (A_\e - \mu_\e) Q_\e^{-1})_\star\big)(\e^h t) = A_\star(\e,t,x,\xi).$$  
Thus
$$ \op_\e(\tilde \chi) \Big( \op_\e( \big((i Q_\e (A_\e - \mu_\e) Q_\e^{-1})_\star\big) u_\star \, \Big)(\e^h t) = \op_\e(\tilde \chi) \op_\e(i A_\star) (u_{\star}(\e^h t)).$$
Similarly, denoting $B := B_\star(\e^h t),$ where $B_\star$ is the order-zero correction to the leading symbol which appears in equation \eqref{eq:ustar}, there holds
 $$ \op_\e(\tilde \chi) \big( \op_\e(B_\star) u_\star \, \big)(\e^h t) = \op_\e(\tilde \chi)\op_\e(B) (u_\star(\e^h t)).$$
We now introduce a commutator:
$$ \op_\e(\tilde \chi) \op_\e(i A_\star + \e B) (u_\star(\e^h t)) = \op_\e(i A_\star + \e B) v + \Gamma (u_\star(\e^h t)),$$
where
\be \label{def:tildeG} \tilde \G := [\op_\e(\tilde \chi), \op_\e(i A_\star + \e B)].\ee
By definition of $v,$ 
$$ \d_t v = \op_\e(\d_t \tilde \chi) (u_\star(\e^h t)) + \e^h \op_\e(\tilde \chi)(( \d_t u_\star) (\e^h t)).$$
Together with equation \eqref{eq:ustar} and the above, this implies 
$$ \d_t v + \e^{h-1} \op_\e(i A_\star + \e B) v = \op_\e(\d_t \tilde \chi) (u_\star (\e^h t)) - \e^{h-1} \tilde \Gamma (u_\star(\e^h t)) + \e^h \op_\e(\tilde \chi) F_\star.$$
We will handle the right-hand side as a remainder. The following Lemma shows that we may introduce a truncation function in the above principal symbol.

\begin{lem} \label{lem:psi:remainder} For any bounded family $P(\e,t) \in S^1,$ there holds, for $\chi$ defined in \eqref{def:chi} and $v$ defined in \eqref{def:vpsi}:
\be \label{to:prove:psi}  \op_\e(P) v = \op_\e(\chi P) v + [\op_\e(P), \op_\e(\chi)] v + \e^n \big(\op_\e({\bf R}_0) u_\star\big)(\e^h t),
\ee 
where ${\bf R}_0$ is a uniform remainder in the sense of Section {\rm \ref{sec:unif}},
\end{lem}

\begin{proof} By definition of $\tilde \chi$ and $\chi,$ and Proposition \ref{prop:composition}, there holds 
$$  \op_\e(\tilde \chi) = \op_\e(\chi \tilde \chi) = \op_\e(\chi) \op_\e(\tilde \chi) + \e^{n'h } \op_\e(R_{n'}(\chi, \tilde \chi)),$$
and the remainder satisfies
 $$ \| \op_\e(R_{n'}(\chi, \tilde \chi)) \|_{L^2 \to \| \cdot \|_{\e,-n'}} \lesssim \| \d_\xi^{n'} \chi \|_{0,C(d)} \| \d_x^{n'} \tilde \chi \|_{0,C(d)}.$$
 We use here norms $\| \cdot \|_{m,r}$ for pseudo-differential symbols of order $m,$ as defined in \eqref{notation:norm:eps}.
By Lemma \ref{lem:chi}, there holds 
 $$ \| \d_\xi^{n'} \chi \|_{0,C(d)} \lesssim \e^{- (n' + C(d))\zeta} \quad \and  \| \d_x^{n'} \tilde \chi \|_{0,C(d)} \leq \e^{-C(d) \zeta}.$$
Since $\zeta < h,$ there holds $\e^{n' h - (n' + 2 C(d)) (h - \zeta)} \leq \e^{n},$ for any $n,$ if $n'$ is chosen large enough. Thus 
 \be \label{chi:id} \op_\e(\tilde \chi) = \op_\e(\chi) \op_\e(\tilde \chi) + \e^{n} \op_\e({\bf R}_{0}),\ee
where ${\bf R}_0$ is bounded for $t \leq T(\e).$ This implies 
$$ \op_\e(P) v = \op_\e(P) \op_\e(\tilde \chi) u_\star(\e^h t) = \op_\e(P) \op_\e(\chi) v + \e^n \op_\e({\bf R}_{0}) (u_\star(\e^h t)).$$
Now 
$$ \op_\e(P) \op_\e(\chi) = \op_\e(\chi P) + [\op_\e(P), \op_\e(\chi)],$$
and \eqref{to:prove:psi} is proved, with a symbol ${\bf R}_0$ which is a uniform remainder in the sense of Section \ref{sec:unif}, meaning that we rescale in time the remainder which appears in \eqref{chi:id}.  
\end{proof}

Applying Lemma \ref{lem:psi:remainder} to $P = iA_\star + \e B,$ we derive the final form of the equation satisfied by $v:$ 
  \begin{equation} \label{eq:v} 
   \d_t v +  \e^{h-1} \op_\e(\chi(i A_\star + \e B)) v = \e^{h} g,
  \end{equation}
  where the source term $g$ is defined in terms of the remainder $F_\star$ from \eqref{eq:ustar}-\eqref{def:tildeF}:
\be \label{def:g}
g =  \e^{-1} (\Gamma v - \tilde \Gamma (u_\star(\e^h t))) + \e^{-h} \op_\e(\d_t \tilde \chi) (u_\star (\e^h t))  +  \op_\e(\tilde \chi) (F_\star(\e^h t)) + \e^n \big(\op_\e({\bf R}_0) u_\star\big)(\e^h t),
\ee
with $\G$ is defined just like $\tilde \G$ \eqref{def:tildeG}, but with $\chi$ in place of $\tilde \chi.$ 

The derivation of \eqref{eq:v}-\eqref{def:g} ends the first step of the proof of Theorem \ref{th:main}. Our goal is now to show a growth in time for the solution to \eqref{eq:v} over the  interval $[0, T(\e)],$ where $T(\e) = \big(T_\star |\ln \e|\big)^{1/(1 + \ell)}.$ In this view, we will first derive an integral representation formula for $v.$ 

\subsection{An integral representation formula}\label{integral}
 At this point we use the theory developped in Appendix \ref{sec:Duhamel}. Theorem \ref{duh1} gives the integral representation formula for the solution $v$ to \eqref{eq:v} issued from $v(0):$  
 \begin{equation}\label{duhamel}
 \begin{aligned}
 v & = \op_\e(\S(0;t)) ) v(0) + \e^{h} \int_0^t \op_\e(\S(\t;t)) (\Id + \e \op_\e({\bf R}_0)) \big(g(\t) + \e{\bf R}_0 v(0) \big)  d\t,\end{aligned}\end{equation}
 where ${\bf R}_0$ are uniform remainders, as defined in Section \ref{sec:unif},
 and the approximate solution operator $\op_\e(\S(s;t))$ is defined 
 by
 \begin{equation} \label{repS}
  \S = \sum_{0 \leq q \leq q_0} \e^{hq} S_{q}, %
   \end{equation}
where $q_0$ is large enough\footnote{Depending on $\zeta,$ the final observation time $T_\star$ \eqref{max-time}, and the growth function $\g;$ see Appendix \ref{sec:Duhamel} and in particular the proof of Lemma \ref{lem:duh-remainder}.}. In \eqref{repS}, the leading term $S_{0}$ is defined for $0 \leq \t \leq t \leq T(\e)$ and all $(x,\xi) \in \R^{d} \times \R^d$ by
  \begin{equation} \label{eq:flow} \d_t S_{0} + \e^{h-1} \chi (i A_\star + \e B)(t,x,\xi) S_{0} = 0, \qquad S_{0}(\t;\t)  = \mbox{Id},\end{equation}
and the $S_{q},$ for $q \geq 1,$ are correctors,
 defined inductively as the solutions of 
 \be \label{Sq} \begin{aligned} 
  \d_t S_q & + \e^{h-1} \chi (i A_\star + \e B)(t,x,\xi)  S_q  + \e^{h-1} \sum_{\begin{smallmatrix} q_1 + q_2 = q \\ 0 < q_1 \end{smallmatrix}} (\chi (i A_\star + \e B)) \sharp_{q_1} S_{q_2}, \quad S_q(\t;\t) \equiv 0. \end{aligned}
  \ee
 with notation $\s_1 \sharp_n \s_2 := (- i)^n (n!)^{-1} \sum_{|\a| = n} \d_\xi^\a \s_1 \d_x^\a \s_2.$ From \eqref{eq:flow} and \eqref{Sq} we deduce the representation, for $q \geq 1,$ 
 \begin{equation} \label{Sq:rep} S_{q}(\t;t) = - \e^{h-1} \int_{\t}^t S_{0}(\t';t) \sum_{\begin{smallmatrix} q_1 + q_2 = q \\ 0 < q_1 \end{smallmatrix}} (\chi(i A_\star + \e B))(\t',x,\xi) \sharp_{q_1} S_{q_2}(\t;\t')\, d\t'.
 \end{equation}
In order to be able to exploit representation \eqref{duhamel}, we need to check that Assumption \ref{ass:BS}, under which Theorem \ref{duh1} holds, is satisfied. This is the object of the forthcoming Section. 

\subsection{Bound on the solution operator} \label{sec:bd-sol-op}

Recall that $S,$ the symbolic flow of $i \e^{h-1} A_\star,$ is defined in \eqref{tildefund}, and is assumed to satisfy the bounds of Assumption \ref{ass:main}. The upper bound \eqref{ass:up} in Assumption \ref{ass:main} is assumed to hold for $S$ in domain ${\mathcal D}$ defined in \eqref{def:mathcalD}: 
$$ {\mathcal D} := \big\{ (\t;t,x,\xi), \quad t_\star(x,\xi) \leq \t \leq t \leq T(\e),  \quad |x| \leq \delta, \quad |\xi - \xi_0| \leq \delta \e^\zeta\big\},$$
where the transition time $t_\star$ is defined in \eqref{def:tstar} and the final observation time $T(\e)$ is defined in \eqref{max-time}. 

The goal in this Section is to show that the symbolic flow $S_0,$ which is defined as the solution of \eqref{eq:flow}, and the correctors $S_q$ defined as the solutions to \eqref{Sq}, satisfy the bounds of Assumption \ref{ass:BS}. This will allow to use the representation Theorem \ref{duh1}, and will justify representation \ref{duhamel}. This will also give a bound for the norm of the approximate solution operator $\op_\e(\Sigma)$ defined in \eqref{repS}.

We are looking for bounds for the correctors $S_q,$ and their derivatives. Consider representation \eqref{Sq:rep}. Disregarding $(x,\xi)$-derivatives, we see on \eqref{Sq:rep} that $S_q$ appears as a time integral of a product $S_0 (\chi (i A_\star + \e B)) S_{q_2},$ with $q_2 < q.$ We may recursively use \eqref{Sq:rep} at rank $q_2,$ and by induction $S_q(\t;t)$ appears as a time-integral of form 
 \be \label{rep:simple} \begin{aligned} \e^{q(h-1)} \int_{\t \leq \t_1 \leq \dots \leq \t_q \leq t} S_0(\t_1;t) (\chi(i A_\star & + \e B))(\t_1)  S_0(\t_2;\t_1) (\chi(i A_\star + \e B))(\t_2) \cdots \\ &  \cdots S_0(\t_{n}; \t_{n-1}) (\chi (i A_\star + \e B))(\t_n) S_0(\t;\t_n) d\t_1\dots d\t_n, \end{aligned}\ee
 in which there are $q$ occurences of $i A_\star + \e B$ and $q+1$ occurences of $S_0.$ Note again that in \eqref{rep:simple} we overlook $(x,\xi)$-derivatives. From there, it appears that we need to 
 \begin{itemize}
 \item derive bounds on $S_0$ and its $(x,\xi)$-derivatives; these will be deduced from the bounds of $S$ postulated in Assumption \ref{ass:main};
 \item derive bounds for products of $S_0$ with $\chi( i A_\star +  \e B);$ here the block structure assumption \ref{block:1}-\eqref{block:2} from Assumption \ref{ass:main} will come in. 
 \end{itemize}

\subsubsection{Product bounds for \texorpdfstring{$S$}{S}} \label{sec:S:first} In a first step, we prove bounds for products of symbols involving the symbolic flow $S$ \eqref{tildefund} and the rescaled and advected principal symbol $A_\star$ \eqref{def:Astar}. We denote for $\a,\b \in \N^{2d}$ and $0 \leq \t \leq t:$ 
\be \label{def:Sab} S_{\a,\b}(\t;t) := \e^{h-1} S(\t;t) \d_x^{\a} \d_\xi^\b A_\star(\t),\ee
and products, for $(\a_i,\b_i) \in \N^{2d}$ and $0 \leq \t \leq \t_n \leq \t_{n-1} \leq \t_1 \leq t:$ 
\be \label{def:bfS} {\bf S}_n(\t,\t_1,\dots,\t_n;t) := S_{\a_1,\b_1}(\t_1;t) S_{\a_2,\b_2}(\t_2;\t_1) \cdots S_{\a_n,\b_n}(\t_n; \t_{n-1}) S(\t;\t_n).\ee

\begin{lem} \label{lem:s00} Under Assumption {\rm \ref{ass:main}}, there holds
 $$ | {\bf S}_n(\t,\t_1,\dots,\t_n;t,x,\xi) | \lesssim \left(\begin{array}{cc} 1 & \e^{-\zeta} \\ \e^{\zeta} & 1 \end{array}\right) {\bf e}_{\g^+}(\t;t,x,\xi),$$
 for all $n \geq 1,$ all $\a_i,\b_i\in \N^{2d},$ all $0 \leq \t \leq \t_n \leq \dots \leq \t_1 \leq t$ with $(\t;t,x,\xi) \in {\mathcal D},$ uniformly in $\e.$ By $\lesssim$ we mean here entry-wise inequalities modulo constants, for each block of ${\bf S}_n,$ as described below \eqref{notation:lesssim}. 
\end{lem} 
\begin{proof} If all blocks of $A_\star$ satisfy \eqref{block:1}, then $\zeta = 0,$ and the stated bound simply follows from the multiplicative nature of the growth function, namely identity
\be \label{mult:e} {\bf e}_{\g^+}(\t_1;t) {\bf e}_{\g^+}(\t;\t_1) = {\bf e}_{\g^+}(\t;t).\ee
 Suppose then that a block $A_{\star j}$ satisfies \eqref{block:2}, and consider the associated component $S_j$ of the symbolic flow \eqref{tildefund}. Bound \eqref{ass:up} from Assumption \ref{ass:main} and the cancellation observed in product
 \be \label{cancel0} \left(\begin{array}{cc} 1 & \e^{h-1} \\ \e^{1-h} & 1 \end{array}\right) \left(\begin{array}{cc} 1 & \e^{h-1} \\ \e^{1-h} & 1 \end{array}\right) = 2 \left(\begin{array}{cc} 1 & \e^{h-1} \\ \e^{1-h} & 1 \end{array}\right),\ee
imply, omitting the index $j,$ 
 \be \label{cancel} |S(\t_1;t) \d_x^\a \d_\xi A_\star(\t_1) S(\t;\t_1)| \lesssim \left(\begin{array}{cc} 1 & \e^{h-1} \\ \e^{1-h} & 1 \end{array}\right) {\bf e}_{\g^+}(\t;t).\ee
  The result follows from \eqref{mult:e}-\eqref{cancel0}-\eqref{cancel} by a straightforward induction.
\end{proof}

\subsubsection{Product bounds for \texorpdfstring{$\d_x^\a S$}{dxa}}
Next we show that spatial derivatives of $S,$ and products involving $\d_x^\a S,$ satisfy the upper bound \eqref{ass:up} from Assumption \ref{ass:main}. We denote for $\a,$ $\b,$ $\b'\in \N^{d},$ $0 \leq \t \leq t:$ 
\be \label{def:tildeSab} \tilde S_{\a,\b}(\t;t) := \e^{h-1} \d_x^{\a} S(\t;t) \d_x^{\b} \d_\xi^{\b'} A_\star(\t),\ee
and products, for $\a_i \in \N^d,$ $\b_i \in \N^{2d}$ and $0 \leq \t \leq \t_n \leq \t_{n-1} \leq \t_1 \leq t:$ 
\be \label{def:tildebfS}
 \tilde {\bf S}_n(\t,\t_1,\dots,\t_n;t) := \tilde S_{\a_1,\b_1}(\t_1;t) \tilde S_{\a_2,\b_2}(\t_2;\t_1) \cdots \tilde S_{\a_n,\b_n}(\t_n; \t_{n-1}) \d_x^\a S(\t;\t_n).
\ee

\begin{lem} \label{lem:s0} Under Assumption {\rm \ref{ass:main}}, there holds the bound in domain ${\mathcal D}:$  %
 $$ |\d_x^\a S| + |\tilde {\bf S}_n| \lesssim \left(\begin{array}{cc} 1 & \e^{-\zeta} \\ \e^{\zeta} & 1 \end{array}\right) {\bf e}_{\g^+},$$
for each block of the block-diagonal matrices $S$ and ${\bf S}_n.$ The precise meaning of $\lesssim$ is described below \eqref{notation:lesssim}. 
\end{lem}

Above, and often below, the time and space-frequency arguments are omitted. In particular, the ``interior" temporal arguments of $\tilde {\bf S}_n,$ namely $\t_n, \dots, \t_1,$ are omitted. It is implicit that the $\t_i$ are constrained only by $\t \leq \t_n \leq \t_{n-1} \leq \dots \leq \t_1 \leq t,$ and that the indices $\a_i, \b_i$ and $n \geq 1$ are arbitrary.

\begin{proof} We first prove by induction on $|\a|$ that $\d_x^\a S$ enjoys the representation
 \be \label{rep:Sa} \d_x^\a S = \sum_{n \leq |\a|} \int {\bf S}_n,\ee
 where ${\bf S}_n$ is defined in \eqref{def:bfS}. By \eqref{rep:Sa}, we mean precisely
 \be \label{rep:Saa} \d_x^\a S(\t;t) = \sum_{1 \leq n \leq |\a|} C_n \int_{\t \leq \t_1 \leq \dots \leq \t_n \leq t} {\bf S}_n(\t,\t_1,\dots,\t_n;t) \, d\t_1 \dots d\t_n,\ee
 with constants $C_n$ independent of $(\e,t,x,\xi).$ In the following, whenever products, such as ${\bf S}_n,$ are integrated in time, the integration variables are the ``interior" variables, as described just above this proof. In order to prove \eqref{rep:Sa} for $|\a| = 1,$ we apply $\d_x^\a$ to the equation \eqref{tildefund} in $S:$ 
 $$ \d_t \d_x^\a S + i \e^{h-1} A_\star \d_x^\a S  = - i \e^{h-1} (\d_x^\a A_\star) S.$$
This implies the representation
 $$ \d_x^\a S(\t;t) = - i \e^{h-1} \int_{\t}^t S(\t';t) \d_x^\a A_\star(\t') S(\t;\t') d\t', \quad |\a| = 1,$$
 which takes the form \eqref{rep:Sa}. For greater values of $|\a|,$ we have similarly
 $$ \d_x^\a S = - i \e^{h-1} \sum_{\begin{smallmatrix} \a_1 + \a_2 = \a \\ 0 < |\a_1| \end{smallmatrix}} \int_\t^t S \d_x^{\a_1} A_\star \d_x^{\a_2} S,$$
 and the induction step is straightforward. From \eqref{rep:Sa} the bound on $\d_x^\a S$ follows by the bound on ${\bf S}_n$ in Lemma \ref{lem:s00} and the multiplicative nature of ${\bf e}_{\g^+}.$ Time integrals only contribute powers of $|\ln \e|,$ which are invisible in $\lesssim$ estimates. Finally, from \eqref{rep:Sa} we deduce (using the same notational convention as in \eqref{rep:Sa})
 \be \label{tildebfS:bfS} \tilde {\bf S}_n = \sum_{n' \leq |\a_1| + \dots + |\a_n| + n + |\a|} \int {\bf S}_{n'},\ee
and Lemma \ref{lem:s00} applies again. %
  \end{proof}

\subsubsection{Bounds for the symbolic flow \texorpdfstring{$S_0$}{S}} 
We bound here $S_0,$ the solution to \eqref{eq:flow}. While $S$ is the flow of $i \e^{h-1} A_\star,$ the symbol $S_0$ is the flow of $\e^{h-1} \chi (i A_\star + \e B),$ where $\chi$ is a stiff truncation and $B$ is a bounded symbol of order zero. We prove here that the upper bound \eqref{ass:up} in Assumption \ref{ass:main} is stable under the perturbations induced by $\chi$ and $B,$ in the sense that $S_0$ and its spatial derivatives satisfy the same upper bound as $S.$ 

In a first step, we consider the solution $S_\chi$ to
\be \label{def:Spsi}
  \d_t S_\chi + \e^{h-1} \chi A_\star S_\chi = 0, \qquad S_\chi(\t;\t) \equiv \Id.
\ee
Associated with $S_\chi,$ define products ${\bf S}_{\chi,n}$ involving $\d_x^\a S_\chi$ just like $\tilde {\bf S}_n$ was defined as a product involving $\d_x^\a S,$ but with $\chi A_\star$ in place of $A_\star,$ explicitly:
\be \label{def:tildeSpsiab} S_{\chi,\a,\b}(\t;t) := \e^{h-1} \d_x^{\a_1} S_\chi(\t;t) \d_x^{\a_2} \d_\xi^{\b} (\chi A_\star(\t)), \quad \a = (\a_1, \a_2) \in \N^{2d}, \, \b \in \N^d,\ee
and products, for $\a_i \in \N^{2d},$ $\b_i \in \N^{d}$ and $0 \leq \t \leq \t_n \leq \t_{n-1} \leq \t_1 \leq t:$ 
\be \label{def:tildebfSpsi}
 {\bf S}_{\chi,n}(\t,\t_1,\dots,\t_n;t) := S_{\chi,\a_1,\b_1}(\t_1;t) S_{\chi,\a_2,\b_2}(\t_2;\t_1) \cdots S_{\chi,\a_n,\b_n}(\t_n; \t_{n-1}) \d_x^\a S_\chi(\t;\t_n).
\ee 

\begin{cor} \label{lem:Spsi} The solution $S_\chi$ to \eqref{def:Spsi} enjoys the bounds  
 \be \label{bd0:S0} 
 |\d_x^\a S_\chi(\t;t,x,\xi)| \lesssim 
   \left(\begin{array}{cc} 1 & \e^{-\zeta} \\ \e^{\zeta} & 1 \end{array}\right)
  {\bf e}_{\g^+}(\t;t,x,\xi),
 \ee
 and
 \be \label{bd:tildeSpsin}
 |{\bf S}_{\chi,n}| \lesssim \e^{- |\b| \zeta} \left(\begin{array}{cc} 1 & \e^{-\zeta} \\ \e^{\zeta} & 1 \end{array}\right)
  {\bf e}_{\g^+}(\t;t,x,\xi),
 \ee
 for any $0 \leq \t \leq t \leq T(\e),$ any $(x,\xi).$
\end{cor}

Just like the bounds of Lemma \ref{lem:s0}, the bounds of Lemma \ref{lem:Spsi} are understood entry-wise, {\it for each block} of the block-diagonal matrices $S_\chi$ and ${\bf S}_{\chi,n}.$

 While the bound of Assumption \ref{ass:main} was stated over domain ${\mathcal D},$ the bound of Corollary \ref{lem:s1} holds for any $(x,\xi),$ for $0 \leq \t \leq t \leq T(\e),$ with $T(\e)$ defined in \eqref{max-time}. This comes from the truncation $\chi$ in \eqref{def:Spsi}.

\begin{proof} %
There are five cases:

\smallskip

$\bullet$ \noindent If $\t \leq t \leq t_\star - \delta,$ then $\chi \equiv 0$ on $[\t,t]$ (see the definition of $\chi$ in \eqref{def:chi}), implying $S_\chi = \Id.$ 

\smallskip

$\bullet$ \noindent If $\t \leq t_\star - \delta \leq t,$ then by property of the flow and the previous case 
 $$ S_\chi(\t;t) = S_\chi(\t;t_\star) S_\chi(t_\star;t) = S_\chi(t_\star;t),$$
 and we are reduced to the case $t_\star \leq \t \leq t \leq T(\e).$

\smallskip

$\bullet$ \noindent If $t_\star \leq \t \leq t \leq T(\e),$ then $\chi \equiv 1$ on $[\t;t],$ and $S_\chi = S$ by uniqueness. 
 
\smallskip

$\bullet$ \noindent If $t_\star - \delta \leq \t \leq t \leq t_\star :$ comparing the equation in $S$ \eqref{tildefund} with the equation in $S_\chi$ \eqref{def:Spsi}, we find the representation
 $$ S_\chi(\t;t) = S(\t;t) - \e^{h-1} \int_\t^t S(\t';t) (1 - \chi)(\t') A_\star(\t') S_\chi(\t;\t') \, d\t',$$
 and, applying $\d_x^\a$ to both sides above,
 \be \label{rep:dS0} \d_x^\a S_\chi(\t;t) = \d_x^\a S(\t;t) - \e^{h-1} \sum \int_\t^t \d_x^{\a_1} S(\t';t) \d_x^{\a_2}\big((1 - \chi)(\t') A_\star(\t')\big) \d_x^{\a_3} S_\chi(\t;\t') \, d\t',\ee
where the sum is over all $\a_1 + \a_2 + \a_3 = \a.$ Since we are only interested in upper bounds, we omitted multinomial constants $C_{\a_i} > 0$ in \eqref{rep:dS0}. We factor out the expected growth by letting 
$S_\chi^\flat := {\bf e}_{\g^+}^{-1} S_\chi,$ $S^\flat := {\bf e}_{\g^+}^{-1} S.$ Then, by property \eqref{mult:e} of the growth function, $\d_x^\a S_{\chi}^\flat$ solves, omitting the summation sign, 
 \be \label{rep:tildeS0flat} \d_x^\a S_\chi^\flat(\t;t) = \d_x^\a S^\flat(\t;t) - \e^{h-1} \int_\t^t \d_x^{\a_1} S^\flat(\t';t) \d_x^{\a_2}\big((1 - \chi)(\t') A_\star(\t')\big)  \d_x^{\a_3} S_\chi^\flat(\t;\t') \, d\t'.\ee
Now given a matrix $M = \left(\begin{array}{cc} m_{11} & m_{12} \\ m_{21} & m_{22} \end{array}\right),$ we let $\underline M := \left(\begin{array}{cc} m_{11} & \e^{\zeta} m_{12} \\ \e^{-\zeta} m_{21} & m_{22} \end{array}\right).$ There holds identity, analogous to \eqref{cancel0}:
 \be \label{cancel:underline} \underline{M_1 M_2} = \underline M_1 \underline M_2.
 \ee In particular, if $\underline M_1$ and $\underline M_2$ are bounded in $\e,$ then $\underline{M_1 M_2}$ is bounded in $\e.$ Thus from \eqref{rep:tildeS0flat}-\eqref{cancel:underline} we deduce 
$$ \d_x^\a \underline{S}_\chi^\flat(\t;t) = \d_x^\a \underline{S}^\flat -\int_\t^t \d_x^{\a_1} \underline S^\flat(\t';t) \d_x^{\a_2} \big( (1 - \chi)(\t') (\e^{h-1} \underline A_\star(\t'))\big) \d_x^{\a_3} {\underline S}_\chi^\flat(\t;\t') \, d\t'.$$
The key is now that $\e^{h-1} \d_x^\a ((1 - \chi)\underline A_\star)9$ is uniformly bounded in $\e.$ Indeed, by Lemma \ref{lem:chi}, spatial derivatives of the truncation $\chi$ are uniformly bounded. By the block conditions \eqref{block:1}-\eqref{block:2} and definition of $\underline A_\star$ just above \eqref{cancel:underline}, the matrix $\e^{h-1} \underline A_\star$ is uniformly bounded in $\e.$ Besides, $\underline S^\flat$ is uniformly bounded, by Lemma \ref{lem:s0}. Thus we obtain the bound
$$ |\d_x^\a {\underline S}_\chi^\flat(\t;t)| \leq C\big(1 + \int_\t^t |\d_x^{\a_3} {\underline S}_\chi^\flat(\t;\t')| \, d\t'\big),$$
with $|\a_3| \leq |\a|,$ for some $C > 0,$ implying by Gronwall and a straightforward induction 
$$ |\d_x^\a {\underline S}_\chi^\flat(\t;t)| \leq C e^{C (t - \t)}, \quad t_\star - \delta \leq \t \leq t \leq t_\star,$$
which is good enough since $t - \t \leq \delta.$ Going back to $S_\chi,$ we find bound \eqref{bd0:S0}.

\smallskip

$\bullet$ \noindent The same arguments apply in the remaining case $t_\star - \delta \leq \t \leq t_\star \leq t.$ 

\smallskip

At this point \eqref{bd0:S0} is proved and we turn to \eqref{bd:tildeSpsin}. First consider products involving no spatial derivatives of $S_\chi,$ which we denote ${\bf S}_{\chi,n},$ for consistency with notation \eqref{def:bfS}. Here we note that the proof of Lemma \ref{lem:s00} uses only the upper bound \eqref{ass:up} for $S,$ via cancellation \eqref{cancel0}. We may thus repeat the proof of Lemma \ref{lem:s00} and derive a bound for ${\bf S}_{\chi,n}.$ The only difference is that, while $A_\star$ is uniformly bounded in $\e,$ the truncation function $\chi$ is stiff in $\xi,$ as seen on \eqref{def:chi} and definition of $t_\star$ in \eqref{def:tstar}, and reflected on Lemma \eqref{lem:chi}. This gives bound \eqref{bd:tildeSpsin} with ${\bf S}_{\chi,n}$ in place of $\tilde {\bf S}_{\chi,n}.$ 

Finally, spatial derivatives are well behaved, in the sense that they are bounded without loss in $\e.$ Thus from the bound in ${\bf S}_{\chi,n}$ we derive bound \eqref{bd:tildeSpsin} exactly as in the proof of Lemma \ref{lem:s0}, via a representation formula identical to \eqref{tildebfS:bfS}, with $\tilde {\bf S}_{\chi,n}$ and ${\bf S}_{\chi,n}$ in place of $\tilde {\bf S}_n$ and ${\bf S}_n.$ 
\end{proof}

\begin{cor} \label{lem:s1} The solution $S_0$ to \eqref{eq:flow} enjoys the bound  
 \be \label{bd:S0} 
 |\d_x^\a S_0(\t;t,x,\xi)| \lesssim 
 \e^{-\zeta} {\bf e}_{\g^+}(\t;t,x,\xi),
 \ee
 for any $0 \leq \t \leq t \leq T(\e),$ any $(x,\xi).$
\end{cor}

\begin{proof} From equations \eqref{def:Spsi} in $\tilde S_\chi$ and \eqref{eq:flow} in $S_0,$ we deduce the representation 
\be \label{for:low:0} \d_x^\a S_0(\t;t) = \d_x^\a S_\chi(\t;t) + \e^h \int_\t^t \d_x^{\a_1} S_\chi(\t';t) \d_x^{\a_2} (\chi B(\t')) \d_x^{\a_3} S_0(\t;\t') \, d\t',\ee
with an implicit summation in $\a_1 + \a_2 + \a_3 = \a,$ and implicit multinomial constants $C_{\a_i} > 0.$ 
We factor out the expected growth before applying Gronwall's Lemma, as we let $S_0^\flat := {\bf e}_{\g^+}^{-1} S_0$ and $S_\chi^\flat := {\bf e}_{\g^+}^{-1} S_\chi.$ From Corollary \ref{lem:Spsi}, we know that $|\d_x^\a S_\chi| \lesssim \e^{-\zeta}.$ Besides, from Lemma \ref{lem:chi} and the fact that $B \in {\bf R}_0,$ the symbol $\d_x^{\a_2} (\chi B)$ is uniformly bounded. Thus from \eqref{for:low:0}, we deduce 
$$ |\d_x^\a S_0^\flat(\t;t)| \lesssim \e^{-\zeta} + \e^{h - \zeta} \int_\t^t |\d_x^{\a_3} S_0^\flat(\t;\t')| \, d\t', \quad |\a_3| \leq |\a|.$$
We may now conclude by Gronwall's Lemma and a straightforward induction, since $T(\e)$ grows at most logarithmically and $h - \zeta > 0.$%
\end{proof}

\subsubsection{Product bounds for \texorpdfstring{$S_0$}{dx}} \label{sec:product:bounds:S0} 

The next step is to prove product bounds for $\d_x^\a S_0,$ as in Lemma \ref{lem:s0}. 
In this view, we let for $\a = (\a_1,\a_2) \in \N^{2d},$ $\b \in \N^{d},$ $0 \leq \t \leq t:$ 
\be \label{def:tildeS0ab} S_{0,\a,\b}(\t;t) := \e^{h-1} \d_x^{\a_1} S_0(\t;t) \d_x^{\a_2} \d_\xi^{\b} (\chi A_\star(\t)),\ee
and products, for $\a_i \in \N^{2d},$ $\b_i \in \N^{d}$ and $0 \leq \t \leq \t_n \leq \t_{n-1} \leq \t_1 \leq t:$ 
\be \label{def:tildebfS0}
 {\bf S}_{0,n}(\t,\t_1,\dots,\t_n;t) := S_{0,\a_1,\b_1}(\t_1;t) S_{0,\a_2,\b_2}(\t_2;\t_1) \cdots S_{0,\a_n,\b_n}(\t_n; \t_{n-1}) \d_x^\a S_0(\t;\t_n).
 \ee

\begin{cor} \label{lem:s2} There holds the bound
\be \label{bd:tildebfS0} |{\bf S}_{0,n}| \lesssim \e^{-\zeta(1 +  |\b|)} {\bf e}_{\g^+},\ee
for any $0 \leq \t \leq \t_n \leq \t_{n-1} \leq \dots \leq \t_1 \leq t \leq T(\e),$ any $(x,\xi),$ any $\a_i, \b_i,$ any $n \geq 1,$ with $|\b| = \sum_i |\b_i|.$ 
\end{cor}

\begin{proof} 
First observe that, in the case $n = 1,$ if we were to directly use Corollary \ref{lem:s1}, we would find the upper bound 
$$ \e^{h-1} |\d_x^{\a_1} S_0(\t';t) \d_x^{\a_2} A_\star(\t') \d_x^{\a_3} S_0(\t;\t')|  \lesssim \e^{h-1 - 2 \zeta} {\bf e}_{\g^+},$$ which is not good enough since $h - 1 - \zeta < 0$ if $\zeta = 1 - h.$ Hence a need for more than the bound on $\d_x^\a S_0$ from Corollary \ref{lem:s1}. 

\smallskip

Second, note that the loss in \eqref{bd:tildebfS0} is coming from $\xi$-derivatives applied to $\chi,$ exactly as in the proof of Corollary \ref{lem:Spsi}. 

\smallskip

We are going to use the representations
\be \label{rep:S0}
 \d_x^\a S_0 = \d_x^\a S_\chi + \sum_{1 \leq k \leq m-1} \int {\bf S}_{\chi,k}^B + \int {\bf S}^B_{\chi,0,m}, \quad |\a| \geq 0,
\ee
for any $m \in \N,$ with notation
\be \label{notation:SBn}
 {\bf S}^B_{\chi,k} := \e^{k h} (\d_x^{\a_1} S_\chi \d_x^{\a'_1} (\chi B)) \cdots (\d_x^{\a_k} S_\chi \d_x^{\a'_k} (\chi B)) \d_x^{\a_{k+1}}S_\chi,
 \ee
 and
\be \label{notation:SB0n}
 {\bf S}^B_{\chi,0,k} := \e^{k h} (\d_x^{\a_1} S_\chi \d_x^{\a'_1} (\chi B)) \cdots (\d_x^{\a_k} S_\chi \d_x^{\a'_k} (\chi B)) \d_x^{\a_{k+1}} S_0.
\ee
In \eqref{rep:S0} we use compact notation as in \eqref{rep:Sa}. In particular, time arguments are implicit and form ``chains" in the sense that in products $(S_\chi \d_x^{\a'_1} \chi B)(S_\chi \d_x^{\a'_2} \chi B),$ the time arguments of the first term are $(\t;\t_1)$ and those of the second term are $(\t_1;\t_2).$ The integrals in \eqref{rep:S0} are time integrals bearing on the ``interior" variables of ${\bf S}^B_k$ and ${\bf S}^B_{0k};$ see the explicit definition of ${\bf S}$ in \eqref{def:bfS} and see how compact notation in \eqref{rep:Sa} is expanded into explicit notation in \eqref{rep:Saa}.

By Assumption \ref{ass:main} and Corollary \ref{lem:s1}, there holds 
\be \label{bd:SB}
 |{\bf S}^B_{\chi,k}| + |{\bf S}^B_{\chi,0,k}| \lesssim \e^{k (h - \zeta)} \e^{-\zeta} {\bf e}_{\g^+}.
\ee 

Representation \eqref{rep:S0} is proved by using recursively \eqref{for:low:0} in itself $n-1$ times. 

\smallskip

From \eqref{rep:S0}, we find ${\bf S}_{0,1} = \e^{h-1}\d_x^{\a_1} S_0 \d_x^{\a_2} \d_\xi^\b (\chi A_\star) \d^{\a_3} S_0$ to be a sum of 9 terms:

\smallskip

\noindent $\bullet$ The term $\e^{h-1} \d_x^{\a_1} S_\chi \d_x^{\a_2} \d_\xi^\b (\chi A_\star) \d_x^{\a_3} S_\chi = {\bf S}_{\chi,1}$ is bounded with Corollary \ref{lem:Spsi}. 

\smallskip

\noindent $\bullet$ For terms ${\rm II} = \e^{h-1} {\bf S}^B_{\chi,k} \d_x^{\a_2} \d_\xi^\b (\chi A_\star) \d_x^{\a_3} S_\chi,$ we use Corollary \ref{lem:Spsi} for the rightmost product, of form $\e^{h-1} \d_x^{\a'_{k+1}} S_\chi \d_x^{\a_2} \d_\xi^\b (\chi A_\star) \d^{\a_3}_x S_\chi,$ and then bound separately the remaining $k$ products of form $\e^h \d_x^{\a'_i} S \d_x^{\a''_i} (\chi B).$ This gives $|{\rm II}| \lesssim \e^{k (h- \zeta)} \e^{-\zeta(1 + |\b|)} {\bf e}_{\g^+},$ and we use $\zeta < h.$ 

\smallskip 

\noindent $\bullet$ Term ${\rm III} = \e^{h-1} {\bf S}^B_{\chi,0,m} \d_x^{\a_2} \d_\xi^\b (\chi A_\star) \d_x^{\a_3} S_\chi$ may not be handled by the same argument as ${\rm II},$ since the rightmost product here has form $\e^{h-1} \d_x^{\a'_{m+1}} S_0 \d_x^{\a_2} \d_\xi^\b (\chi A_\star) \d_x^{\a_3} S_\chi.$ Here we use \eqref{bd:SB} with $k = m$ large, the bound $| \d_x^{\a_2} \d_\xi^\b (\chi A_\star)| \lesssim \e^{-|\b| \zeta},$ and the bound of Corollary \ref{lem:Spsi} for $|\d_x^{\a_3} S_\chi|.$ There occurs a loss of $\e^{-\zeta},$ but this is compensated by an appropriate choice of $m.$ Precisely, we find $|{\rm III}| \lesssim \e^{m (h - \zeta) - 2 \zeta - |\b| \zeta} {\bf e}_{\g^+}$ and then choose $m$ large enough, depending on $h,$ so that $m(h - \zeta) - \zeta \geq 0.$ 

\smallskip

\noindent $\bullet$ Term ${\rm IV} = \e^{h-1} \d_x^{\a_1} S_\chi \d_x^{\a_2} \d_\xi^\b (\chi A_\star) {\bf S}^B_{\chi,k}$ is symmetric to ${\rm II}.$ Here we isolate the leftmost product $\e^{h-1} \d_x^{\a_1} S_\chi \d_x^{\a_2} \d_\xi^\b (\chi A_\star) \d_x^{\a'_1} S_\chi.$ 

\smallskip

\noindent $\bullet$ In term ${\rm V} = \e^{h-1} {\bf S}^B_{\chi,k} \d_x^{\a_2} \d_\xi^\b (\chi A_\star) {\bf S}^B_{\chi,k'},$ we use the cancellation of Corollary \ref{lem:Spsi} for the term in the middle $\e^{h-1} \d_x^{\a_{k+1}} S_\chi \d_x^{\a_2} \d_\xi^\b (\chi A_\star) \d_x^{\a''_1} S_\chi.$ The remaining terms contain $k + k'$ occurences of $S,$ and there is a prefactor $ \e^{h (k + k')}.$ Thus we obtain the bound $|{\rm V}| \lesssim \e^{(k + k')(h -\zeta)} \e^{-\zeta(1 + |\b|)} {\bf e}_{\g^+}.$ 

\smallskip

\noindent $\bullet$ The remaining terms all involve at least $m$ factors, hence, by \eqref{bd:SB}, have an $\e^{m (h - \zeta)}$ prefactor. We handle these terms just like ${\rm III}$ above. 

\smallskip

From the above, we conclude that there holds 
$$ |{\bf S}_{0,1}| \lesssim \e^{-\zeta(1 + |\b|)} {\bf e}_{\g^+},$$
where $|\b|$ is the number of $\xi$-derivatives in $\tilde {\bf S}_{0,1}.$ 

The general case $n \geq 2$ is handled in exactly the same way. Via representations \eqref{rep:S0}, products involving $\d_x^\a S_0,$ as the statement of the Corollary, are expanded into sums of products involving $\d_x^{\a'} S_\chi,$ and remainders which involve products with a large number of $\e^h B$ terms, and $S_0$ terms. These remainders are handled as term ${\rm III}$ above, using $h < \zeta,$ hence $\e^{m ( h - \zeta)}$ as small as needed for $m$ large. We are then left with products involving only $\chi A_\star,$ $\chi B$ and spatial derivatives of $S_\chi.$ For these, we use Corollary \ref{lem:s0} as we did above in the treatment of terms {\rm I} and {\rm II} (using the case $n \geq 2$ in Corollary \ref{lem:Spsi}, while we used only $n = 1$ in the above treatment of terms {\rm I} and {\rm II}).   %
\end{proof}

The final bound in these preparation steps involves products of $S_0$ with $\chi(i A_\star + \e B),$ as already met in \eqref{rep:simple}. We let for $\a = (\a_1,\a_2) \in \N^{2d},$ $\b \in \N^{d},$ $0 \leq \t \leq t:$ 
$$ S^B_{0,\a,\b}(\t;t) := \e^{h-1} \d_x^{\a_1} S_0(\t;t) \d_x^{\a_2} \d_\xi^{\b} (\chi (i A_\star  + \e B)(\t)),$$
and products, for $\a_i \in \N^{2d},$ $\b_i \in \N^{d}$ and $0 \leq \t \leq \t_n \leq \t_{n-1} \leq \t_1 \leq t:$ 
\be \label{def:bfS0B}
 {\bf S}^B_{0,n}(\t,\t_1,\dots,\t_n;t) := S^B_{0,\a_1,\b_1}(\t_1;t) S_{0,\a_2,\b_2}(\t_2;\t_1) \cdots S^B_{0,\a_n,\b_n}(\t_n; \t_{n-1}) \d_x^\a S_0(\t;\t_n).
 \ee

\begin{cor} \label{cor:f} There holds the bound
 $$ |{\bf S}^B_{0,n}| \lesssim \e^{-\zeta(1 + |\b|)} {\bf e}_{\g^+},$$
 for any $0 \leq \t \leq \t_n \leq \t_{n-1} \leq \dots \leq \t_1 \leq t \leq T(\e),$ any $(x,\xi),$ any $\a_i, \b_i,$ any $n \geq 1,$ with $|\b| = \sum_i |\b_i|.$ 
\end{cor}

\begin{proof} Developing, we find ${\bf S}^B_{0,n}$ to be a product of terms of form ${\bf S}_{0,n'}$ \eqref{def:tildebfS0} with terms of form $\e^h \d_x^\a \d_\xi^\b \chi B.$ By Corollary \ref{lem:s2}, there holds $|{\bf S}_{0,n}| \lesssim \e^{-\zeta(1 + |\b|)} {\bf e}_{\g^+},$ where $|\b|$ is the total number of $\xi$ derivatives that appear in ${\bf S}_{0,n}.$ Besides, $|\e^h \d_x^\a \d_\xi^\b (\chi B)| \lesssim \e^{h - |\b|\zeta}.$ The result then follows by $\zeta < h.$  
\end{proof}

\subsubsection{Bounds on the correctors $S_q$} \label{sec:S:last} We are now ready to prove bounds on spatial derivatives of the correctors $S_q$ introduced in \eqref{Sq}: 

\begin{cor} \label{lem:new:Sq}
There holds the bounds for $0 \leq \t \leq t \leq T(\e),$ for any $(x,\xi):$ 
$$
  | \d_x^\a S_q | \lesssim \e^{- \zeta(1 + q)} {\bf e}_{\g^+}, \quad \mbox{for $0 \leq q \leq q_0.$}
  $$ 
\end{cor}

\begin{proof} For $q = 0,$ Corollary \ref{lem:s1} gives the desired bound. For $q =1,$ the corrector $S_1$ admits representation \eqref{Sq:rep}, so that $\d_x^\a S_1$ appears as a sum of terms
$$ \e^{h-1} \int_{\t}^t \d_x^{\a_1} S_{0}(\t';t) \d_x^{\a_2} \d_{\xi}^\k (\chi(i A_\star + \e B))(\t') \d_x^{\a_3 + \k} S_{0}(\t;\t')\, d\t',$$
where $|\k| = 1,$ and $\a_1 + \a_2 + \a_3 = \a.$ Corollary \ref{lem:s2} applies and gives the desired bound. Consider now representation \ref{Sq:rep} for $S_q,$ any $q \geq 2.$ In this representation, $\d_x^\a S_q$ appears as a sum of terms 
\be \label{forSq} \e^{h-1} \int_{\t}^t \d_x^{\a_1} S_{0}(\t';t) \d_x^{\a_2} \d_{\xi}^\k (\chi(i A_\star + \e B))(\t') \d_x^{\a_3 + \k} S_{q'}(\t;\t')\, d\t',\ee
where $|\k| + q' = q,$ $|\k| > 0,$ and $\a_1 + \a_2 + \a_3 = \a.$ We may recursively use representation \eqref{Sq:rep} in \eqref{forSq}, so that $\d_x^\a S_q$ appears as a time integral of terms ${\bf S}^B_{0,n}$ \eqref{def:bfS0B}, with exactly $q$ derivatives bearing on the $\xi$ variables. The result then follows from Corollary \ref{cor:f}.  
\end{proof} 

\subsubsection{Bounds on the approximate solution operator $\op_\e(\Sigma)$} 

We arrive at a bound for the action of the approximate solution operator $\op_\e(\Sigma)$ defined in \eqref{repS}: 
\begin{cor} \label{cor:op:norm} There holds the bound, for $0 \leq \t \leq t \leq T(\e):$ 
 $$ \big\| \op_\e(\Sigma(\t;t)) w \big\|_{L^2(B(x_0,\delta))} \lesssim \e^{-\zeta} \exp\big( {\bm \g}^+ (t^{1 + \ell} - \t^{1 + \ell}) \big) \| w \|_{L^2(\R^d)},$$
 where $\dsp{ {\bm \g}^+ := \max_{\begin{smallmatrix} |x| \leq \delta \\ |\xi - \xi_0| \leq \delta \end{smallmatrix}} \g(x,\xi).}$
\end{cor}

\begin{proof} Let $\theta_1$ be a spatial cut-off that is identically equal to one on a neighborhood of $B(x_0,\delta).$ Then 
 $$ \| \op_\e(\Sigma(\t;t)) w \big\|_{L^2(B(x_0,\delta))} \leq \| \op_\e(\theta_1 \Sigma(\t;t)) w \|_{L^2(\R^d)}.$$
 Now by Proposition \ref{prop:action},
 $$ \| \op_\e(\theta_1 \Sigma(\t;t)) w \|_{L^2(\R^d)} \leq \sum_{|\a| \leq d + 1} \sup_{\xi \in \R^d} | \d_x^{\a} \big( \theta_1 \Sigma(\t;t,\cdot,\xi)\big) \big|_{L^1(\R^d)} \, \| w \|_{L^2(\R^d)}.$$
 By Corollary \ref{lem:s2}, there holds for all $0 \leq q \leq q_0,$ for $|\xi - \xi_0| \leq \delta:$ 
 $$ \e^{q h} |\theta_1 \d_x^\a S_q(\t;t,\cdot,\xi)|_{L^1(\R^d)} \lesssim \e^{- \zeta(1 + q)} \sup_{|x| + |\xi - \xi_0| \leq \delta} {\bf e}_{\g^+}(\t;t,x,\xi).$$
The result then follows from $\zeta < h$ and the pointwise bound 
 $${\bf e}_{\g^+}(\t;t,x,\xi) \leq \exp( \,{\bm \g}^+(t^{1 + \ell} - \t^{1 + \ell} )\,),  \quad \mbox{for $\ell \geq 0.$}$$ 
\end{proof}

\subsubsection{Lower bound for $S_0$} We verify that the lower bound \eqref{ass:low} in Assumption \ref{ass:main} is stable by perturbation:

\begin{lem} \label{lem:really:new}  For $\delta$ and $\e$ small enough, the flow $S_0$ satisfies lower bound \eqref{ass:low} from Assumption {\rm \ref{ass:main}}, that is
 $$
  \e^{-\zeta} {\bf e}_{\g^-}(0;T(\e),x,\xi_0) \lesssim \big| \, S_0(0;T(\e),x,\xi_0) \vec e(x) \,\big|\,,
 $$ %
for $|x | \leq \delta$ and $\vec e(x)$ as in Assumption {\rm \ref{ass:main}}.
\end{lem}

\begin{proof} First observe that $t_\star(\e,0,\xi_0) = 0,$ so that the equation in $S$ and $S_\chi$ coincide over the time interval $[\delta, T(\e)],$ at $(x,\xi) = (0,\xi_0).$ A simple perturbation argument, similar to the arguments developed in detail above, then imply that $S_\chi(0; T(\e), 0, \xi_0)$ satisfies the lower bound \eqref{ass:low}. Next we use representation \eqref{for:low:0} from the proof of Corollary \ref{lem:s1}, with $\a = 0,$ which implies
$$ |S_0 \vec e \,| \geq |S_\chi \vec e \,| - \e^h \int_0^t |S_\chi(\t';t)| |B(\t')| |S_0(\t;\t')|\,d\t'.$$
As argued above, we may use lower bound \eqref{ass:low} for $S_\chi;$ besides, Corollaries \ref{lem:Spsi} and \ref{lem:s1} provide upper bounds for $S_0$ and $S_\chi.$ These yield the lower bound
$$|S_0(0;T(\e),x,\xi_0) \vec e \,| \gtrsim \e^{-\zeta} {\bf e}_{\g^-}(0;T(\e),x,\xi_0) -  \e^{h - \zeta} {\bf e}_{\g^+}(0; T(\e), x, \xi_0).$$
  Now at $(\t;t,x,\xi) = (0;T(\e),x,\xi_0),$ there holds for $\delta$ and $\e$ small enough
$$ {\bf e}_{\g^+} {\bf e}_{\g^-}^{-1} \lesssim \exp\Big( \big( \g^+(x,\xi_0) - \min_{|x| \leq \delta} \g^-(\cdot,\xi_0) \big) \,  T(\e)^{1 + \ell}\Big).$$
The constant 
$$ %
 \delta_0 := \max_{\begin{smallmatrix} |x| \leq \delta \\ |\xi - \xi_0| \leq \delta \end{smallmatrix} } \g^+(\cdot,\xi_0) - \min_{|x| \leq \delta} \g^-(\cdot,\xi_0)
 $$ %
is small for small $\delta,$ by continuity of $\g^\pm$ and the fact that $\g^+(0,\xi_0) = \g^-(0,\xi_0).$ In particular, we may choose $\delta,$ depending on $T_\star$ introduced in \eqref{max-time} and satisfying condition \eqref{cond:T1}, so that
\be \label{cond:delta0} h - \zeta - T_\star \delta_0 > 0.\ee
Thus for $|x| < \delta$ and $\e$ small enough, there holds
$$ 1 - \e^{h - \zeta} {\bf e}_{\g^+} {\bf e}_{\g^-}^{-1}(0;T(\e),x,\xi_0) \geq 1 - \e^{h - \zeta - T_\star \delta_0} \geq 1/2,$$
and the result follows.
\end{proof}

\subsubsection{Bounds on $\d_x^\a \d_\xi^\b S_q$} \label{sec:xibounds} We finally give bounds on $(x,\xi)$-derivatives of $S_0$ and correctors $S_q.$ This ends the verification of Assumption \ref{ass:BS}. 

\begin{lem} \label{lem:crude} There holds bounds, for $0 \leq \t \leq t \leq T(\e),$ for any $(x,\xi):$ 
 \be \label{bd:crude} |\d_x^\a \d_\xi^\b S_q| \lesssim \e^{- \zeta(1 + |\b| + q)} {\bf e}_{\g^+}.\ee
\end{lem}

\begin{proof} From representation
 $$ \d_\xi^\b S_0(\t;t) = - \e^{h-1} \sum_{\begin{smallmatrix} \b_1 + \b_2 = \b \\ |\b_1| > 0 \end{smallmatrix} } \int_0^t S_0(\t';t) \d_\xi^{\b_1} (\chi(i A_\star + \e B))(\t') \d_\xi^{\b_2} S_0(\t;\t') \, d\t'$$
 and Corollary \ref{cor:f}, we find by induction bound \eqref{bd:crude} for $q = 0.$ Besides, applying $\d_x^\a \d_\xi^\b$ to the equation in $S_q,$ we find the representation
  $$ \d_x^\a \d_\xi^\b S_{q} = - \e^{h-1} \int_{\t}^t S_{0} \Big(  \d_x^{\a_1} \d_\xi^{\b_1} (\chi( i A_\star + \e B)) \d_x^{\a_2} \d_\xi^{\b_2} S_q  + \d_x^\a \d_\xi^\b \big( \, (\chi(i A_\star + \e B)) \sharp_{q_1} S_{q_2} \, \big)\, \Big),$$ %
  with an implicit summation over $\a_1 + \a_2 = \a,$ $\b_1 + \b_2 = \b,$ $q_1 + q_2 = q$ with $|\a_1| > 0,$ $|\b_1| > 0,$ $q_1 > 0.$ We use recursively the above representation, and find $\d_x^\a \d_\xi^\b S_q$ to be a sum of terms ${\bf S}^B_{0,n},$ with a total number of $\xi$ derivatives equal to $|\b| + q.$ It now suffices to apply Corollary \ref{cor:f}. 
\end{proof}
 
\subsubsection{Conclusion} \label{sec:conclusion} Before moving on to the third and last part of the proof of Theorem \ref{th:main}, we recapitulate our arguments so far. 

 The bound $|\d_x^\a \d_\xi^\b (\chi A_\star)| \lesssim \e^{- |\b|\zeta}$ and Lemma \ref{lem:crude} verify Assumption \ref{ass:BS}. Thus at this point the integral representation \eqref{duhamel} is justified, via Theorem \ref{duh1}. We reproduce here equation \eqref{duhamel}:
 $$ v  = \op_\e(\S(0;t)) ) v(0) + \e^{h} \int_0^t \op_\e(\S(\t;t)) (\Id + \e \op_\e({\bf R}_0)) \big(g(\t) + \e\op_\e({\bf R}_0) v(0) \big)  d\t.$$
This ends the second part of the proof. Next we will prove a lower bound for $v,$ in norm $L^2(B(0,\delta)).$ In this view, we will bound from above the time-integrated term in \eqref{duhamel}, and bound from below the ``free" solution $\op_\e(\S(0;t)) v(0).$ The bound from above for the time-integrated term rests on Corollary \ref{cor:op:norm} bounding the action of $\op_\e(\S),$ and on a bound for the source $g,$ which is the object of the forthcoming Section \ref{sec:bd-source}. The bound from below for the free solution is a consequence of the above Lemma \ref{lem:really:new}.

 \subsection{Bound on the source term} \label{sec:bd-source} 
 
 Our goal in  this Section is to give an upper bound for the source term $g$ from equation \eqref{eq:v}. The source $g$ is defined in \eqref{def:g}, which we reproduce here:
$$ g =  \e^{-1} (\Gamma v - \tilde \Gamma (u_\star(\e^h t))) + \e^{-h} \op_\e(\d_t \tilde \chi) (u_\star (\e^h t))  +  \op_\e(\tilde \chi) (F_\star(\e^h t)) + \e^n \big(\op_\e({\bf R}_0) u_\star\big)(\e^h t).$$
We let 
$$ g_0 := \op_\e(\tilde \chi) (F_\star(\e^h t)) + \e^n \big(\op_\e({\bf R}_0) u_\star\big)(\e^h t),$$
and
$$ g_\chi := \e^{-1} (\Gamma v - \tilde \Gamma (u_\star(\e^h t))) + \e^{-h} \op_\e(\d_t \tilde \chi) (u_\star (\e^h t)),$$
so that $g = g_0 + g_\chi.$ We first consider $g_0.$ 

\begin{lem} \label{lem:source:new} There holds the bound 
 $$ \|\op_\e(\chi^\flat) g_0(t)\|_{L^2} \lesssim  \|\dot u(\e^h t) \|_{W^{1,\infty}(B(0,\delta))}^2 + \e^n \| \dot u(\e^h t) \|_{L^2(B(0,\delta))},$$
 uniformly in $t \in [0, T(\e)],$ for $n$ as large as allowed by the regularity of $\phi,$ where the possibly stiff truncation $\chi^\flat$ is defined just below \eqref{def:chi}, in particular such that $\chi^\flat \prec \tilde \chi,$ with notation introduced in \eqref{def:prec}. 
 \end{lem}
 
\begin{proof} There are three types of terms in $g_0:$ (a) a commutator issued from the spatial cut-off $\theta,$ (b) nonlinear terms, (c) remainders of size $O(\e^n).$ The term $F_\star$ is defined in \eqref{def:tildeF}, in terms of $F^\flat$ defined in \eqref{def:Fflat}, itself defined in terms of $\dot F$ \eqref{def:dotF}. We reproduce here their definitions:
 $$ \begin{aligned}
\dot F & = G_0(\e, t,x, \dot u) \cdot (\dot u, \dot u)+ \sum_{1 \leq j \leq d} G_{1j}(\e, t,x, \dot u) \cdot (\dot u, \d_{x_j} \dot u), \\
F^\flat & = \op_\e(Q_\e)(\theta \dot F)  -  \sum_{1 \leq j \leq d} \op_\e(Q_\e) \big( A_j(\phi_\e) \dot u \d_{x_j} \theta \big) + \e^n {\bf R}_0 (\theta \dot u), \\ 
 F_\star & = M^\star F^\flat + \e^n \op_\e({\bf R}_0) u^\flat.
 \end{aligned}$$

(a) Commutator term: this is the term in $\d_{x_j} \theta$ in $F^\flat.$ 

 This term is rendered small by the left action of $\op_\e(\chi^\flat).$ Indeed, commutators that arise from a localization step depend on derivatives $p',$ where $p$ is the localization symbol. If we further localize with $p^\flat,$ such that $p^\flat$ is identically equal to 1 on the support of $p,$ then from $(1 - p) p^\flat = 0$ we deduce $p^\flat p' \equiv 0,$ and the associated commutator is arbitrarily small. This is made precise below. 

Consider first the term 
 $$ \op_\e(\chi^\flat) M^\star \op_\e(Q_\e) (A_j(\phi_\e) \dot u \d_{x_j} \theta).$$
We start by approximately commuting $\op_\e(\chi^\flat)$ and $M^\star,$ by use of \eqref{egorov} (with $M$ in place of $M^\star$ and conversely): in view of Lemma \ref{lem:chi}, there holds
 \be \label{chi:M} \op_\e(\chi^\flat) M^\star = M^\star \op_\e(\chi^\flat_{(\star)}) + \e^{h - \zeta} \op_\e({\bf R}_{0}),\ee
where $\chi^\flat_{(\star)}$ denotes evaluation of $\chi^\flat$ along the backward characteristics of $\mu_\e.$ Note that we use here definition \eqref{R0'} for ${\bf R}_0.$ In particular, the characteristics depend on time through $\e^h t.$ Expanding $\chi^\flat_{(\star)}$ in powers of $\e^h t,$ for $t \leq T(\e),$ we find terms that are all supported in the support of $\chi^\flat,$ up to a remainder $O(\e^n),$ with $n$ as large as allowed by the regularity of $\phi.$ In conjunction with Proposition \ref{prop:composition}, this implies
$$ M^\star \op_\e(\chi^\flat_{(\star)}) \op_\e(Q_\e) = \op_\e(\theta_{0} {\bf R}_0) + \e^{n} \op_\e( {\bf R}_{0}),$$
where $\theta_0$ is defined in Section \ref{sec:trunc}, in particular such that $(1 - \theta) \theta_0 \equiv 0.$ Now
$$ \op_\e(\theta_0 {\bf R}_0) (A_j(\phi_\e) \dot u \d_{x_j} \theta) = \op_\e(\theta_0 \d_{x_j} \theta {\bf R}_0) (A_j (\phi_\e) \dot u) + \big[ \op_\e(\theta_0 {\bf R}_0),\, \d_{x_j} \theta\big] (A_j(\phi_\e) \dot u),$$ 
and the first term above is identically zero, since $\theta_0 \d_{x_j} \theta \equiv 0.$ For the second term, we use Proposition \ref{prop:composition}: up to $O(\e^n),$ the operator involves product of $\theta_0$ and its derivatives with derivatives of $\d_{x_j} \theta.$ These products are identically zero. Thus 
 $$ \| M^\star \op_\e(\chi^\flat_{(\star)}) \op_\e(Q_\e)(A_j (\phi_\e) \dot u) \|_{L^2} \lesssim \e^n \| \dot u \|_{L^2(B(0,\delta))}.$$
Now Egorov's lemma \eqref{egorov}, as used in \eqref{chi:M}, may be expanded up to arbitrary order, as in \eqref{R:M:bis}. The supports of the symbols that appear in this expansion share the property that we used for $\chi^\flat_{(\star)}.$ We obtain 
 $$ \| \op_\e(\chi^\flat) M^\star \op_\e(Q_\e) (A_j(\phi_\e) \dot u \d_{x_j} \theta) \|_{L^2} \lesssim \e^n \| \dot u\|_{L^2(B(0,\delta))}.$$

(b) The nonlinear terms are {\it local} in $\dot u$ and $\nabla_x \dot u,$ so that
$$ \theta(x) G_0(\e,t,x, \dot u) \cdot (\dot u, \dot u) \equiv \theta(x) G_0(\e, t,x, \theta^\sharp \dot u) \cdot (\theta^\sharp \dot u, \theta^\sharp \dot u),$$
where $\theta \prec \theta^\sharp$ (see notation \eqref{def:prec}), with $\mbox{supp}\,\theta^\sharp \subset B(0,\delta).$ A similar identity holds for $G_1.$ Thus
$$ \begin{aligned} \| \theta G_0(\dot u) \cdot (\dot u, \dot u) \|_{L^2} & \leq \|  \theta G_0(\theta^\sharp \dot u) \cdot (\theta^\sharp \dot u, \theta^\sharp \dot u) \|_{L^2} \\ & \lesssim C\big(|\theta^\sharp \dot u|_{L^\infty}\big) |\theta^\sharp \dot u|_{L^\infty} \|\theta^\sharp \dot u\|_{L^2}.\end{aligned}$$
Since $\theta^\sharp$ has support in $B(0,\delta),$ there holds $\| \theta^\sharp \dot u\|_{L^2} \lesssim \| \dot u\|_{L^2(B(0,\delta))},$ and same in $L^\infty.$ By the a priori bound \eqref{bound:dotu}, there holds in particular $|\dot u|_{L^\infty(B(0,\delta))} \lesssim 1,$ since $K' > 0.$ Thus 
 $$ \| \theta G_0(\dot u) \cdot (\dot u, \dot u) \|_{L^2} \lesssim \| \dot u\|_{L^\infty(B(0,\delta))}^2.$$ 
A similar argument for $G_1$ yields an upper bound that involves $|\dot u|_{W^{1,\infty}(B(0,\delta))}.$ 
We conclude that
 $$ \| \op_\e(\chi^\flat) M^\star \op_\e(Q_\e) (\theta \dot F) \|_{L^2} \lesssim \| \dot u\|_{W^{1,\infty}(B(0,\delta))}^2.$$

(c) The remainders of form $\e^n \op_\e({\bf R}_0) (\theta \dot u)$ and $\e^n \op_\e({\bf R}_0) u^\flat$ in $g$ contribute $\e^n \| \theta \dot u\|_{L^2}$ to the estimate for $\op_\e(\chi^\flat) g,$ by definition of uniform remainders (Section \ref{sec:unif}), and bound \eqref{ap:flat} in $u^\flat.$ Same for the remainder in $u_\star,$ since $\| u_\star \|_{L^2} \lesssim \| u^\flat \|_{L^2},$ by property of $M^\star.$ We use in addition Corollary \ref{cor:chi} for all these remainders.
\end{proof}

\begin{cor} \label{cor:so:new} There holds, for any $P \in S^0,$ for the source term $g_0$ defined just above the statement of Lemma {\rm \ref{lem:source:new}}, the bound 
 \be \label{bd:p:g} \| \op_\e(\chi^\flat) \op_\e(P) g_0 \|_{L^2} \lesssim \e^{2K'} \| P \|_{0,C(K')} ,
 \ee
 uniformly in $t \in [0, T(\e)],$ for some $C(K') > 0,$ and where $\|\cdot \|_{0,r}$ is the norm in $S^0$ defined in \eqref{notation:norm:eps}. The constant $K'$ is defined in \eqref{def:K'} in terms of $K,\a,m$ and $d.$  \end{cor}

\begin{proof} By Lemma \ref{lem:source:new} and Proposition \ref{prop:action}, there holds
 $$ \| \op_\e(P)  \op_\e(\chi^\flat) g \|_{L^2} \lesssim \| P \|_{0,C(d)} \big( \|\dot u\|_{L^2(B(0,\delta))}^2 + \e^n \| \dot u \|_{L^2(B(0,\delta))} \big),$$
 and with \eqref{bound:dotu} we obtain the upper bound \eqref{bd:p:g} by taking $n = K'.$  
 There remains to handle the commutator $[\op_\e(\chi^\flat), \op_\e(P)].$ Here we use Proposition \ref{prop:composition}: modulo terms which are $O(\e^n),$ the symbol of the commutator is a sum of terms of the form $P_\a := \e^{h |\a|} (\d_x^\a \chi^\flat \d_\xi^\a P - \d_\xi^\a \chi^\flat \d_x^\a P).$ Sine $\chi^\flat \prec \tilde \chi$ (using notation introduced in \eqref{def:prec}), there holds $P_\a \equiv \tilde \chi P_\a.$ Thus, by Proposition \ref{prop:composition} again
  \be \label{pqa} \op_\e(P_\a) g = \op_\e(P_\a) \op_\e(\tilde \chi) g  + \e^{h(|\a| + 1)} \op_\e(Q_\a) g,
  \ee
 where the leading terms in $Q_\a$ have the same form as $P_\a.$ For the first term in \eqref{pqa} above, we use Corollary \ref{lem:source:new}, as we may since $\tilde \chi \prec \chi.$ We use also the fact that $\| \op_\e(P_\a)\|_{L^2 \to L^2}$ is bounded uniformly in $\e,$ by Lemma \ref{lem:chi} and bound \eqref{bd:action:H} in Proposition \ref{prop:action}. For the leading terms in $Q_\a,$ we use inductively \eqref{pqa}, and arrive at \eqref{bd:p:g}. The 
Taylor expansion in the composition of operators needs to be carried out up to order $O(K'),$ hence a dependence in $\| P \|_{0,C(K')}$ in the upper bound. \end{proof}

\begin{lem} \label{lem:source:psi} There holds the bound, for any $P \in S^0,$ for $g_\chi$ defined just above Lemma {\rm \ref{lem:source:new}}: 
 $$ \| \op_\e(\chi^{\flat} )\op_\e(P) g_\chi \|_{L^2} \lesssim \e^{K' + n} \| P \|_{0,n},$$
 uniformly in $t \in [0, T(\e)],$ for $n$ as large as allowed by regularity of $\phi.$%
\end{lem}

\begin{proof} First consider the case $P = \Id.$ Since derivatives of $\tilde \psi_0$ vanish identically on the support of $\psi_0^{\flat},$ there holds, by Proposition \ref{prop:composition}: 
$$ \| \op_\e(\chi^{\flat}) \op_\e(\d_t \tilde \chi) \|_{L^2 \to L^2} \lesssim \e^{n h} \| \d_\xi^n \chi^{\flat} \|_{0,C(d)} \| \d_x^n \d_t \tilde \chi \|_{0,C(d)}.$$
By Lemma \ref{lem:chi}, this implies
$$ \| \op_\e(\chi^{\flat}) \op_\e(\d_t \tilde \chi) \|_{L^2 \to L^2}  \lesssim \e^{n h} \e^{- (n + 2 C(d)) \zeta},$$
and the above is arbitrarily small if $n$ is large enough, since $\zeta < h.$  The same argument holds for the other two terms in $g_\psi.$ In the general case of $P \in S^0,$ we may reason as in the proof of Corollary \ref{cor:so:new}, that is by first spelling out the composition $\op_\e(\psi^{\flat}) \op_\e(P)$ up to a large order, then using the above. Remainders are small by condition $\zeta < h.$ 

The above argument prove the bound  
\be \label{for:P} \| \op_\e(\psi^{\flat} )\op_\e(P) g_\psi \|_{L^2} \lesssim \e^n (\| u_\star(\e^h t)\|_{L^2} + \| v\|_{L^2}).\ee
We finally use bound \eqref{bd:action} in Proposition \ref{prop:action} in order to control $\| v \|_{L^2(\R^d)},$ and obtain, via Lemma \ref{lem:chi}, 
$$\| v \|_{L^2(\R^d)} \lesssim \e^{-\zeta C(d)} \| u_\star (\e ^h t) \|_{L^2(\R^d)}.$$ This loss is irrevelant since in \eqref{for:P} the integer $n$ can be chosen to be very large. We conclude by bound $\| u_\star \|_{L^2} \lesssim \| u^\flat \|_{L^2}$ and bound \eqref{ap:flat} for $u^\flat.$ \end{proof}

 \subsection{Lower bound for the free part of the solution} \label{sec:free}
 
 First we describe the time transition function and the truncation $\psi_0$ for frequencies close to $\xi_0:$
 
\begin{lem} \label{lem:tstar} There holds the bound, for $|x| \leq \delta$ and $|\xi| \leq \delta \e^{\zeta - h},$
$$ 0 \leq t_\star(\e,x,\xi_0 + \e^h \xi) \leq C \delta,$$
for some $C>0$ independent of $\delta.$ In particular, for such $(x,\xi),$ there holds
$$ \psi_0\big(- t_\star(\e,x,\xi_0 + \e^h \xi)\big) \equiv 1, \quad \psi_0\big(T(\e) - t_\star(\e,x,\xi_0 + \e^h \xi)\big) \equiv 1,$$
and same for $\tilde \psi_0$ and $\psi_0^\flat.$  
\end{lem}

\begin{proof} We may assume $\ell = 1/2$ here. 
There holds
\be \label{taylor:thetastar} \begin{aligned} t_\star(\e,x,\xi_0 + \e^h \xi)  & = \e^{-h} \theta_\star(\e^{1-h} x, \xi_0 + \e^h \xi ) \\ &  = \e^{-h} \theta_\star(\e^{1-h} x, \xi_0) + \int_0^1 \d_\xi \theta_\star(\e^{1-h} x, \xi_0 + \e^h \t \xi) \cdot \xi d\t,\end{aligned}\ee
and, by assumption on $\theta_\star$ \eqref{def:tstar}, for $|x| \leq \delta$ and $|\xi| \leq \delta \e^{- h + \zeta},$ there holds
$$ |\e^{-h} \theta_\star(\e^{1-h} x, \xi_0)| \leq C \e^{-h + 2(1-h)} \delta^2 = C \delta^2,$$
and
$$ |\d_\xi \theta_\star(\e^{1-h} x, \xi_0 + \e^h \t \xi)| \leq C (\e^{1- h} |x| + \e^h |\xi|) \leq C \e^\zeta \delta,$$
for some $C > 0$ which does not depend on $\delta.$ This proves the bound on $t_\star.$ In particular, for such $(x,\xi),$ there holds $- t_\star \geq - \delta/9,$ and $T(\e) - t_\star \geq -\delta/9,$ implying the result for $\psi_0,$ by definition of $\psi_0$ (see for instance Figure 6).
\end{proof}

 Based on the above result for $t_\star$ and $\tilde \psi_0,$ we identify the leading term in the datum for $v$ defined in \eqref{def:vpsi}:
 
 \begin{cor} \label{lem:vk0} There holds
 \begin{equation} \label{dec-vk0} \big\| v(0) - \e^K e^{i x \cdot \xi_0/\e} \tilde \theta_0(x) \vec e(x) \, \big\|_{L^2} \lesssim \e^{K+h - \zeta},\ee
 where $\tilde \theta_0$ is the spatial cut-off introduced in Section {\rm \ref{sec:trunc0}}.
\end{cor}

\begin{proof} At this point the reader may find useful to jump back to Section \ref{sec:avatars}. Since $M^\star_{|t = 0} = \Id,$  
  the datum for $v$ is
 $$ v(0) = \e^K \op_\e(\tilde \chi(0)) \op_\e(Q_\e(0)) \Big( \theta \, \Re e \, \Big( \op_\e(Q_\e(0)^{-1}) \Big( e^{i (\cdot) \cdot \xi_0/\e^{h}} \theta \vec e \, \Big) \, \Big) \Big).$$
 We may commute $\op_\e(Q_\e)$ and $\theta,$ since this produces an error that is $O(\e^{K + h})$ in $L^2,$ thanks to Proposition \ref{prop:composition}. Then in the datum we handle separately the oscillations in $+\xi_0/\e^{h}$ and the oscillations in $-\xi_0/\e^{h}.$ There holds
$$ \e^K \op_\e(Q_\e(0)) \op_\e(Q_\e(0)^{-1}) \big( e^{i (\cdot)\cdot \xi_0/\e^{h}} \theta \vec e \,\big) = \e^K e^{i x \cdot \xi_0/\e^{h}} \theta \vec e\, + O(\e^{K +1}),$$
where $O(\cdot)$ denotes a control in $L^2.$ 
Thus the oscillation in $\xi_0/\e^{h}$ contributes to $v(0)$ the term 
 $$ \e^K \op_\e(\tilde \chi(0)) \big( e^{i x \cdot\xi_0/\e^h} \theta^2 \vec e\, \big) = \e^K e^{i x \cdot\xi_0/\e^h} \op_\e(\tilde \chi(0,x,\xi_0 + \cdot))  \big(  \theta^2 \vec e\, \big)$$
 modulo terms that are $O(\e^{K+h}).$ 
 By definition of $\tilde \chi$ in Section \ref{sec:trunc0}: 
$$ \tilde \chi(0, x, \xi_0 + \e^h \xi) = \tilde \chi_0(\e^{h - \zeta} \xi) \tilde \theta_0(x) \tilde \psi_0(- t_\star(\e,x,\xi_0 + \e^h \xi).$$
In the above we may use Lemma \ref{lem:tstar}, since $\tilde \chi_0(\e^{h - \zeta} \xi)$ is non-zero only if $|\xi| \leq \delta \e^{\zeta - h}.$ Thus
$$ \tilde \chi(0, x, \xi_0 + \e^h \xi) = \chi_0(\e^{h - \zeta} \xi) \tilde \theta_0(x) = \tilde \theta_0(x) + \e^{h - \zeta} O(\xi).$$
We may expand the above up to an arbitrary power of $\e^{h - \zeta}.$ The remainder is a stiff symbol in $\xi,$ hence we will lose $\e^{- \zeta C(d)}$ in evaluating its operator norm (in accordance with Proposition \ref{prop:action}), but such a loss is irrevelant if the prefactor $\e^{n ( h-  \zeta)}$ is large enough. Also, the Taylor expansion brings out a large $\xi^n$ prefactor, implying that the $L^2$ norm of the remainder depends on the high Sobolev norm $\| \tilde \theta_0 \vec e \, \|_{H^n},$ but $\tilde \theta_0 \in C^\infty_c$ and $\vec e$ is assumed to be smooth, hence this derivative loss is irrevelant as well. We thus obtain 
 $$ \| \op_\e(\tilde \chi(0)) \big( e^{i x \cdot\xi_0/\e^h} \theta^2 \vec e \, \big) - e^{i x \cdot\xi_0/\e^h} \tilde \theta_0 \theta^2 \vec e \, \|_{L^2} \lesssim \e^{h - \zeta}.$$
 By choice of $\theta_0,$ there holds $\theta_0 \prec \theta.$ Thus $\tilde \theta_0 \theta^2 \equiv \tilde \theta_0,$ and we obtained the leading term $\e^K e^{i x \cdot \xi_0/\e^h} \tilde \theta_0(x) \vec e\,(x)$ as claimed in \eqref{dec-vk0}. 
  
  It remains to prove that the oscillation in $- \xi_0/\e^{h}$ has a small contribution to the datum. The leading term in the datum associated with this oscillation has the form
  $$ \e^K  e^{- i x \cdot \xi_0/\e^{h}} \op_\e(\tilde \chi(-\xi_0 + \cdot) P_\e) \tilde \theta, \qquad \tilde \theta := \overline{\theta \vec e \,},$$
  where $P_\e \in S^0,$ uniformly in $\e.$ The key is then that $\op_\e(\chi(- \xi_0 + \cdot) P_\e) \tilde \theta,$ which is smooth and supported in $B(0,\delta)$ (because $\tilde \theta_0$ is supported in $B(0,\delta)$) is pointwise bounded by $\big| {\mathcal F} \tilde \theta \big|_{L^1(|\xi| \geq c/\e^{h})},$ with $c = 2 |\xi_0| - \e^{\zeta} \delta/2 > 0.$ This $L^1$ norm is arbitrarily small, since ${\mathcal F} \tilde \theta$ belongs to the Schwartz class. Spatial derivatives are handled in the same way. 
\end{proof}

\begin{cor}\label{lem:low} For $v$ defined in \eqref{def:vpsi}, there holds for $\e$ and $\delta$ small enough the lower bound 
\begin{equation} \label{low0} 
\big\| \op_\e(\chi^\flat(T(\e))) \op_\e(S_{0}(0;T(\e))) v(0) \big\|_{L^2(B(0, \delta))} \gtrsim \e^{K - \zeta} \exp\big( {\bm \g}^-  T(\e)^{\ell + 1}\big),
\end{equation}
where ${\bm \g}^- := \min_{|x| \leq \delta} \g^-(x,\xi_0)$ and $\chi^\flat$ is introduced in Section {\rm \ref{sec:trunc0}}.%
\end{cor}

\begin{proof}
 According to Corollary \ref{lem:vk0}, the datum is 
$$ %
 v(0,x) = \e^K \tilde v(0,x) + O(\e^{K + h - \zeta}), \qquad \tilde v(0,x) = e^{i x \cdot \xi_0/\e^{h}} \tilde \theta_0(x) \vec e(x).
$$ %
  There holds
\be \label{calc:V0}
   \op_\e(S_{{0}}(0;T(\e))) \tilde v(0)(x)  =  e^{i x \cdot\xi_0/\e^{h}} S_{{0}}(0;T(\e),x,\xi_0) \tilde \theta_0 \, \vec e (x)+ \e^{h} V_0, 
\ee
 with
 $$  V_0 := e^{i x \cdot (\xi + \xi_0/\e^h)}
  \sum_{|\a| = 1} \op_\e\left( \int_0^1 \left(\d^\a_\xi S_0\right)(0;T(\e),x,\xi_0 + \t (\cdot)) \, d\t \right) \big(\d_x^\a ( \tilde \theta_0 \vec e\,)\big)(x).
  $$
We now apply $\op_\e(\chi^\flat(T(\e)))$ to \eqref{calc:V0}. The leading term is 
$$ e^{i x \cdot \xi_0/\e^h} \op_\e(\chi^\flat(T(\e),x,\xi_0 + \cdot)) \big( S_{{0}}(0;T(\e),\cdot,\xi_0) \tilde \theta_0 \vec e \, \big).$$
By definition of $\chi^\flat$ in Section \ref{sec:trunc},
 $$ \chi^\flat(T(\e), x, \xi_0 + \e^h \xi) = \chi_0^\flat(\e^{h - \zeta} \xi) \theta_0^\flat(x) \psi_0^\flat\big(T(\e) - t_\star(\e,x,\xi_0 + \e^h \xi)\,\big).$$
We use Lemma \ref{lem:tstar}, as we did in the proof of Corollary \ref{lem:vk0}: there holds
 $$ \chi^\flat(T(\e), x, \xi_0 + \e^h \xi) = \chi_0^\flat(\e^{h - \zeta} \xi) \theta_0^\flat(x) = \theta_0^\flat(x) \left(1 + \e^{h - \zeta} \int_0^1 \d_\xi \chi_0^\flat(\e^{h - \zeta} \t \xi) \cdot \xi \, d\t\right).$$
The leading term now appears as
$$ U_{0} = e^{i x \cdot \xi_0/\e^h} S_{{0}}(0;T(\e),x,\xi_0) \theta^\flat_0(x) \vec e(x) \,,$$
since $\theta_0^\flat \prec \tilde \theta_0.$ And since $\theta_0^\flat \equiv 1$ in a neighborhood of $0,$ we may use Lemma \ref{lem:really:new}, which states that for $\e$ and $\delta$ small enough and $|x| < \delta,$ there holds
  $$ \e^{-\zeta} {\bf e}_{\g^-}(0;T(\e),x,\xi_0) \lesssim  \big| \, S_{{0}}(0;T(\e),x,\xi_0) \vec e\,(x)\big|.$$
 Consider the lower growth function ${\bf e}_{\g^-},$ as defined in \eqref{def:bf-e}, at $(x,\xi) = (x,\xi_0).$ It involves $t_\star(\e,x,\xi_0).$ By Lemma \ref{lem:tstar}, $0 \leq t_\star(\e,x,\xi_0) \leq C \delta.$
Thus there holds the lower bound 
 $$  \e^{-\zeta} \exp \big( {\bm \g^-} T(\e)^{1 + \ell} \big) \lesssim \e^{-\zeta} {\bf e}_{\g^-}(0;T(\e),x,\xi_0),$$
 uniformly in $|x| \leq \delta,$  with ${\bm \g^-}$ defined in the statement of this Corollary. Hence a lower bound for the $L^2(B(0,\delta))$ norm of the leading term of the free solution as desired. 
 
There remains to bound from above the terms we overlooked so far. The first involves the remainder in the datum, which is $O(\e^{K + h - \zeta})$ in $L^2$ norm. With Corollaries \ref{cor:chi} and \ref{cor:op:norm}, the action of $\op_\e(S_0)$ on this remainder is controlled by $\e^{h - \zeta} \e^{K - \zeta} \exp\big({\bm \g}^+ T(\e)^{1 + \ell}\big).$ 

The other terms are $\e^{K + h} \op_\e(\chi^\flat(T(\e))) V_0$ and $\e^{K + h - \zeta} W_0,$ with notation $$W_0 := e^{i x \cdot \xi_0/\e^h} \theta_0^\flat(x) \sum_{|\a| = 1} \op_\e\left(\int_0^1 \d_\xi^\a \chi_0^\flat(\e^{ - \zeta} \t (\cdot)) \, d\t\right) \big( \d_x^\a \big( S_{{0}}(0;T(\e),\cdot,\xi_0) \tilde \theta_0 \vec e \,\big) \big)(x).$$
By Corollary \ref{cor:chi}, Lemma \ref{lem:crude} and Proposition \ref{prop:action}, there holds
 $$ \e^{K + h} \| \op_\e(\chi^\flat(T(\e))) V_{0}\|_{L^2(\R^d)} \lesssim \e^{h - \zeta} \e^{K - \zeta} \exp\big( {\bm \g}^+ T(\e)^{1 + \ell} \big).$$%
For $W_0,$ we observe that $\op_\e(\d_\xi^\a \chi_0^\flat)$ is a Fourier multiplier, bounded in $L^2 \to L^2$ norm by the maximum of its symbol, equal to a constant independent of $\e.$ We may then bound $\d_x^\a S_0$ with Corollary \ref{lem:s1}. Thus $\e^{K + h - \zeta} W_0$ is controlled in $L^2$ just like $\e^{K + h} \op_\e(\chi^\flat(T(\e))) V_0.$ 

Summing up, we obtained a lower bound of form  
 $$ \e^{K - \zeta} \exp\big( {\bm \g}^-  T(\e)^{\ell + 1}\big)\big(1 - \e^{h - \zeta} \exp\big( ({\bm \g}^+ - {\bm \g}^-) T(\e)^{1 + \ell}\big)\,\big).$$
By definition of $T(\e)$ in \eqref{max-time}, 
$$ \e^{h - \zeta} \exp\big( ({\bm \g}^+ - {\bm \g}^-) T(\e)^{1 + \ell}\big) = \e^{h - \zeta - ({\bm \g}^+ - {\bm \g}^-) T_\star}.$$
Given $T_\star,$ since $h - \zeta > 0,$ we may choose $\delta$ small enough so that the difference ${\bm \g}^+ - {\bm \g}^-$ is so small that $h - \zeta - ({\bm \g}^+ - {\bm \g}^-) T_\star$ is strictly positive. The result follows is $\e$ is small enough.
\end{proof}

\subsection{Endgame} \label{sec:end0} We apply $\op_\e(\chi^\flat(T(\e)))$ to the left of the representation formula \eqref{duhamel} for $v$ at $t = T(\e),$ with $\chi^\flat$ defined in Section \ref{sec:trunc}, and prove that the contribution of the initial datum dominates the time-integrated Duhamel term. This eventually provides a contradiction to the assumed a priori bound \eqref{new-big-bound}, and concludes the proof. 
 
Based on \eqref{duhamel}, we find 
$$ \|  \op_\e(\chi^\flat(T(\e))) v(T(\e)) \|_{L^2(B(0,\delta))} \geq {\rm I} - ({\rm II} + {\rm III}).$$
 The leading term is 
 $$ {\rm I} =  \op_\e(\chi^\flat(T(\e))) \op_\e(S_0(0;T(\e))) v(0).$$ 
 This term is bounded from below in $L^2(B(0,\delta))$ norm by Corollary \ref{lem:low}:
 $$ \| {\rm I} \|_{L^2(B(0,\delta))} \gtrsim \e^{K - \zeta} \exp\big( {\bm \g}^-  T(\e)^{\ell + 1}\big).$$
 The error term in the contribution of the datum is 
 $$ {\rm II} := \sum_{1 \leq q \leq q_0} \e^{h q} \op_\e(\chi^\flat(T(\e))) \op_\e(S_q(0;T(\e))) v(0).$$
We control ${\rm II}$ by Corollary \ref{cor:chi} (action of $\op_\e(\chi^\flat)$ in $L^2$), Lemma \ref{lem:crude} (bounds for $S_q$ and their derivatives) and bound \eqref{bd:action:H} from Proposition \ref{prop:action}. This gives  
$$  \| {\rm II} \|_{L^2(B(0,\delta))} \lesssim \e^{K + h - \zeta} \exp\big( {\bm \g}^+  T(\e)^{\ell + 1}\big).$$
 The Duhamel term is 
 $$ {\rm III} = \e^h \int_0^{T(\e)} \op_\e(\chi^\flat(T(\e))) \op_\e(\Sigma(\t;T(\e)))  (\Id + \e \op_\e({\bf R}_0))( g  + \e \op_\e({\bf R}_0) v(0)) \, dt'.$$
We bound ${\rm III}$ with Corollary \ref{cor:so:new} and Lemma \ref{lem:source:psi}, in which we choose $n = K':$      
$$ \| {\rm III} \|_{L^2(B(0,\delta))} \lesssim  \e^{h - \zeta} \exp\big({\bm \g}^+ T(\e)^{1 + \ell}\big) \big( \e^{2K'} + \e^{1 + K} \big).$$
Above, we used the bounds of Corollary \ref{lem:new:Sq} in order to control the norm $\| \Sigma \|_{0,C(K')},$ which appears in the upper bound of Corollary \ref{cor:so:new}. Since $2 K' > K$ by condition \eqref{cond:K}, we obtained 
$$ \| \op_\e(\chi^\flat) v(T(\e)) \|_{L^2(B(0,\delta))} \gtrsim \e^{K - \zeta} \exp\big( {\bm \g}^-  T(\e)^{\ell + 1}\big) - \e^{K + h - \zeta} \exp\big( {\bm \g}^+  T(\e)^{\ell + 1}\big).$$
We may now conclude as in the proof of Corollary \ref{lem:low}. Precisely, rewriting the lower bound as 
$$ \|  \op_\e(\chi^\flat) v(T(\e)) \|_{L^2(B(0,\delta))} \gtrsim \e^{K - \zeta} \exp\big( {\bm \g}^-  T(\e)^{\ell + 1}\big) \big(1  - \e^{h - ({\bm \g}^+ - {\bm \g}^-) T_\star}\big),$$
and choosing $\delta$ small enough so that $h - ({\bm \g}^+ - {\bm \g}^-) T_\star > 0$ (recall that ${\bm \g}^+$ is the local maximum of the rate function $\g$ from Assumption \ref{ass:main} and ${\bm \g}^-$ is the local minimum of the lower rate function $\g^-$), we find for $\e$ small enough the lower bound
$$ \| \op_\e(\chi^\flat) v(T(\e)) \|_{L^2(B(0,\delta))} \geq C |\ln \e|^* \e^{K - \zeta - {\bm \g}^- T_\star},$$
for some $C > 0$ independent of $\e,$ where $|\ln \e|^*$ is some powers of $|\ln \e|.$  By choice of $T_\star$ in \eqref{cond:T1}, there holds $K - \zeta - {\bm \g}^- T_\star < 0$ if $\delta$ is small enough. Hence a lower bound which blows up as $\e \to 0,$ contradicting the a priori bound \eqref{bound:dotu}.

\section{Proof of Theorem \ref{th:0}: initial ellipticity} \label{sec:proof0}

 We are going to verify that, under the assumption of initial ellipticity \eqref{cond:ell}, Assumption \ref{ass:main} holds with parameters
\be \label{param:0}
 \ell = 0, \quad h= 1, \quad \zeta = 0, \quad \mbox{and} \quad \mu \equiv 0, \quad t_\star \equiv 0.
\ee
Then Theorem \ref{th:0} is a consequence of Theorem \ref{th:main}. 

\subsection{Block decomposition} \label{sec:blockdiag} Let $\l_0, \l_1, \dots,\l_p$ be the spectrum of $A$ at $(0,x_0,\xi_0).$ The ellipticity condition \eqref{cond:ell} states that one of the $\l_j$ is non real. By reality of $A,$ complex eigenvalues come in conjugate pairs. In particular, at least one of $\l_j$ has strictly positive imaginary part. We may assume $\Im m \, \l_0 > 0$ and 
\be \label{ineg:l0}  \Im m \, \l_0 > \max_{1 \leq j \leq p} \Im m \, \l_j.\ee
Let $m$ be the algebraic multiplicity of $\l_0$ in the spectrum, and $E_0(t,x,\xi)$ the generalized eigenspace associated with the family 
$\l_{0,1}(t,x,\xi), \dots, \l_{0,m}(t,x,\xi)$ of (possibly non distinct) eigenvalues of $A$ which coalesce at $(0,x_0,\xi_0)$ with value $\l_0,$ that is $\l_{0,j}(0,x_0,\xi_0) = \l_0.$ Let $E_1(t,x,\xi)$ be the direct sum of the other generalized eigenspaces. Let $Q_0$ be the projector onto $E_0$ and parallel to $E_1.$ The $\l_{0,j}$ may not be smooth, but are continuous (see for instance Proposition 1.1 in \cite{Tp}), and $Q_0$ is smooth (see for instance \cite{K}, or Proposition 2.1 in \cite{Tp}), and determines a smooth change of basis $Q(t,x,\xi)$ in which $A$ is block diagonal:
\begin{equation} \label{prop:0} 
 Q A Q^{-1} =: \left(\begin{array}{cc} A_{(0)} & 0 \\ 0 & A_{(1)} \end{array}\right).
 \end{equation}
We focus on the block $A_{(0)}$ associated with eigenvalues $\l_{0,j}.$ The symbolic flow accordingly splits into $S_{(0)},$ $S_{(1)},$ where $S_{(0)}$ solves
\be \label{S000}
 \d_t S_{(0)} + i A_{(0)}(\e t, x_0 + x,\xi) S_{(0)} = 0, \qquad S_{(0)}(\t;\t) = \Id.
\ee
 By \eqref{ineg:l0} and a repeat of the arguments below, the component $S_{(1)}$ of the symbolic flow is seen to grow not as fast as $S_{(0)}$ near $(0,\xi_0).$ 

\subsection{Reduction to upper triangular form at the distinguished point} Let now $P$ be a {\it constant} (independent of $t,x,\xi$) change of basis to upper triangular form of $A_{(0)}(0,x_0,\xi_0),$ and $Q_\mu$ be the diagonal matrix 
$$Q_\mu = \mbox{diag}(1,\mu^{-1},\mu^{-2},\dots,\mu^{1-m}).$$
The parameter $\mu$ will be chosen small enough below. There holds $$ Q_\mu P A_{(0)}(0,x_0,\xi_0) P^{-1} Q_\mu^{-1} = \l_0 \Id + \mu J,$$
where $J$ is upper triangular, bounded in $\mu,$ with zeros on the diagonal. By a Taylor expansion of $A_{(0)}(\e t, x_0 + x,\xi)$ in $(t,x,\xi),$ we observe that there holds
\be \label{A001} i A_{(0)}(\e t,x_0 + x,\xi) = P^{-1} Q_\mu^{-1} \big(  i\l_0 \Id + i \mu J + B(\e,t,x,\xi)\big) P Q_\mu ,
\ee
where the Taylor remainder $B$ has form
$$ B = \e t B_1(\e,t,x,\xi) + (x,\xi - \xi_0) \cdot B_2(\e,t,x,\xi),$$
 with $B_1$ and $B_2$ bounded, uniformly in $\e,$ in domain
  \be \label{domain}
  |x| \leq \delta, \quad |\xi - \xi_0| \leq \delta, \quad t \leq T_\star |\ln \e|.  \ee 
 Let
\be \label{tildeS000} 
\tilde S(\t;t) := \exp\Big( i (t - \t) \l_0 \Big) P Q_\mu S_{(0)}(\t;t).
\ee
Then $\tilde S$ solves 
\be \label{eq:tildeS000}
 \d_t \tilde S + (i \mu J + B) \tilde S = 0, \quad \tilde S(\t;\t) = P Q_\mu.
\ee

\subsection{Bounds for the symbolic flow} 
Consider the Hermitian matrix
 $$ \Re e \, (i \mu J + B) := \frac{1}{2} \ \big( (i \mu J + B) + (i\mu J + B)^*\big).$$ 
 Its eigenvalues $\l$ are semisimple, and vanish at $(\mu,t,x,\xi) = (0,0,0,\xi_0),$ hence satisfy, for $(t,x,\xi)$ in domain \eqref{domain}, the bound (see for instance Corollary 3.4 in \cite{Tp}): 
 $$ |\l| \leq c_0 \big(\mu + \e T_\star |\ln \e| + |x| + |\xi - \xi_0|\big) =: \g_0(\e,\mu,t,x,\xi),$$
 for some $c_0 > 0$ independent of $\e,\mu,t,x,\xi,$ for $(t,x,\xi)$ in \eqref{domain} and $(\e,\mu)$ small enough. Thus 
 \be \label{bd:Hermit}
 - \g_0 \Id \leq \Re e \, (i \mu J + B) \leq \g_0 \Id.
\ee
Let
$$ S_\pm := \exp\big( \pm (t - \t) \g_0(\e,\mu,t,x,\xi)\big) \tilde S.$$
From \eqref{eq:tildeS000} we deduce, for any fixed vector $\vec e \in \C^m,$ %
\be \label{eq:pm} \frac{1}{2} \Re e \, ( \d_t S_\pm \vec e \, , S_\pm \, \vec e )_{\C^m} + \big( \big( \Re e \, ( i \mu J + B) \pm \g_0 \big) S_\pm \vec e \,, S_\pm \vec e \,\big)_{\C^m} = 0.\ee
By \eqref{bd:Hermit}, this implies, for $\vec e$ unitary,  
 \be \label{bounds:tildeS} |P Q_\mu| e^{ - (t - \t) \g_0} \leq |\tilde S(\t;t) \vec e\,| \leq |P Q_\mu| e^{(t - \t) \g_0}  \ee 
Back to $S_{(0)}$ via \eqref{tildeS000}, we now have 
$$ %
  |S_{(0)}(\t;t,x,\xi)| \leq |P Q_\mu| |(P Q_\mu)^{-1}| e^{(t -\t) (\Im m \, \l_0 + \g_0)}.
$$ %
 We now choose $\mu = |\ln \e|^{-1},$ and let 
 \be \label{def:g+:0}
 \g^+ := \Im m \, \l_0 + c_0(|x| + |\xi - \xi_0|).
 \ee
 We obtained, for $\e$ small enough, 
 $$|S_{(0)}(\t;t)| \leq |\ln \e|^{\star} e^{(t - \t) \g^+},$$
 corresponding to the upper bound \eqref{ass:up}. Finally, since $|(P Q_\mu)^{-1} \vec e\,| \geq c_1 \mu^{1-m} |\vec e\,|$ for some $c_1 > 0$ and all $\vec e\,,$ we deduce from \eqref{bounds:tildeS}, the lower bound
 $$ |S_{(0)}(\t;t) \vec e \, | \gtrsim e^{(t - \t) \g^-}, \quad \g^- = \Im m \, \l_0 - c_0 (|x| + |\xi - \xi_0|),$$
 corresponding to the lower bound \eqref{ass:low}.

\section{Proof of Theorem \ref{th:1/2}: non semi-simple defect of hyperbolicity} \label{sec:proof1/2}

It suffices to verify that, under the assumptions of Theorem \ref{th:1/2}, Assumption \ref{ass:main} holds, with parameters
$$ \ell = 1/2, \quad h= 2/3, \quad \zeta = 1/3.$$
Then Theorem \ref{th:1/2} appears as a consequence of Theorem \ref{th:main}. We may assume initial hyperbolicity \eqref{init:hyp}, since otherwise Theorem \ref{th:0} applies. 

\subsection{The branching eigenvalues} \label{sec:branch}

 Let $\o_0 := (x_0,\l_0,\xi_0).$ By assumption, there holds 
 \begin{equation} \label{ass-for-tfi} P(0,\o_0) =0, \qquad \d_\l P(0,\o_0) = 0, \qquad \d_t P(0,\o_0) \neq 0, \qquad \d_\l^2 P(\o_0) \neq 0.
 \end{equation} 
 By the second and fourth conditions in \eqref{ass-for-tfi} and the implicit function theorem, there exists a smooth function $\mu_\star,$ with $\mu_\star(0,x_0,\xi_0) = \l_0,$ such that $\d_\l P = 0$ is equivalent to $\l = \mu_\star(t,x,\xi),$ for $(t,x,\xi)$ close to $(0,x_0,\xi_0).$ 
 
By the first three conditions in \eqref{ass-for-tfi} and the implicit function theorem, there is a smooth $\t_\star,$ with $\t_\star(x_0,\xi_0) = 0,$ such that $P(\mu_\star) = 0$ is locally equivalent to $t = \t_\star(x,\xi).$

  We now use the above implicitly defined functions $\mu_\star$ and $\t_\star$ to describe the spectrum of $A$ near $(t,x,\xi,\l) = (0,\o_0).$  
  
  \begin{lem} \label{lem:sp-1402} In a neighborhood of $(0,\o_0),$ there holds $P = 0$ if and only if 
 \begin{equation} \label{eq:branch} (\l - \mu(x,\xi))^2 = -(t - \t_\star(x,\xi)) e(t,x,\xi,\l),
 \end{equation}
 where $\mu(x,\xi) := \mu_\star(\t_\star(x,\xi),x,\xi)$ and $e$ is smooth and satisfies $e(0,x_0,\xi_0,\l_0) > 0.$ 
\end{lem}

\begin{proof} Given $t$ close to $0$ and $(x,\xi,\l)$ close to $\o_0,$ we expand the characteristic polynomial: 
 \begin{equation} \label{eq:P} \begin{aligned} P(t,x,\xi,\l) & = P(\t_\star(x,\xi),x,\xi, \mu_\star(\t_\star(x,\xi)),x,\xi)) \\ & \quad + (t - \t_\star(x,\xi)) e_1(t,x,\xi) + (\l - \mu(x,\xi))^2 e_2(t,x,\xi,\l) \\ & = (t - \t_\star(x,\xi)) e_1(t,x,\xi) + (\l - \mu(x,\xi))^2 e_2(t,x,\xi,\l),\end{aligned}
  \end{equation}
since $P(\t_\star,\cdot,\mu_\star(\t_\star)) \equiv 0,$ with 
 $$ \begin{aligned} e_1(t,x,\xi) & := \int_0^1 (\d_t P)((1 - \t) \t_\star(x,\xi) + \t t ,x,\mu(x,\xi)) \, d\t, \\ e_2(t,x,\xi,\l) & : = \int_0^1 (1 - \t) (\d_\l^2 P)(t,x,\xi,(1 - \t)\mu(x,\xi) + \t \l) \, d\t.\end{aligned}$$
We let $e := e_1 e_2^{-1}.$ Then $e(0,x_0,\xi_0,\l_0) > 0,$ as a consequence of the definition of $e_1$ and $e_2$ and condition \eqref{cond:coal2}. The result follows from \eqref{eq:P}.
\end{proof}

 Equation \eqref{eq:branch} describes a pair of eigenvalues branching at $t = \t_\star(x,\xi)$ out of the real axis, with imaginary parts growing like $(t - \t_\star)^{1/2}.$ The time curve $t = \t_\star(x,\xi)$ is the boundary between the hyperbolic region $t < \t_\star(x,\xi)$ in which the eigenvalues are real, and the elliptic region $t > \t_\star(x,\xi)$ in which the eigenvalues are not real and where we expect to record an exponential growth for the symbolic flow. In the introduction, Figure 3 pictures the hyperbolic and elliptic zones in the $(t,x,\xi)$ domain near $(0,x_0,\xi_0).$ 
We note that under the assumptions of Theorem {\rm \ref{th:1/2}}, the above defined time transition function $\t_\star$ satisfies 
  \begin{equation} \label{nabla-t} 
 \t_\star \geq 0, \quad \t'_\star(x_0,\xi_0) = 0, \quad \t''_\star(x_0,\xi_0) \geq 0, %
 \end{equation}
 where $\t_\star'(x_0,\xi_0)$ is the differential at $(x_0,\xi_0)$ and $\t''_\star(x_0,\xi_0)$ the Hessian. Indeed, if the first condition in \eqref{nabla-t} were violated, then hyperbolicity would not hold at $t = 0,$ contradicting \eqref{init:hyp}. Thus $0$ is a global mininum for $\t_\star,$ and \eqref{nabla-t} ensues.

\begin{rem} \label{rem:decaying:airy} Note that under assumption $\d_\l^2 P \d_t P < 0,$ eigenvalues stay real for small $t > 0,$ by Lemma {\rm \ref{lem:sp-1402}}.
\end{rem}

\subsection{Change of basis} \label{sec:prep}

By Lemma \ref{lem:sp-1402}, the eigenvalues $\l_\pm(t,x,\xi)$ of $A(t,x,\xi) - \mu(x,\xi) {\Id}$ satisfy
 \be \label{lpm:1} \l_\pm(t,x,\xi)^2 = - \big(t - \t_\star(x,\xi)\big) e\big(t,x,\xi, \mu(x,\xi) + \l_{\pm}(t,x,\xi)\big),\ee
 where $e$ is smooth in all its arguments, and $e(0,x_0,\xi_0, \mu(x_0,\xi_0)) > 0,$ with $\mu(x_0,\xi_0) = \l_0.$ From \eqref{lpm:1} and continuity of $e,$ we deduce the fact that $\l_-$ and $\l_+$ are purely imaginary, hence $\l_+ + \l_- = 0,$ since the matrix $A - \mu {\Id}$ has real coefficients. From \eqref{lpm:1}, we also deduce the bound, for some $C > 0,$ locally around $(0,x_0,\xi_0),$
 $$ |\l_\pm| \leq C |t - t_\star|^{1/2},$$
  which we may plug back in \eqref{lpm:1} and deduce, by regularity of $e,$ 
 \be \label{lpm:2} 
 \l_\pm = \pm i (t - \t_\star)^{1/2} e(t,x,\xi,\mu(x,\xi))^{1/2} + O(t - \t_\star).
 \ee
  We now reduce $A$ to canonical form:
 
 \begin{lem} \label{lem:1402-2}
 There exists a smooth change of basis $Q$ such that locally around $(0,x_0,\xi_0),$ 
 \begin{equation} \label{QAQ} Q(t,x,\xi) \big(A(t,x,\xi) - \mu(x,\xi) \Id\big) Q(t,x,\xi)^{-1} = \left(\begin{array}{cc} A_{(0)} & 0 \\ 0 & A_{(1)} \end{array}\right),\end{equation}
 where $A_{(0)} = \left(\begin{array}{cc} 0 & 1 \\  - ( t - \t_\star)e_0 + O(t - \t_\star)^{3/2}  & 0 \end{array}\right),$ and $A_{(1)} \in \C^{(N-2) \times (N-2)}$ is smooth. In the bottom left entry of $A_{(0)},$ the function $\t_\star$ is the time transition function introduced just above Lemma {\rm \ref{lem:sp-1402}}, and we denote $e_0(t,x,\xi) = e(t,x,\xi, \mu(t,x,\xi)),$ with $e$ as in Lemma {\rm \ref{lem:sp-1402}.}\end{lem}
 
 \begin{proof} We may smoothly block diagonalize $A - \mu,$ for instance as described in Section \ref{sec:blockdiag}. The block associated with $\l_0$ is size two, since the multiplicity of $\l_0$ is equal to two \eqref{cond:coal}. Thus for some smooth $\tilde Q,$ there holds 
  $\tilde Q( A - \mu) \tilde Q^{-1} = \left(\begin{array}{cc} B_0 & 0 \\ 0 & A_{(1)} \end{array}\right),$
  where $B_0$ is the $2 \times 2$ matrix $B_0 = \left(\begin{array}{cc} a_{11} & a_{12} \\ a_{21} & a_{22} \end{array}\right),$ with smooth entries $a_{ij}.$ The spectrum of $B_0$ is $\{ \l_-, \l_+\},$ where $\l_\pm$ satisfy \eqref{lpm:2}. Since, as noted above, there holds $\l_- = - \l_+,$ the trace of $B_0$ is zero, that is $a_{22} = - a_{11}.$ Besides, there holds $a_{21} a_{12} \neq 0$ at $(0,x_0,\xi_0).$ Indeed, if $a_{21} a_{12}(0,x_0,\xi_0) = 0,$ the spectrum of $B_0$ would be smooth in time, contradicting \eqref{eq:branch}. Without loss of generality, we assume $a_{21}(0,x_0,\xi_0) \neq 0.$ Then $\dsp{Q_0 = \left(\begin{array}{cc} 0 & - a_{21}^{-1}\\[2pt] 1  & - a_{21}^{-1} a_{11} \end{array}\right)}$ is a smooth change of basis such that 
   $Q_0 B_0 Q_0^{-1} = \left(\begin{array}{cc} 0 & 1 \\ \star & 0 \end{array}\right).$
 The bottom left entry of $B_0$ is equal to the product $\l_- \l_+$ of its eigenvalues. By \eqref{lpm:2}, we find that $\l_- \l_+ = - (t - \t_\star) e_0 + O(t - \t_\star),$ and the result holds with $Q = \left(\begin{array}{cc} Q_0 & 0 \\ 0 & {\rm Id}_{\C^{N-2}} \end{array}\right) \tilde Q.$
  \end{proof}

\subsection{The symbolic flow} \label{sec:flow1/2}

Our goal is to prove the bounds of Assumption \ref{ass:main} for solution $S$ to the ordinary differential equation 
\begin{equation} \label{tilde-fund-in-proof}
  \d_t S(\t;t) + i \e^{-1/3} A_\star(t) S(\t;t)  = 0,
  \qquad S(\t;\t)  \equiv \Id,
  \end{equation}
 where $A_\star$ is defined by \eqref{def:Astar}, which we reproduce here: 
  \be \label{Astar:1/2} A_\star(t) = \big(Q \big(A - \mu \Id\big) Q^{-1}\big)\big(\e^{2/3} t, x_0 + \e^{1/3} x_\star(\e^{2/3} t, x,\xi), \xi_\star(\e^{2/3} t, x, \xi) \, \big).\ee
 Recall that $h = 2/3$ and $\zeta = 1/3$ here.  %
   The change of basis $Q$ is given by Lemma \ref{lem:1402-2}, the real part of the branching eigenvalues $\mu$ is introduced in Lemma \ref{lem:sp-1402}, and $(x_\star,\xi_\star)$ is the bicharacteristic flow, solving \eqref{bichar}, which we reproduce here: 
  \be \label{bichar:airy} \d_t x_\star = - \d_\xi \mu(t, x_0 + \e^{1/3} x_\star, \xi_\star), \quad \d_t \xi_\star = \e^{1/3} \d_x \mu(t, x_0 + \e^{1/3} x_\star, \xi_\star).\ee
The block decomposition of $A$ given by Lemma \ref{lem:1402-2} induces a block decomposition of $A_{\star}.$ We focus on the top left block $A_{(0)\star}$ in $A_\star.$ As per Lemma \ref{lem:1402-2}, its bottom left entry involves the function
 \be \label{t:starstar} \t_{\star\star}(\e,t,x,\xi) := \t_\star\big(x_0 + \e^{1/3} x_\star(\e^{2/3} t, x,\xi), \xi_\star(\e^{2/3} t, x, \xi)\,\big).\ee 
 The bicharacteristic flow \eqref{bichar:airy} satisfies
 \be \label{for:bichar} x_0 + \e^{1/3} x_\star(\e^{2/3} t, x,\xi) = x_0 + \e^{1/3} x + O(\e t), \quad \xi_\star(\e^{2/3} t, x,\xi) = \xi + O(\e),\ee
 uniformly in $(x,\xi)$ with $|x| + |\xi - \xi_0| \leq \delta.$ Here notation $O(\e)$ refers to a uniform bound of form $\lesssim.$ In particular for $t$ bounded from above by a power of $|\ln \e|,$ there holds $O(\e t) = O(\e).$  
 Thus $\t_{\star\star}$ defined in \eqref{t:starstar} satisfies
$$ %
\t_{\star\star} = \theta_\star(\e^{1/3} x,\xi) + O(\e), \quad \theta_\star(x,\xi) := \t_{\star}(x_0 + x, \xi).
$$ %
 By \eqref{nabla-t}, we see that $\theta_\star$ as defined above satisfies the conditions of equation \eqref{def:tstar}. In accordance with \eqref{def:tstar}, we let $t_\star(\e,t,x,\xi) := \e^{-2/3} \theta_\star(\e^{1/3} x, \xi).$ %
 There holds
 \be \label{for:a0star} \e^{-1/3} (\e^{2/3} t - \t_{\star\star}) = \e^{1/3} (t - \e^{-2/3} \theta_\star(\e^{1/3} x, \xi)) + O(\e^{2/3}) = \e^{1/3}(t - t_\star) + O(\e^{2/3}),\ee
 uniformly in $(x,\xi)$ with $|x| + |\xi - \xi_0| \leq \delta.$ 
 
 \begin{lem} \label{lem:astar} The top left block $A_{(0)\star}$ of $A_{\star}$ in the block decomposition of Lemma {\rm \ref{lem:1402-2}} satisfies 
 $$\e^{-1/3} A_{(0)\star} = \left(\begin{array}{cc} 0 & \e^{-1/3} \\ - \e^{1/3} (t - t_\star) f_0 + O(\e^{2/3}) & 0 \end{array}\right),$$
 with notation $f_0(\e,x,\xi) = e_0(0,x_0 + \e^{1/3} x, \xi),$ so that $f_0 > 0$ for $(x,\xi)$ near $(0,\xi_0).$ 
 \end{lem}
 
 \begin{proof} The bottom left entry of  $\e^{-1/3} A_{(0)\star}$ involves the function $e_0 = e(\mu)$ evaluated in the time-rescaled and advected frame. In view of \eqref{for:bichar}, there holds by regularity of $e$ and $\mu$
 $$ e_0\big(\e^{2/3} t, x_0 + \e^{1/3} x_\star(\e^{2/3} t, x,\xi), \xi_\star(\e^{2/3} t, x, \xi) \big) = e_0(0,x_0 + \e^{1/3} x, \xi) + O(\e^{2/3}).$$
 The bottom left entry of  $\e^{-1/3} A_{(0)\star}$ also involves $t - \t_\star$ and $(t - \t_\star)^{3/2}$ in the time-rescaled and advected frame. In view of \eqref{for:a0star}, these functions contribute to $\e^{-1/3} A_{(0)\star}$ 
 $$ \e^{-1/3}(\e^{2/3}t - \t_{\star\star}) = \e^{1/3}(t - t_\star) + O(\e^{2/3})$$
 and $\e^{-1/3}(\e^{2/3}t - \t_{\star\star})^{3/2} = O(\e^{2/3}).$ We may conclude with Lemma \ref{lem:1402-2}.
 \end{proof}

By Lemma \ref{lem:astar}, the flow $S_{(0)}$ of $i \e^{-1/3} A_{\star (0)}$ solves 
\be \label{def:S_0}
  \d_t S_{(0)}  + i \e^{-1/3} \left(\begin{array}{cc} 0 & \e^{-1/3} \\ - \e^{1/3} (t - t_\star) f_0 & 0 \end{array}\right) S_{(0)} = \e^{2/3} C S_{(0)}, \quad S_{(0)}(\t;\t)  = \Id,
  \end{equation} 
  where $C:= C(\e,t,x,\xi) = \left(\begin{array}{cc} 0 & 0 \\ c & 0 \end{array}\right),$ with $|\d_x^\a \d_\xi^\b c| \lesssim 1$ for $0 \leq t \leq T(\e)$ and $|x| + |\xi - \xi_0| \leq \delta.$ The coefficient $f_0 = f(\e,x,\xi)$ satisfies $f_0 > 0$ for $|x| + |\xi - \xi_0| \leq \delta.$ 
\subsection{Reduction to a perturbed Airy equation} \label{sec:airy0}

Let 
\be \label{def:D:matrix}
 D(x,\xi) :=\left(\begin{array}{cc} -i \e^{1/3} f_0(x,\xi)^{1/3}   & 0 \\ 0 & 1 \end{array}\right),
\ee %
so that $D$ is well-defined and invertible on $|x| < \delta,$ $|\xi - \xi_0| \leq \delta,$ and 
\be \label{def:Z}
 Z(\t;t) := D S_{(0)}\big( f_0^{-1/3} \t + t_\star\, ; f_0^{-1/3} t + t_\star\,\big),
\ee
where $f_0,$ $t_\star,$ $D,$ $S_{(0)}$ and $Z$ all depend on $(x,\xi).$ For future use, we note that
\be \label{computation:D}
 D(x,\xi) \left(\begin{array}{cc} z_{11} & z_{12} \\ z_{21} & z_{22} \end{array}\right) D(x,\xi)^{-1} = \left(\begin{array}{cc} z_{11} & - i (\e f_0)^{1/3} z_{12} \\ i (\e f_0)^{-1/3} z_{21} & z_{22} \end{array}\right)
\ee
and
\be \label{computation:D-1}
 D(x,\xi)^{-1} \left(\begin{array}{cc} z_{11} & z_{12} \\ z_{21} & z_{22} \end{array}\right) D(x,\xi) = \left(\begin{array}{cc} z_{11} & i (\e f_0)^{-1/3} z_{12} \\ - i (\e f_0)^{1/3} z_{21} & z_{22} \end{array}\right).
 \ee
\begin{lem} \label{lem:1402-3} On $|x| \leq \delta,$ $|\xi - \xi_0| \leq \delta,$ the map $Z$ satisfies the perturbed Airy equation

 \begin{equation} \label{airy-app}
  Z' +  \left(\begin{array}{cc} 0 & 1 \\ t & 0 \end{array}\right) Z = \e^{1/3} \tilde C Z, \qquad Z(\t;\t) = D,
   \end{equation}
  where $\tilde C := (D C D^{-1})(f_0^{-1/3} t + t_\star),$ with $C$ as in \eqref{def:S_0}.  
\end{lem}

\begin{proof} It suffices to use \eqref{computation:D} and observe that, by Lemma \ref{lem:astar}, 
 $$ \frac{i \e^{-1/3}}{f_0(x,\xi)^{1/3}} D(x,\xi) A_{\star (0)}\left(\frac{t}{f_0(x,\xi)^{1/3}} + t_\star\,, \, x, \, \xi\right) D(x,\xi)^{-1} = \left(\begin{array}{cc} 0 & 1 \\ t & 0 \end{array}\right) - \e^{1/3} \tilde C.$$
 \end{proof} 
From the above, we will deduce lower and upper bounds for $S_{(0)},$ by comparison with the vector Airy function ${\bf Z},$ defined as the solution of 
 \begin{equation} \label{bfZ} {\bf Z}' + \left(\begin{array}{cc} 0 & 1  \\  t &  0\end{array}\right) {\bf Z} = 0, \qquad {\bf Z}(\t;\t) = \Id.
 \end{equation}

\subsection{Bounds for the Airy function} \label{sec:airy}

We will use \eqref{airy-app} to show that the symbolic flow grows in time like the Airy function, for which the following is known (see for instance \cite{H1}, chapter 7.6). 

  \begin{lem}[Airy equation] \label{lem:ai}
  Let ${\rm Ai}$ be the inverse Fourier transform of 
 $e^{i \xi^3/3}$,
 and $j= e^{2 i \pi/3}.$ The functions ${\rm Ai},$ ${\rm Ai}(j \cdot)$ form a basis of solutions of the ordinary differential equation
 $ y'' = t y,$ %
  and there holds
  \begin{align} \label{ai-bd} {\rm Ai}(t) &=  \frac{1}{2\sqrt{\pi}}e^{-(2/3) t^{3/2}} t^{-1/4} (1 + O(t^{-3/2})), \quad t\rightarrow+\infty, 
  \\
  \label{ai-bd-} {\rm Ai}(-t) &= \frac{1}{\sqrt{\pi}}t^{-1/4} \Bigl(\sin\bigl(\frac23 t^{3/2}+\frac\pi4\bigr) + O(t^{-3/2})\Bigr), \quad t\rightarrow+\infty,\\
  \label{ai-bd2} {\rm Ai}( j t) &= \frac{1}{2\sqrt{\pi}}e^{-i \pi/6} e^{(2/3) t^{3/2}} t^{-1/4} (1 + O(t^{-3/2})),\quad t\rightarrow+\infty,\\
  \label{ai-bd--}  {\rm Ai}(- j t) &= \frac{1}{2\sqrt{\pi}} e^{i \pi/6} e^{(2/3) i t^{3/2}} t^{-1/4} (1 + O(t^{-3/2})),\quad t\rightarrow+\infty.
  \end{align}
   \end{lem}
 From the above Lemma, we deduce uniform bounds for the time derivative ${\rm Ai}':$ 
$$e^{(2/3) t^{3/2}} |{\rm Ai}'(t)| + e^{-(2/3) t^{3/2}} |{\rm Ai}'(jt)| + |{\rm Ai}'(-t)| + |{\rm Ai}'(-j t)| \leq C (1 + t)^{1/4}.$$
  for some $C > 0,$ for all $t \geq 0.$ 
 By Lemma \ref{lem:ai}, the solution to \eqref{bfZ} is 
 $$ {\bf Z}(\t;t) = \frac{1}{W(\t)} \left(\begin{array}{cc} - j {\rm Ai}'(j \t) {\rm Ai}(t) + {\rm Ai}'(\t) {\rm Ai}(j t) &  - {\rm Ai}(j \t) {\rm Ai}(t) + {\rm Ai}(\t) {\rm Ai}(j  t) \\ j {\rm Ai}'(j \t) {\rm Ai}'(t) - j {\rm Ai}'(\t) {\rm Ai}'(j t) & {\rm Ai}(j \t) {\rm Ai}'(t) - j {\rm Ai}(\t) {\rm Ai}'(j t)\end{array}\right),$$
 where $W$ is the Wronskian, satisfying 
$$W(\t) := {\rm Ai}(j \t) {\rm Ai}'(\t) - j {\rm Ai}'(j \t) {\rm Ai}(\t) \equiv \frac{1}{4\pi}(-\sqrt 3+i).$$
The bounds for Ai and $\mbox{Ai}'$ imply the upper bound, for $0 \leq \t \leq t:$ 
\begin{equation} \label{bd:Airy}
 |{\bf Z}(\t;t)| \leq C (1 + |\t|)^{1/4}(1 + |t|)^{1/4} e_{\rm Ai}(\t;t),
\end{equation}
and the lower bound
\begin{equation} \label{low:Airy}
 \left| \left(\begin{array}{cc} 1 & 0 \end{array}\right) {\bf Z}(0;t) \left(\begin{array}{c} 0 \\ 1 \end{array}\right)\right| \geq c \, e_{\rm Ai}(0;t),
 \end{equation}
for some $c > 0$ independent of $\t,t,$ where the growth function $e_{\rm Ai}$ is defined by  
\begin{equation} \label{def:growth}
 {e}_{\rm Ai}(\t;t) = \exp\Big( \frac{2}{3} \big(t_+^{3/2} - \t_+^{3/2}\big)\Big), \qquad x_+ := \max(x,0).
 \end{equation}
 We note that $e_{\rm Ai}$ is multiplicative:
\begin{equation} \label{prop:e}
 {e}_{\rm Ai}(\t;t') {e}_{\rm Ai}(t';t) = {e}_{\rm Ai}(\t;t), \qquad \mbox{for all $\t,t',t'.$}
 \end{equation}

\begin{rem} \label{rem:decaying:airy2} If we had assumed $\d_\l^2 P \d_t P < 0,$ then we would have had to consider the Airy condition for negative times. Lemma {\rm \ref{lem:ai}} would then have yielded polynomial bounds for the symbolic flow.     
\end{rem}

\subsection{Bounds for the symbolic flow} \label{sec:verif-hyp}

Let 
\be \label{def:Theta} \Theta(t,x,\xi) := f_0(x,\xi)^{1/3} (\t - t_\star(\e,x,\xi)).\ee %
Our goal is to verify the bounds of Assumption \ref{ass:main} for $S_{(0)}$ in the elliptic domain ${\mathcal D}$ defined in \eqref{def:mathcalD}. We reproduce here the definition of ${\mathcal D}:$ 
$$ {\mathcal D} := \big\{ (\t;t,x,\xi), \quad t_\star(\e,x,\xi) \leq \t \leq t \leq T(\e),  \quad |x| \leq \delta, \quad |\xi - \xi_0| \leq \delta \e^{1/3} \big\}.$$

\begin{lem} \label{lem:Zflat} There holds the bounds in domain ${\mathcal D}:$  
$$ %
\big| D^{-1} {\bf Z}(\Theta(\t);\Theta(t)) D\big| \lesssim \left(\begin{array}{cc} 1 & \e^{-1/3} \\ \e^{1/3} & 1 \end{array}\right) e_{\rm Ai}(\Theta(\t);\Theta(t)),
$$ %
 with $e_{\rm Ai}$ defined in \eqref{def:growth}. 
\end{lem}

 Above, $\lesssim$ means {\it entry-wise} inequality ``modulo constants", as defined in \eqref{notation:lesssim}.
\begin{proof} There holds $0 \leq \Theta(\t) \leq \Theta(t) \lesssim \Theta(T(\e))$ in domain ${\mathcal D}.$ Bound \eqref{bd:Airy} states that there holds $|{\bf Z}(\Theta)| \lesssim e_{\rm Ai}(\Theta).$ Then \eqref{computation:D-1} implies the result.
\end{proof}

From Lemma \ref{lem:Zflat} we now derive bounds for $S_{(0)}.$ Given that $S_{(0)}$ is expressed in terms of $Z$ \eqref{def:Z} and that $Z$ is a perturbation of ${\bf Z},$ we find ourselves in a situation very much like the one encountered in Section \ref{sec:bd-sol-op}. Accordingly, the proof of the following Corollary borrows from Section \ref{sec:bd-sol-op}, in particular from the proof of Corollary \ref{lem:s1} and Corollary \ref{lem:Spsi}. 

\begin{cor} \label{cor:bd:S00} 
The flow $S_{(0)}$ of the top left block $A_{\star (0)}$ in $A_\star,$ solution of \eqref{def:S_0}, satisfies the bounds:
 \be \label{bd:S00}
  | S_{(0)}| \lesssim \left(\begin{array}{cc} 1 & \e^{-1/3} \\ \e^{1/3} & 1 \end{array}\right) e_{\rm Ai}(\Theta), 
 \ee
\end{cor}

\begin{proof} By Lemma \ref{lem:1402-3} and definition of ${\bf Z}$ \eqref{bfZ}, there holds
$$ Z(\t;t) = {\bf Z}(\t;t) D + \e^{1/3} \int_\t^t {\bf Z}(t';t) \tilde C(t') Z(\t;t') \, dt'.$$
By definition of $Z$ \eqref{def:Z}, there holds $S_{(0)} = D^{-1} Z(\Theta).$ Thus
$$ S_{(0)}(\t;t)  = D^{-1} {\bf Z}(\Theta(\t); \Theta(t)) D + \e^{2/3} \int_{\Theta(\t)}^{\Theta(t)} D^{-1} {\bf Z}(t';\Theta(t)) \tilde C(t') Z(\Theta(\t);t') \, dt'.$$ 
Since $\tilde C$ is defined in Lemma \ref{lem:1402-3} to be equal to $(D^{-1} C D)(f_0^{-1/3} t + t_\star),$ we obtain
\be \label{rep:S:Z} \begin{aligned} S_{(0)}(\t;t) &  = D^{-1} {\bf Z}(\Theta(\t);\Theta(t)) D \\ & + \e^{2/3} \int_{\Theta(\t)}^{\Theta(t)} D^{-1} {\bf Z}(t';\Theta(t)) D C(f_0^{-1/3} t' + t_\star) S_{(0)}(\t;f_0^{-1/3} t' + t_\star) \, dt'. \end{aligned}
\ee
The change of variable $t' = \Theta(\t'),$ corresponding to $\t' = f_0^{-1/3}(t' + t_\star)$ transforms the above integral into
$$ \e^{2/3} f_0^{1/3} \int_\t^t D^{-1} {\bf Z}(\Theta(\t'); \Theta(t)) D C(\t') S_{(0)}(\t; \t') \, d\t'.$$ We now factor out the expected growth in view of applying Gronwall's lemma, as we did before in the proof of Corollary \ref{lem:s1}: we let 
$$S_{(0)}^\flat := e_{\rm Ai}(\Theta)^{-1} S_{(0)}, \quad \mbox{and} \quad {\bf Z}^\flat(\t';t) := e_{\rm Ai}(\Theta)^{-1} D^{-1} {\bf Z}(\Theta(\t'); \Theta(t)) D.$$
By the multiplicative property \eqref{prop:e} of the growth function $e_{\rm Ai},$ we find
 $$ S_{(0)}^\flat(\t;t)  =  {\bf Z}^\flat(\t;t) + \e^{2/3} f_0 \int_\t^t {\bf Z}^\flat(\t';t) C(\t') S_{(0)}^\flat(\t';t) \, d\t'.$$
We now rescale the top right and bottom left entries, as we consider the equation in $\underline{S}_{(0)}^\flat,$ with notation introduced just above \eqref{cancel:underline} in the proof of Corollary \ref{lem:Spsi}. In view of \eqref{cancel:underline}, there holds
\be \label{rep:S0:flat} \underline{S}_{(0)}^\flat(\t;t)  =  \underline {\bf Z}^\flat(\t;t) + \e^{1/3} f_0 \int_\t^t \underline {\bf Z}^\flat(\t';t) \big( \e^{1/3} \underline{C}(\t')\big) \underline{S}_{(0)}^\flat(\t';t) \, d\t'.\ee
There holds $\e^{1/3} \underline{C}(t) = O(t),$ and $t$ is bounded by some power of $|\ln \e|$ in ${\mathcal D}.$ Lemma \ref{lem:Zflat} implies that $|\underline {\bf Z}^\flat | \lesssim 1.$ Hence Gronwall's lemma implies the bound
$|\underline{S}_{(0)}^\flat(\t;t)| \lesssim 1,$
which corresponds precisely to \eqref{bd:S00}.
\end{proof}

\begin{lem} \label{lem:low:1/2} There holds the lower bound
 $$ \left| \left(\begin{array}{cc} 1 & 0 \end{array}\right) S_{(0)} \left(\begin{array}{c} 0 \\ 1 \end{array}\right)\right|  \geq c_0 \,\e^{-1/3} e_{\rm Ai}(\Theta),$$
 for some universal constant $c_0 > 0.$ 
\end{lem}

\begin{proof} Consider representation \eqref{rep:S0:flat}. We focus on the top right entry. The lower bound \eqref{low:Airy} for the vector Airy function states that the top right entry of ${\bf Z}$ is bounded from below by $e_{\rm Ai}.$ By \eqref{computation:D-1}, this implies
 \be \label{low:DZ} \left| \left(\begin{array}{cc} 1 & 0 \end{array}\right) D^{-1} {\bf Z} D \left(\begin{array}{c} 0 \\ 1 \end{array}\right) \right| \geq c_0 \e^{-1/3} e_{\rm Ai},\ee
 for some $c_0 > 0$ independent of $\t,t.$ Borrowing notation from the proof of Corollary \ref{cor:bd:S00}, this means that the top right entry of $\underline{\bf Z}^\flat$ is bounded away from zero, uniformly in time. We know from Corollary \ref{cor:bd:S00} that $|\underline{S}_{(0)}^\flat| \lesssim 1$ and $|\underline {\bf Z}^\flat| \lesssim 1.$ Thus from \eqref{rep:S0:flat} and \eqref{low:DZ} we deduce the result, since $t \lesssim 1$ in ${\mathcal D}.$ 
\end{proof}

We observe that there holds, for $e_{\rm Ai}$ defined in \eqref{def:growth} and $\Theta$ defined in \eqref{def:Theta}:
$$ e_{\rm Ai}(\Theta) \equiv {\bf e}_{\g}, \quad \mbox{with $\dsp{\g(x,\xi) := \frac{2}{3} f_0(x,\xi)^{1/2}}, \,\, t_\star = \e^{-2/3} \t_\star(x_0 + \e^{1/3} x, \xi),$}$$
where $\t_\star$ is given the implicit function theorem in Section \ref{sec:branch}.  
Hence Corollary \ref{cor:bd:S00} and Lemma \ref{lem:low:1/2} verify the bounds of Assumption \ref{ass:main} for $S_{(0)},$ with $\g^+ = \g^- = \g,$ and with $\vec e$ being equal to the constant vector $\left(\begin{array}{c} 0 \\ 1 \end{array}\right).$

In order to complete the verification of Assumption \ref{ass:main}, and thus conclude the proof of Theorem \ref{th:1/2}, it only remains to show that the other components of the symbolic flow do not grow faster than $S_{(0)}.$ This follows directly from the simplicity hypothesis in Theorem \ref{th:1/2}. Indeed, by the simplicity hypothesis, we may smoothly diagonalize the other component $A_{\star (1)}$ of $A_\star$ near $(0,x_0,\xi_0)$ (use for instance Corollary 2.2 in \cite{Tp}). The eigenvalues of $A_{\star (1)}$ are real near $(0,x_0,\xi_0).$ The equation for the symbolic flow of $A_{(\star (1)}$ splits into scalar differential equations, with purely imaginary coefficients. Thus the symbolic flow of $A_{\star (1)}$ is bounded.

\section{Proof of Theorem \ref{th:1}: smooth defect of hyperbolicity} \label{sec:proof1}

 It suffices to verify that, under the assumptions of Theorem \ref{th:1}, Assumption \ref{ass:main} holds with parameters
\be \label{param:1}
 \ell = 1, \quad h =  1/2, \quad \zeta = 0, \quad \mu = \Re e \, \l_\pm, \quad t_\star \equiv 0,
\ee
where $\l_\pm$ are the bifurcating eigenvalues, as given by Proposition \ref{prop:evalues}.

\subsection{Block decomposition} As in the proof of Theorem \ref{th:0}, we may smoothly block diagonalize $A$ by a change of basis $Q(t,x,\xi),$ for small $t$ and $(x,\xi)$ close to $(x_0,\xi_0).$ Then identity \eqref{prop:0} holds, and we focus on block $A_{(0)},$ of size two, such that 
 $$ \mbox{sp}\, A_{(0)}(0,x_0,\xi_0) = \{ \l_0 \}, \qquad \l_0 \in \R,$$
 where $(x_0,\xi_0,\l_0)$ are the coordinates of $\o_0 \in \G$ which intervenes in Hypothesis \ref{hyp}. By Hypothesis \eqref{hyp} and Proposition \ref{prop:evalues}, the eigenvalues $\l_\pm$ of $A_{(0)}$ branch out of the real axis at $t =0,$ for all $(x,\xi)$ in a neighborhood of $(x_0,\xi_0).$ We define $\mu$ to be the real part of these eigenvalues. The corresponding equation for the symbolic flow is 
\be \label{eq:S0:1} \d_t S_{(0)} + \e^{-1/2} (A_{(0)} - \mu)\big(\e^{1/2} t, x_0 + \e^{1/2} x_{\star}(\e^{1/2} t,x,\xi), \xi_{\star}(\e^{1/2} t,x,\xi) \big) S_{(0)} = 0,\ee
where $(x_\star,\xi_\star)$ are the bicharacteristics of $\mu.$  

\subsection{Time regularity and cancellation} 

 By Proposition \ref{prop:evalues}, the eigenvalues $\l_\pm$ are differentiable in time, at $t =0$ and for all $(x,\xi)$ near $(x_0,\xi_0).$ Indeed, Hypothesis \ref{hyp} implies that conditions \eqref{non-R} are satisfied in a whole neighborhood of $(x_0,\xi_0).$ We may thus write
\be \label{tilde:lpm} \begin{aligned}
 \l_\pm\big(\e^{1/2} t, x_0 + \e^{1/2} x_\star, \xi_\star\big) - \mu(0,x_0 + \e^{1/2} x_\star,\xi_\star) 
   = i \e^{1/2} t \tilde \l_\pm(\e,t,x,\xi) + o(\e^{1/2}),\end{aligned}\ee
uniformly in $t = O(|\ln \e|)$ and $(x,\xi)$ near $(x_0,\xi_0),$ where $(x_\star, \xi_\star)$ is evaluated at $(\e^{1/2} t, x, \xi),$ and where 
\be \label{tilde:lpm:2} \tilde \l_\pm(\e,0,x,\xi) = \d_t \Im m \, \l_\pm(0,x_0 + \e^{1/2} x_\star(0,x,\xi),\xi_\star(0,x,\xi)) \in \R.\ee
Consider the $2 \times 2$ matrix $A_{(0)}(0,x,\xi).$ It has one semi-simple eigenvalue $\mu(0,x,\xi)$ (the assumption of semisimplicity is part of Hypothesis \eqref{hyp}). Thus 
$$ A_{(0)}(0,x,\xi) = \mu(0,x,\xi) \Id.$$
In particular, by regularity of the entries of $A,$ 
\be \label{cancel:1} \e^{-1/2} A_{(0)}(\e^{1/2} t, x_0 + \e^{1/2} x_\star, \xi_\star) = t \tilde A_{(0)}(\e,0,x,\xi) + \e^{1/2} t^2 B(\e,t,x,\xi),\ee
where $B$ is uniformly bounded for $\e$ close to 0, $t = O(|\ln \e|^*)$ and $(x,\xi)$ close to $(x_0,\xi_0).$ Thus equation \eqref{eq:S0:1} takes the form
\be \label{eq:S0:new} \d_t S_{(0)} + t \tilde A_{(0)}(\e, 0, x,\xi) S_{(0)} = \e^{1/2} t^2 B(\e, t,x,\xi) S_{(0)}.\ee
The key cancellation that takes place in \eqref{cancel:1} transformed equation in $S_{(0)}$ into an autonomous equation with a small, linear, time-dependent perturbation. The eigenvalues of $\tilde A_{(0)}$ are $\tilde \l_\pm(\e,0,x,\xi)$ from \eqref{tilde:lpm}-\eqref{tilde:lpm:2}. These eigenvalues are distinct by Proposition \ref{prop:evalues}.

\subsection{Bounds for the symbolic flow.} 

The solution $S$ to 
$$ \d_t S + i t \tilde A_{(0)}(\e,0, x,\xi) S = 0, \qquad S(\t;\t) = \Id$$
is
$$ S(\t;t) = \exp \left( - i \tilde A_{(0)}(\e,0,x,\xi) (t^2 - \t^2)/2 \right).$$
The eigenvalues of $\tilde A_{(0)},$ being distinct, are smooth in $(\e,x,\xi)$ (see for instance Corollary 2.2 in \cite{Tp}). In particular, there holds
$$ \tilde \l_\pm(\e,0,x,\xi) = \tilde \l_\pm(0,0,x,\xi) + O(\e) = \Im m \, \d_t \l_\pm(0,x_0, \xi_\star(0,x,\xi))$$
locally uniformly in $(x,\xi).$ Let $\l_+$ be the eigenvalue with positive imaginary part, and 
\be \label{def:g:1}
 \g(x,\xi) := \frac{1}{2} \tilde \l_\pm(0,0,x,\xi) = \frac{1}{2} \Im m \, \d_t \l_+(0, x_0, \xi_\star(0,x,\xi)).
 \ee
 Then,
\be \label{bd:S:1}
 |S(\t;t,x,\xi)| \lesssim \exp\big( \g(x,\xi) (t^2 - \t^2)\big),
\ee
and, since $\tilde A_{(0)}$ is smoothly diagonalizable, for some smoothly varying vector $\vec e(x,\xi)$ there holds
\be \label{low:S:1}
|S(\t;t,x,\xi) \vec e(x,\xi)\,| \gtrsim \exp\big( \g(x,\xi) (t^2 - \t^2)\big).
\ee
Perturbation arguments already encountered in Section \ref{sec:bd-sol-op} (specifically, in the proof of Corollary \ref{lem:s1}) show that the bounds \eqref{bd:S:1}-\eqref{low:S:1} for $S$ yield similar bounds for the symbolic flow $S_{(0)}$ solution to \eqref{eq:S0:new}.
These bounds verify the upper and lower bound \eqref{ass:up} and \eqref{ass:low} from Assumption \ref{ass:main}. 

For the other components of the flow, we use the simplicity assumption in Theorem \ref{th:1}, as we did in the last paragraph of Section \ref{sec:verif-hyp} in the proof of Theorem \ref{th:1/2}. 

{
\section{Examples} \label{sec:examples}

\subsection{One-dimensional Burgers systems} \label{sec:Burgers} 
 The $2\times 2,$ one-dimensional Burgers system
\begin{equation} \label{2b}
\D_t \begin{pmatrix} u_1\\u_2\end{pmatrix} + \begin{pmatrix} u_1 &-b(u)^2 u_2\\u_2&u_1\end{pmatrix}\D_x\begin{pmatrix} u_1\\u_2\end{pmatrix} = F(u_1, u_2),
\end{equation}
where {$F$ and $b$ are} smooth and real-valued, has a complex structure if $b$ is constant. In the case $b \equiv 1,$ $F \equiv (0,1),$ a strong instability result for the Cauchy-Kovalevskaya solution issued from $(u_1^0,0),$ where $u_1^0$ is analytic and real-valued, was proved in \cite{LMX}.

 We assume $b > 0,$ and the existence of a local smooth solution $\phi = (\phi_1,\phi_2).$ 
 The principal symbol is
 $$ A(t,x,\xi) = \xi \left(\begin{array}{cc} \phi_1 & - b(\phi)^2 \phi_2 \\ \phi_2 & \phi_1 \end{array}\right).$$
  Without loss of generality, we let $\xi = 1.$ 
 The eigenvalues and eigenvectors are \begin{equation} \label{eeb}
 \l_\pm  = \phi_1 \pm i \phi_2 b(\phi), \qquad  e_\pm = \frac{1}{(1 + b(\phi)^2)^{1/2}} \left(\begin{array}{c} \pm i b(\phi) \\ 1 \end{array}\right).
 \end{equation}
 The characteristic polynomial is 
 $$ P = (\l - \phi_1)^2 + b(\phi)^2 \phi_2^2.$$

\medskip

{\it Initial ellipticity.} If $\phi_2(0,x_0) \neq 0$ for some $x_0 \in \R,$ then the principal symbol is elliptic at $t = 0,$ and Theorem \ref{th:0} appplies.

\medskip

{\it Smooth defect of hyperbolicity.} Consider the case $\phi_2(0,x) \equiv 0.$ We cannot observe a defect of hyperbolicity as in Theorem \ref{th:1/2}, since the eigenvalues are smooth in time. Via Proposition \ref{prop:evalues}, we see that Theorem \ref{th:1} holds as soon as 
 \be \label{f2} F_2(\phi(0,x_0)) \neq 0, \quad \mbox{for some $x_0 \in \R.$}\ee

\medskip

In the case $b(u) = b(u_2),$ then \eqref{2b} is a system of conservation laws
$$ \d_t u_1 + \d_x f_1(u)  = F_1(u), \quad \d_t u_2 + \d_x f_2(u)  = F_2(u),
$$ with fluxes
 $$f_1(u) = \frac{1}{2} u_1^2 - \int_0^{u_2} y b(y)^2 \, dy, \qquad f_2(u) = u_1 u_2.$$
 If, for instance, $F(u) = (0,u_1^2)$ and $b(u_2) = 1 + u_2^2,$ then the system is ill-posed for all data. 

\subsection{Two-dimensional Burgers systems} \label{sec:burgers:2}  Consider the family of $2 \times 2$ systems in $\R^2:$ 
  \begin{equation} \label{2d-burgers} 
  \d_t u + \left(\begin{array}{cc} u_1 \d_{x_1} & - b(u)^2 u_2 (\d_{x_2} + \d_{x_1}) \\ u_2 (\d_{x_1} + \d_{x_2}) & u_1 \d_{x_1} \end{array}\right) u = F(u).
  \end{equation}
   We assume $b > 0,$ and the existence of a local smooth solution $\phi = (\phi_1,\phi_2).$ The principal symbol is
   $$ A = \left(\begin{array}{cc} \xi_1 \phi_1 & - (\xi_1 + \xi_2) b(\phi)^2 \phi_2 \\ (\xi_1 + \xi_2) \phi_2 & \xi_1 \phi_1\end{array}\right). $$
The eigenvalues and eigenvectors are
$$ \l_\pm = \xi_1 \phi_1 \pm i (\xi_1 + \xi_2) \phi_2 b(\phi), \qquad e_\pm = \frac{1}{(1 + b(\phi)^2)^{1/2}} \left(\begin{array}{c} \pm i b(\phi) \\ 1 \end{array}\right).$$

\medskip

{\it Initial ellipticity.} If $\phi_2(0,x_0) \neq 0$ for some $x_0 \in \R^2,$ then the principal symbol is initially elliptic at any $(\xi_1, \xi_2) \in \SS^1$ such that $\xi_1 + \xi_2 \neq 0.$ 

\medskip

{\it Smooth defect of hyperbolicity.} Consider the case $\phi_2(0,x) \equiv 0.$ By Proposition \ref{prop:evalues}, the assumptions of Theorem \ref{th:1} are satisfied under condition \eqref{f2}.

\subsection{Van der Waals gas dynamics} \label{sec:VdW}

The compressible Euler equations in one space dimension, in Lagrangian coordinates are 
  \begin{equation} \label{VdW}
   \left\{\begin{aligned} \d_t u_1 + \d_x u_2 & = 0, \\ \d_t u_2 + \d_x p(u_1) & = 0.
   \end{aligned}\right.
   \end{equation}
We assume that the smooth pressure law $p$ satisfies the Van der Waals condition
 \begin{equation}
  \label{vdw-pressure} p'(u_1) \leq 0, \qquad \mbox{for some $u_1 \in \R,$}
 \end{equation}
and assume existence of a smooth solution $\phi = (\phi_1, \phi_2).$ The principal symbol at $\xi = 1$ is
 $$ A = \left(\begin{array}{cc} 0 & 1 \\ p'(\phi_1) & 0 \end{array}\right).$$
 The eigenvalues are
 $$ \l_\pm = (p'(\phi_1))^{1/2}.$$

\medskip

{\it Initial ellipticity.} If $p'(\phi_1(0,x_0)) < 0$ for some $x_0 \in \R,$ then Theorem \ref{th:0} applies.

\medskip

{\it Non-semi-simple defect of hyperbolicity.} If $p'(\phi_1(0,x)) \geq 0$ for all $x$ (initial hyperbolicity) and $p'(\phi_1(0,x_0)) = 0$ for some $x_0$ (coalescence of two eigenvalues), if
 $$ p''(\phi_1(0,x_0)) \d_x \phi_2(0,x_0) > 0,$$
then condition \eqref{cond:coal2} holds, and Theorem \ref{th:1/2} applies. 

\subsection{Klein-Gordon-Zakharov systems} \label{sec:KG}

Consider the family of systems in one space dimension
\begin{equation} \label{kgwbis} \left\{ \begin{aligned} \d_t \left(\begin{array}{c} u \\ v \end{array}\right) + \d_x \left(\begin{array}{c} v \\ u \end{array}\right) + \left(\begin{array}{cc} \a & 0 \\ 0 & 0 \end{array}\right) \d_x \left(\begin{array}{c} n \\ m \end{array}\right) & = & (n + 1) \left(\begin{array}{c} v \\ - u \end{array}\right), \\ \d_t \left(\begin{array}{c} n \\ m \end{array}\right) +  c \d_x \left(\begin{array}{c} m \\  n \end{array}\right)  + \left(\begin{array}{cc} \a & 0 \\ 0 & 0 \end{array}\right) \d_x \left(\begin{array}{c} u \\ v \end{array}\right) & = &  \d_x  \left(\begin{array}{c} 0 \\ u^2 + v^2 \end{array}\right),\end{aligned}\right.\end{equation}
 indexed by $\a \in \R,$ $c \in \R \setminus \{-1,1\}.$ We assume existence of a smooth solution $\phi = (u,v,n,m).$ The principal symbol at $\xi = 1$ is %
  \begin{equation} \label{a-kgz} A  = \left(\begin{array}{cccc} 0 & 1 & \a & 0 \\ 1 & 0 & 0 & 0 \\ \a & 0 & 0 & c \\ - 2 u & -2 v & c & 0 \end{array}\right).\end{equation}

\medskip

{\it The case $\a = 0.$} The principal symbol is block diagonal, and there are four distinct eigenvalues $\{ \pm 1, \pm c\}.$ This implies that \eqref{kgwbis} is strictly hyperbolic, hence locally well-posed in $H^s,$ for $s > 3/2$ (see for instance Theorem 7.3.3, \cite{M3}). It was observed in \cite{CEGT} that for $c \notin \{ - 1, 1\}$ and $\a = 0,$ system \eqref{kgwbis} is conjugated to a {\it semi-linear} system, which implies a sharper existence result:

 \begin{prop}[{\rm\cite{CEGT}, Section 2.2}] \label{prop:ex-KGZ} If $c \notin \{-1, 1\}$ and $\a = 0,$ the system \eqref{kgwbis} is locally well-posed in $H^s(\R),$ for $s > 1/2.$
 \end{prop}

 \begin{proof} The change of variables
 $$ \begin{aligned} (\tilde u, \tilde v) & = (u + v, u - v), \\ (\tilde n,\tilde m) & = \left(n + m - \frac{1}{1 - c} \tilde u^2 - \frac{1}{1 + c} \tilde v^2 \, , \, n - m - \frac{1}{1 + c} \tilde u^2 - \frac{1}{1 - c} \tilde v^2\right),\end{aligned}$$
 transforms \eqref{kgwbis} into the system in $\tilde U := (\tilde u, \tilde v, \tilde n, \tilde m):$
 \begin{equation} \label{kgw2}
 \d_t \tilde U + \left(\begin{array}{cccc} 0 & 1 & 0 & 0 \\ 1 & 0 & 0 & 0 \\ 0 & 0 & 0 & c \\ 0 & 0 & c & 0 \end{array}\right) \d_x \tilde U = (n + 1) \left(\begin{array}{c} - \tilde v \\ \tilde u \\ -2 (1 - c)^{-1} \tilde u \tilde v  \\ -2 (1 + c)^{-1} \tilde u \tilde v \end{array}\right).
 \end{equation}
 System \eqref{kgw2}, being symmetric hyperbolic and semi-linear, is locally well-posed in $H^s(\R),$ for $s > 1/2.$
\end{proof}

\medskip

{\it The case $\a \neq 0.$} By Proposition \ref{prop:ex-KGZ}, system \eqref{kgwbis} takes the form of a symmetric perturbation of a well-posed system. The characteristic polynomial of the principal symbol \eqref{a-kgz} at $\xi = 1$ is
\begin{equation} \label{eik} P(t,x,\l) = (\l^2 - c^2)(\l^2 - 1) - \a^2 \l^2 + 2 \a c( v + u \l).
 \end{equation}
Consider an initial datum for $(u(0), v(0), n(0), m(0))$ such that, for some $x_0 \in \R,$ 
\begin{equation} \label{cond-KGZ} u(0,x_0) = 0, \qquad v(0,x_0) = - \frac{c}{2 \a}, \qquad \a c \d_x u(0,x_0) > 0.
  \end{equation}
The first two conditions in \eqref{cond-KGZ} imply that at $\o_0 = (x_0,1,0)$ there holds
$$ P(0,\o_0) = \d_\l P(0,\o_0) = 0.$$ 
The third condition in \eqref{cond-KGZ} implies 
$$(\p_{t}P\p_{\lambda}^{2}P)(0,\o_0) =\bigl(2\alpha c\p_{t}v(0,x_0)\bigr)(-1-c^{2}-\alpha^{2})=
\bigl(2\alpha c\p_{x}u(0,x_0)\bigr)(1+c^{2}+\alpha^{2})>0,
$$
 so that the third condition in \eqref{cond-KGZ} implies condition \eqref{cond:coal2}. Theorem \ref{th:1/2} thus asserts instability of the Cauchy problem for \eqref{kgwbis} in the vicinity of any smooth solution $\phi$ satisfying \eqref{cond-KGZ} at $t = 0.$

 In particular, for any given $\a_0 > 0,$ we can find initial data, depending on $\a_0,$ such that \eqref{kgwbis} with $\a = 0$ is well-posed whereas \eqref{kgwbis} with $\a = \a_0$ is ill-posed. Such initial data are $O(1/\a_0)$ in $L^\infty(\R).$
\begin{appendix}
 
\section{Proof of Proposition \ref{prop:evalues}} \label{app:eigenvalue}

The principal symbol can be block diagonalized, with a $2 \times 2$ block $A_0$ with double 
real eigenvalue $\l_0$ at $(0,x,\xi),$ and an $(N-2) \times (N-2)$ block which does not admit $\l_0$ as an eigenvalue at $(0,x,\xi).$ Throughout this proof $(x,\xi)$ are fixed and omitted in the arguments. The characteristic polynomial of $A$ factorizes into $P = P_0 P_1,$ where $P_1(0,\o_0) \neq 0$ and $P_{0},P_{1}$ have real coefficients.
We may concentrate on $P_0:$ 
$$ P_0(\l) = \l^2 - \l \tr A_0 + \det A_0.$$
The eigenvalues $\l_\pm$ of $A_0$ at $(t,x,\xi)$ are %
\be \label{l+-} \l_\pm(t) = \frac{1}{2} \tr A_0(t) \pm \frac{1}{2} \Delta(t)^{1/2}, \qquad \Delta(t) := (\tr A_0)^2 - 4 \det A_0.\ee
By assumption, these eigenvalues coalesce at $t = 0,$ so that $\Delta(0) = 0.$ The goal is then to prove equivalence \eqref{non-R}. 
 
If the left proposition in \eqref{non-R} holds, then $\Delta(t) = - \a t^2 + O(t^3),$ with $\a > 0.$ Thus $\d_t \Delta(0) = 0;$ on the other hand, 
 $$ \begin{aligned}
 \d_t \Delta(0) & = 2 \tr A_0(0) \d_t \tr A_0(0) - 4\d_t \det A_0(0) =4\lambda_{0} \d_t \tr A_0(0)
- 4\d_t \det A_0(0)
 \\ & =-4(\p_{t}P_{0})(0).
\end{aligned}$$
Besides, $\d_t^2 \Delta(0) < 0;$ on the other hand, 
$$ %
 \d_t^2 \Delta(0)  = 2 (\d_t \tr A_0(0))^2 + 2 \tr A_0(0) \d_t^2 \tr A_0(0) - 4\d_t^2 \det A_0(0),$$
 implying, since $\tr A_0(0) = 2 \l_0,$ 
 $$ \begin{aligned}
 \d_t^2 \Delta(0) &  =2 (\d_t \tr A_0(0))^2
+4\lambda_{0} \d_t^2 \tr A_0(0) - 4\d_t^2 \det A_0(0)
\\ &  =2(\p_{t}\p_{\lambda} P_{0})^{2}-2\p_{\lambda}^{2}P_{0}\p_{t}^{2}P_{0}(0),\end{aligned}$$
which gives  indeed $(\d^2_{t\l} P)^2 < \d_t^2 P \d_\l^2 P$ at $t = 0.$ 

The converse implication is proved in the same way: the right proposition in \eqref{non-R} implies $\d_t \Delta(0) = 0,$ $\d_t^2 \Delta(0) < 0,$ as shown above, and this implies that the eigenvalues in \eqref{l+-} are differentiable and leave the real axis at $t = 0.$

 \section{Symbols and operators} \label{sec:symbols}

 Pseudo-differential operators in $\e^h$-semi-classical quantization are defined by %
\be \label{def:app:semicl}
\op_\e(a) u := (2\pi)^{-d} \int_{\R^{d}} e^{i x \cdot \xi} a(x, \e^{h} \xi) \hat u(\xi) \,d\xi, \qquad 0 < \e, \quad 0 < h.
\ee
Here $h = 1/(1 + \ell),$ as in Assumption \ref{ass:main}. Above $a$ is a classical symbol of order $m$: $a \in S^m,$ for some $m \in \R,$ that is a smooth map in $(x,\xi),$ with values in a finite-dimensional space, such that 
\be \label{notation:norm:eps}
  \| a \|_{m,r} := \sup_{\begin{smallmatrix} |\a| \leq r, |\b| \leq r \\ (x,\xi) \in \R^{2d} \end{smallmatrix}} \langle \xi \rangle^{|\b| - m} |\d_x^\a \d_\xi^\b a(x,\xi)| < \infty, 
  \ee
The family $\| \cdot \|_{\e,s}$ of $\e$-dependent norms is defined by 
 \begin{equation} \label{hse} \| u \|_{\e,s} :=  \big\| \langle \e^h \xi \rangle^{s/2} \hat u(\xi) \big\|_{L^2(\R^d_\xi)}, \quad s \in \R, \quad \langle \cdot \rangle := (1 + |\cdot|^2)^{1/2}.
 \ee
 Introducing dilations $(d_\e)$ such that $(d_\e u)(x) = \e^{h d/2} u(\e^{h}
 x),$ we observe that there holds
 \be \label{dilation}
 \|d_\e u\|_{H^s} = \| u \|_{\e,s}, \quad \op_\e(a) = d_\e^{-1} \op(\tilde a) d_\e, \quad \tilde a(x,\xi) := a(\e^{h} x, \xi).
 \ee 
\begin{prop} \label{prop:action} Given $m \in \R,$ $a \in S^m,$ there holds the bound
 \be \label{bd:action} \| \op_\e(a) u \|_{L^2} \lesssim \| a \|_{m,C(d)} \| u \|_{\e,-m},\ee
 for all $u \in H^{-m},$ for some $C(d) > 0$ depending only on $d.$ If $m = 0,$ there holds the bound 
 \be \label{bd:action:H} \| \op_\e(a) u \|_{L^2} \lesssim \sum_{0 \leq |\a| \leq d+1}  \sup_{\xi \in \R^d} | \d_x^{\a} a(\cdot,\xi) |_{L^1(\R^d_x)} \| u \|_{L^2}.\ee
\end{prop}

\begin{proof} By use of dilations \eqref{dilation}, we observe that $\op_\e(a) u = \op_1(\langle \xi \rangle^{-m} \tilde a) \langle D \rangle^m d_\e u.$ Bound \eqref{bd:action} with any $C(d) > [d/2] + 1$ then follows for instance from Theorem 1.1.4 and its proof from \cite{Le}. Bound \eqref{bd:action:H} is proved in Theorem 18.8.1 from volume 3 of \cite{H1}. 
\end{proof}

\begin{prop} \label{prop:composition} Given $a_1 \in S^{m_1},$ $a_2 \in S^{m_2},$ $n \in \N,$ 
 \be \label{compo:e}
 \op_\e(a_1) \op_\e(a_2) = \sum_{0 \leq q \leq n} \e^{ hq }\op_\e(a_1 \sharp_q a_2) + \e^{h (n+1)} \op_\e(R_{n+1}(a_1,a_2)),
 \ee
 where
\be \label{def:sharp} a_1 \sharp_q a_2 = \sum_{|\a| = q} \frac{(-i)^{|\a|}}{\a!} \d_\xi^\a a_1 \d_x^\a a_2,\ee
 and $R_{n+1}(a_1,a_2) \in S^{m_1 + m_2 - (n+1)}$ satisfies 
 \be \label{composition:e}
 \| \op_\e(R_{n+1}(a_1,a_2)) u\|_{L^2} \lesssim \| \d_\xi^n a_1 \|_{m_1,C(d)} \| \d_x^n a_2 \|_{m_2,C(d)} \| u \|_{\e, m_1 + m_2 - n- 1},
 \ee
 with $C(d) > 0$ depending only on $d,$ for all $u \in H^{m_1 + m_2 - n-  1}.$ 
\end{prop}

\begin{proof} Based for instance on Theorem 1.1.20, Lemma 4.1.2 and Remark 4.1.4 of \cite{Le}, and the use of dilations \eqref{dilation}.
\end{proof}

Specializing to symbols with a slow $x$-dependence, we obtain:
\begin{prop} \label{prop:composition:slow:x} 
Given $a_1 \in S^{m_1},$ $a_2 \in S^{m_2},$ if $a_2$ depends on $x$ through $\e^{1-h} x,$ there holds 
 \be \label{compo:e:slow}
 \big\| \big( \op_\e(a_1) \op_\e(a_2) - \op_\e(a_1 a_2) \big) u \|_{\e,s} \lesssim \e \| a_1 \|_{m_1,C(d)} \| a_2 \|_{m_2,C(d)} \| u \|_{\e,s + m_1 + m_2 - 1},
 \ee
\end{prop}

\section{On extending locally defined symbols} \label{sec:extension:symbols}

Our assumptions are local in $(x,\xi)$ around $(x_0,\xi_0).$ Accordingly the symbols $Q$ and $\mu$ that appear in Assumption \ref{ass:main} are defined (after a change of spatial frame) only locally around $(0,\xi_0).$ We explain here how to extend the locally defined family of invertible matrices $Q(x,\xi)$ into an element of $S^0$ with an inverse (in the sense of matrices) which belongs to $S^0.$ 

The spectrum of $Q(0,\xi_0)$ is a finite subset of $\C.$ In particular, we can find $\a \in \R $ such that the spectrum of $Q(0,\xi_0)$ is included in $\C \setminus e^{i \a} \R_-.$ By continuity of the spectrum, for all $(x,\xi)$ close enough to $(0,\xi_0),$ the spectrum of $Q(x,\xi)$ does not intersect the half-line $e^{i \a} \R_-.$ Let $\delta > 0$ such that this property holds true over $B_\delta = B(0,\delta) \times B(\xi_0,\delta).$ We may then define the logarithm of matrix $e^{-i\a} Q$ in $B_\delta$ by 
$$ \mbox{Log}\,(e^{-i\a} Q) = \int_0^1 (e^{-i\a} Q - \Id) \big( (1 - t) \Id + t e^{-i \a} Q \big)^{-1} \, dt,$$
and the notation $\mbox{Log}$ is justified by the identity
\be \label{explog} \exp \mbox{Log}\,(e^{-i\a} Q) = e^{-i\a} Q, \quad \mbox{in $B_\delta.$}\ee
Let $\s(x,\xi)$ be a smooth cut-off in $C^\infty_c(\R^d \times \R^d),$ such that $0 \leq \s(x,\xi) \leq 1,$ with $\s \equiv 1$ on a neighborhood of $(x_0,\xi_0),$ and such that the support of $\s$ is included in $B_{\delta/2}.$ Let
$$ R(x,\xi) = \s(x,\xi) \mbox{Log}(e^{-i\a} Q(x,\xi)) + (1 - \s(x,\xi)) \Id, \quad \mbox{in $B_\delta.$}$$
We may extend smoothly $R$ by $R \equiv \Id$ on the complement of $B_\delta$ in $\R^{2d}.$ Then for all $(x,\xi) \in \R^{2d},$ the matrix 
$$ \tilde Q(x,\xi) = \exp R(x,\xi)$$
is smooth and invertible. There holds
$$ \inf_{\R^{2d}} \det \tilde Q > 0.$$
Indeed, the infimum over the closed ball $\bar B_\delta$ is positive, by compactness and continuity, and the determinant is constant equal to $e^N$ outside $\bar B_\delta.$ Thus the norms $|\tilde Q(x,\xi)|$ and $|\tilde Q(x,\xi)^{-1}|$ are globally bounded over $\R^{2d}.$ Since $\tilde Q$ is constant outside a compact, this implies $\tilde Q \in S^0,$ and $\tilde Q^{-1} \in S^0.$ Finally, by \eqref{explog} and definition of the cut-off $\s,$ there holds
$$ \tilde Q(x,\xi) = e^{-i\a} Q(x,\xi), \quad \mbox{for $(x,\xi)$ close to $(0,\xi_0).$}$$
Thus $e^{i\a} \tilde Q$ is an appropriate extension of $Q.$

 \section{An integral representation formula} \label{sec:Duhamel}
 
 We adapt to the present context an integral representation formula introduced in \cite{T3}. Consider the initial value problem, posed in time interval $[0, T(\e)],$ with the limiting time $T(\e) := \big( T_\star |\log \e|\big)^{1/(1 + \ell)},$ for some $T_\star > 0,$ some $ \ell \geq 0:$ 
 \begin{equation} \label{buff01} 
  \d_t u + \op_\e({\mathcal A}) u  = g, \qquad u(0) = u_0,
  \end{equation}
 where $\cA=  \cA(\e ,t)$ belongs to $S^0$ for all $\e > 0$ and all $t \leq T(\e).$ Recall that $\op_\e(\cdot)$ denotes $\e^h$-semiclassical quantization of operators, as defined in \eqref{def:app:semicl}. The parameter $h$ belongs to $(0,1].$ The datum $u_0$ belongs to $L^2,$ and the source $g$ is given in $C^0([0,T(\e)],L^2(\R^d)).$ Denote $S_0$ the flow of $-\cA,$ defined for $0 \leq \t \leq t \leq T(\e)$ by  
 \begin{equation} \label{resolvent0}\d_t S_0(\t;t) + \cA S_0(\t;t) = 0, \qquad S_0(\t;\t) = \mbox{Id}.\end{equation}
 For some $q_0 \in \N^*,$ denote $\{ S_q\}_{1 \leq q \leq q_0}$ the solution to the triangular system of linear ordinary differential equations
 \begin{equation} \label{resolventk}
  \d_t S_q + \cA S_q + \sum_{\begin{smallmatrix} q_1 + q_2 = q \\ 0 < q_1 \end{smallmatrix} } \cA  \sharp_{q_1} S_{q_2} = 0, \qquad S_q(\t;\t) = 0,
  \end{equation}
 with notation $\sharp_q$ introduced in \eqref{def:sharp}. 

 \begin{ass} \label{ass:BS} The symbol $\cA$ is compactly supported in $x,$ uniformly in $\e,t,\xi,$ and there holds bounds
  $$ \langle \xi \rangle^{|\b|} |\d_x^\a \d_\xi^\b \cA(\e,t,x,\xi)| \lesssim \e^{- |\b| \zeta},$$
  and
  $$ |\d_x^\a \d_\xi^\b S_q(\e,\t;t,x,\xi)| \lesssim \e^{- \zeta(1 + |\b| + q)} \exp\big(\g t^{\ell + 1}\big).$$
  for some $0 \leq \zeta < h,$ some $\g > 0,$ for all $x,\xi,$ all $t \leq T(\e).$  
 \end{ass}

Denote $\dsp{\Sigma := \sum_{0 \leq q \leq q_0} \e^{qh} S_q.}$ Then $\op_\e(\Sigma)$ is an approximate solution operator for \eqref{buff01}:
 \begin{lem} \label{lem:duh-remainder} Under Assumption {\rm \ref{ass:BS}}, if $q_0$ is large enough, depending on $\zeta,$ $h,$ $\g$ and $T_\star,$ there holds the bound  
  \begin{equation} \label{tildeS1-duh} \d_t \op_\e(\Sigma) + \op_\e(\cA) \op_\e(\Sigma) = \rho,\end{equation}
  where $\rho$ satisfies for $0 \leq \t \leq t \leq T(\e),$
for all $u \in L^2(\R^d),$  
  \begin{equation} \label{tilde-remainder1} \|  \rho(\t;t) u \|_{L^2} \lesssim \e  \| u\|_{L^2}.\end{equation}
  \end{lem}
  
  \begin{proof}  
By Proposition \ref{prop:composition},
$$ \op_\e(\cA) \op_\e(S_q) = \op_\e(\cA S_q) + \sum_{1 \leq q' \leq q_0} \e^{q' h} \op_\e(\cA \sharp_{q'} S_q) +  \e^{(q_0 +1) h} \op_\e(R_{q_0 + 1}(\cA, S_q)),
 $$ 
 and, summing over $0 \leq q \leq q_0:$ 
$$ \op_\e(\cA) \op_\e(\Sigma) = \op_\e(\cA \Sigma) + \sum_{\begin{smallmatrix} 0 \leq q_2 \leq q_0 \\ 1 \leq q_1 \leq q_0\end{smallmatrix}} \e^{(q_1 + q_2) h} \op_\e(\cA \sharp_{q_1} S_{q_2}) + \e^{(q_0 + 1) h} R,$$
\vs

where $\dsp{ R: = \sum_{0 \leq q \leq q_0} \e^{q h} \op_\e(R_{q_0+ 1}(\cA, S_q)).}$ Besides, by definition of the correctors \eqref{resolventk}, 
$$ - \d_t \op_\e(\Sigma) = \op_\e(\cA \Sigma) + \sum_{\begin{smallmatrix} 1 \leq q_1 + q_2 \leq q_0 \\ 0 < q_1 \end{smallmatrix}} \e^{(q_1 + q_2)h} \op_\e(\cA \sharp_{q_1} S_{q_2}).$$
Comparing with the above, we find that \eqref{tildeS1-duh} holds, with 
$$ - \rho := \sum_{\begin{smallmatrix} q_0 +1 \leq q_1 + q_2 \leq 2 q_0 \\ 1 \leq q_1 \leq q_0 \\ 0 \leq q_2 \leq q_0 \end{smallmatrix}} \e^{(q_1 + q_2) h} \op_\e(\cA \sharp_{q_1} S_{q_2}) + \e^{(q_0 + 1)h} R.$$
Since $\cA$ is compactly supported in $x,$ so are the correctors $S_q,$ for $q \geq 1,$ and the derivatives of $S_0.$ Thus under Assumption \ref{ass:BS}, there holds
  $$ \big|\d_x^\a \big( \cA \sharp_{q_1} S_{q_2} \big)\big|_{L^1(\R^d_x)}  \lesssim \e^{-\zeta(1 + q_1 + q_2)} \exp\big(\g t^{1 + \ell}\big),$$
  uniformly in $\xi.$ Hence, by Proposition \ref{prop:action}, the bound 
  $$ \big\| \op_\e\big( \cA \sharp_{q_1} S_{q_2} \big)\big\|_{L^2 \to L^2(\R^d)} \lesssim \e^{-\zeta(1 + q_1 + q_2)} \exp\big(\g t^{1 + \ell}\big).$$
Thus a control of the $L^2 \to L^2$ norm of the first term in $\rho$  by
 $$  \sum_{\begin{smallmatrix} q_0 +1 \leq q_1 + q_2 \leq 2 q_0 \\ 1 \leq q_1 \leq q_0 \\ 0 \leq q_2 \leq q_0 \end{smallmatrix}} \e^{(q_1 + q_2) h}  \e^{-\zeta(1 + q_1 + q_2)} \exp\big(\g t^{1 + \ell}\big) \lesssim \e^{(h - \zeta) (q_0 + 1) - \zeta - \g T_\star},$$
 over the interval $[0, T(\e)],$ implying the desired bound as soon as
 $$ (h- \zeta)(q_0 + 1) \geq 1 + \zeta + \g T_\star.$$
 Besides, by Proposition \ref{prop:composition}, there holds 
$$ \big\| \op_\e\big(R_{q_0 + 1}(\cA, S_q)\big) \big\|_{L^2 \to L^2} \lesssim \| \d_\xi^{q_0 + 1} \cA \|_{0,C(d)} \| \d_x^{q_0 + 1} S_q\|_{0,C(d)},$$
and by Assumption \ref{ass:BS}:
$$ \| \d_\xi^{q_0 + 1} \cA \|_{C(d)} \lesssim \e^{- (q_0 + 1 + C(d)) \zeta}, \quad \| \d_x^{q_0 + 1} S_q\|_{C(d)} \lesssim \e^{- \zeta(1 + q + C(d))} \exp\big( \g t^{\ell + 1} \big).$$
This gives a control of the $L^2 \to L^2$ norm of the second term in $\rho$ by 
$$ \e^{(q_0 + 1) h} \sum_{0 \leq q \leq q_0} \e^{q h} \e^{- (q_0 + 1 + C(d)) \zeta - (1 + q + C(d)) \zeta}  \exp\big( \g t^{\ell + 1} \big).$$ 
We conclude that \eqref{tilde-remainder1} holds if $q_0$ satisfies
$$ (h- \zeta)(q_0 + 1) \geq 1 + \g T_\star + 2 C(d) \zeta,$$
which can be achieved since in Assumption \ref{ass:BS} we postulated $\zeta < h.$ 
 \end{proof}

  \begin{theo} \label{duh1} Under Assumption {\rm \ref{ass:BS}}, the initial value problem \eqref{buff01} has a unique solution
  $u \in C^0([0,T(\e)], L^2(\R^d)),$ given by
  \begin{equation} \label{buff0.01} u = \op_\e(\Sigma(0;t)) u_0 + \int_0^t \op_\e(\Sigma(t';t))(\Id + \e R_1(t')) (g(t') + \e R_2(t') u_0) \, dt',
  \end{equation}
  where $R_1$ and $R_2$ are bounded: for all $v \in L^2,$ 
  \begin{equation} \label{bd-R121} \begin{aligned} \| R_1(t) v \|_{L^2} + \| R_2(t) v \|_{L^2}  \lesssim \| v \|_{L^2},
  \end{aligned}\end{equation}
  uniformly in $\e$ and $t \in [0,T(\e)].$ 
 \end{theo}

\begin{proof} Let $h \in L^\infty([0,T(\e)],L^2(\R^d)).$ By Lemma \ref{lem:duh-remainder}, the map $u$ defined by
  \begin{equation} \label{def-u-duh1} u := \op_\e(\Sigma(0;t)) u_0 + \int_0^t \op_\e( \Sigma(t';t)) h(t') \, dt'
  \end{equation}
   solves \eqref{buff01} if and only if, for all $t,$ there holds
\begin{equation} \label{cond-g1}
  \big((\Id + \rho_0) h\big)(t)  = g - \rho(0;t) u_0,
   \end{equation}
 where $\rho_0$ is the linear integral operator 
  $$\rho_0: \qquad v \in C^0([0, T(\e)],L^2) \to \Big( t \to \int_0^t  \rho(\t;t) v(\t) \, d\t \Big) \in C^0([0, T(\e)], L^2).$$ 
By \eqref{tilde-remainder1}, there holds 
$$ \sup_{0 \leq t \leq T(\e)} \| (\rho_0 v)(t) \|_{L^2} \lesssim \e \sup_{ 0\leq t \leq T(\e)} \| v(t) \|_{L^2}.$$
Thus $\Id + \rho_0$ is invertible in the Banach algebra of linear bounded operators acting on $C^0([0, T(\e)], L^2(\R^d)).$ This provides a solution $h$ to \eqref{cond-g1}, and we obtain representation \eqref{buff0.01} 
with $R_1 := \e^{-1}((\Id + \rho_0)^{-1} - \Id)$ and $R_2 = - \rho(0;\cdot).$ Bound \eqref{bd-R121} is a consequence of \eqref{tilde-remainder1}. Uniqueness follows from Cauchy-Lipschitz, since $\cA \in S^0.$ 
\end{proof}

\end{appendix}

\end{document}